\documentclass{amsart}
\pdfoutput=1

\usepackage{amssymb,amsthm,amsmath}
\usepackage{microtype}
\usepackage{tikz}
\usetikzlibrary{arrows}
\usepackage[pdftex,colorlinks,citecolor=black,linkcolor=black,urlcolor=black,bookmarks=false]{hyperref}
\def\MR#1{\href{http://www.ams.org/mathscinet-getitem?mr=#1}{MR#1}}
\def\arXiv#1{arXiv:\href{http://arXiv.org/abs/#1}{#1}}
\usepackage{doi}

\newtheorem{theorem}{Theorem}[section]
\newtheorem{prop}[theorem]{Proposition}
\newtheorem{lemma}[theorem]{Lemma}

\theoremstyle{definition}
\newtheorem{definition}[theorem]{Definition}
\newtheorem{example}[theorem]{Example}
\theoremstyle{remark}
\newtheorem{remark}[theorem]{Remark}
\newtheorem*{remark*}{Remark}

\numberwithin{equation}{section}

\newcommand{\wt}{\widetilde}

\newcommand{\cC}{\mathcal{C}}
\newcommand{\abs}[1]{\left\lvert#1\right\rvert}
\newcommand{\snorm}[1]{\lVert#1\rVert}
\newcommand{\norm}[1]{\left\lVert#1\right\rVert}

\newcommand\sign{\mathop{\textup{sign}}}

\newcommand\MaxCut{\mathop{\textup{MaxCut}}}
\newcommand\eps{\varepsilon}
\newcommand\Haus{\textup{Hf}}
\newcommand\iv{\widehat}
\newcommand\E{\mathop{\mathsf{E}}}
\newcommand\Ent{\mathop{\mathsf{Ent}}\nolimits}
\newcommand\PD{\triangle}
\newcommand\FP{\mathop{\mathsf{FP}}}
\newcommand\dcut{d_{\square}}
\newcommand\delcutn{\delta_{\square,\textup{norm}}}

\newcommand{\one}{\mathbf{1}}

\newcommand\R{{\mathbb R}}

\newcommand\N{{\mathbb N}}

\newcommand\ba{\mathbf{a}}
\newcommand\bb{\mathbf{b}}

\newcommand\bq{\mathbf{q}}

\newcommand\CC{\mathcal{C}}
\newcommand\EE{\mathcal{E}}
\newcommand\FF{\mathcal{F}}
\newcommand\PP{\mathcal{P}}
\newcommand\cP{\mathcal{P}}
\renewcommand\SS{\mathcal{S}}

\title[An $L^p$ theory of sparse graph convergence II]{An $L^p$ theory of sparse graph convergence II:\\
LD convergence, quotients, and\\ 
right convergence}

\author{Christian Borgs}
\address{Microsoft Research\\
One Memorial Drive\\
Cambridge, MA 02142} \email{borgs@microsoft.com}

\author{Jennifer T.\ Chayes}
\address{Microsoft Research\\
One Memorial Drive\\
Cambridge, MA 02142} \email{jchayes@microsoft.com}

\author{Henry Cohn}
\address{Microsoft Research\\
One Memorial Drive\\
Cambridge, MA 02142} \email{cohn@microsoft.com}

\author{Yufei Zhao}
\address{Department of Mathematics\\
Massachusetts Institute of Technology\\
Cambridge, MA 02139} \email{yufeiz@mit.edu}
\thanks{Zhao was supported by a
Microsoft Research PhD Fellowship and internships at Microsoft Research New England.}

\begin{document}

\begin{abstract}
We extend the $L^p$ theory of sparse graph limits, which was introduced in
a companion paper, by analyzing different notions of convergence. Under
suitable restrictions on node weights, we prove the equivalence of metric
convergence, quotient convergence, microcanonical ground state energy
convergence, microcanonical free energy convergence, and large deviation
convergence.  Our theorems extend the broad applicability of dense graph
convergence to all sparse graphs with unbounded average degree, while the
proofs require new techniques based on uniform upper regularity.
Examples to which our theory applies include stochastic block models, power
law graphs, and sparse versions of $W$-random graphs.
\end{abstract}

\setcounter{tocdepth}{1}
\maketitle

\tableofcontents

\section{Introduction}

In the companion paper \cite{part1}, we developed a theory of graph
convergence for sequences of sparse graphs whose average degrees tend to
infinity. These results fill a major gap in the theory of convergent graph
sequences, which dealt primarily with either bounded degree graphs or dense
graphs. While progress in this direction was made by Bollob\'as and Riordan
in \cite{BR}, their approach required a ``bounded density'' condition that
excludes many graphs of interest. For example, it cannot handle graphs with
heavy-tailed degree distributions such as power laws. To accommodate these
and other graphs excluded by the bounded density condition, we generalized
the Bollob\'as-Riordan approach in \cite{part1} to graphs obeying a condition
we called $L^p$ upper regularity. We then showed that when $p>1$, every
sequence of $L^p$ upper regular graphs contains a subsequence converging to a
symmetric, measurable function $W \colon [0,1]^2\to \R$ that is in
$L^p([0,1]^2)$. Such a function is an \emph{$L^p$ graphon}. Conversely, only
$L^p$ upper regular sequences can converge to $L^p$ graphons, and so our
results characterize these limits. The work of Bollob\'as and Riordan in
\cite{BR} and the prior work on dense graph sequences amount to the special
case $p=\infty$, while $L^p$ graphons with $p < \infty$ describe limiting
behaviors that occur only in the sparse setting. Thus, the $L^p$ theory of
graphons completes the previous $L^\infty$ theory to provide a rich setting
for limits of sparse graph sequences with unbounded average degree.

One attractive feature of dense graph limits is that many definitions of
convergence coincide, and it is natural to ask whether the same is true for
sparse graphs. After all, there are many ways to formulate the idea that two
graphs are similar. For example, one could base convergence on subgraph
counts or quotients. Furthermore, statistical physics provides many numerical
measures for similarity, such as ground state energies or free energies.

Let us first address the question of subgraph counts. For dense graphs, the
sequence $(G_n)_{n \ge 0}$ converges under the cut metric if and only if the
$F$-density in $G_n$ converges for all graphs $F$, where the $F$-density is
the probability that a random map from $F$ to $G_n$ is a homomorphism
\cite{dense1}. One might guess that suitably normalized $F$-densities would
characterize sparse graph convergence as well, but this fails dramatically:
for sparse graphs, cut metric convergence does not determine subgraph
densities (see Section~2.9 of \cite{part1}). This is not merely a
technicality, but rather a fundamental fact about sparse graphs. We must
therefore give up on convergence of subgraph counts as a criterion for sparse
graph convergence.

By contrast, we show in this paper that several other widely studied forms of
convergence are indeed equivalent to cut metric convergence in the sparse
setting. Thus, with the exception of subgraph counts, the scope and
consequences of sparse graph convergence are comparable with those of dense
graph convergence.

We will consider several notions of convergence motivated by statistical
physics and the theory of graphical models from machine learning, such as
convergence of ground state energies and free energies, as well as
convergence of quotients,\footnote{Quotient convergence is also called
partition convergence in some of the literature.} which encode ``global''
graph properties of interest to computer scientists, such as max-cut and
min-bisection.  We will also analyze the notion of large deviation (LD)
convergence, which was recently introduced for graph sequences with bounded
degrees \cite{BCG} and can easily be adapted to our more general context. For
bounded degree graphs, LD convergence was strictly stronger than convergence
of quotients or other notions introduced before, but we will see that in our
setting it is equivalent to these other forms of convergence.

All these question can be studied for $L^p$ upper regular sequences of sparse
graphs, but they can also be studied directly for $L^p$ graphons. While the
former might be more interesting from the point of view of applications, the
latter turns out to be more elegant from an abstract point of view. We
therefore first develop the theory for sequences of graphons, and then prove
our results for sparse graph sequences.

We begin in Section~\ref{sec:results} with motivation, definitions, and
precise statements of our results, with some ancillary results stated in
Section~\ref{sec:remarks}.  We begin the proofs in Section~\ref{sec:Gproofs}
by completing the cases that do not require the notion of upper regularity.
We then make use of upper regularity to deal with graphons in
Section~\ref{sec:W-sequences} and graphs in
Section~\ref{sec:graphs->graphons}. Finally, in
Section~\ref{sec:inferring-upper-regularity}, we show that any sequence whose
quotients, microcanonical free energies, or ground state energies converge to
those of a graphon must be upper regular, which completes the proofs.

Before turning to these details, though, we will explain the motivations
behind the different types of convergence analyzed in this paper.

\subsection{Motivation}

When formulating a notion of convergence for growing sequences of graphs, one
is immediately faced with the problem of deciding when to consider two large
graphs on different numbers of vertices to be similar.

One natural approach is to compare summary statistics, such as weighted
counts of homomorphisms to or from small graphs.  Convergence based on these
statistics is called \emph{left convergence} if it uses homomorphisms
\emph{from} small graphs and \emph{right convergence} if it uses
homomorphisms \emph{to} small graphs.  Left convergence amounts to using
subgraph counts, and as discussed in the previous section it is not a useful
tool for characterizing sparse graph convergence.  By contrast, right
convergence is far more useful in the sparse setting.  It amounts to using
statistical physics models, and it encompasses quantities such as max-cut,
min-bisection, etc.\ that are important in combinatorial optimization.

The advantage of using summary statistics is that they can easily be
normalized to compare graphs on different numbers of nodes.  For a more
direct approach, one must find other ways to compare such graphs.

One way to deal with this is to blow up both graphs to obtain two new graphs
on a common, much larger set of vertices. Conceptually, the most elegant way
to do this is probably an infinite blow-up, replacing the vertex sets of both
graphs by the interval $[0,1]$ and the adjacency matrices by appropriate step
functions on $[0,1]^2$. Comparing the two graphs then reduces to comparing
two functions on $[0,1]^2$, leading to the notion of convergence in the cut
norm. A priori, this has the problem that relabeling the nodes of a graph
would change its representation as a function on $[0,1]^2$, but this can be
cured by defining the distance as the cut distance of ``aligned'' step
functions, where alignments are formalized as measure preserving
transformations from $[0,1] \to [0,1]$, chosen in such a way that the
resulting two functions are as close to each other as possible.  The
resulting definition is known as \emph{cut metric convergence}, and it was
analyzed for sparse graphs in \cite{part1}.

Another way to deal with the different vertex sets is to ``squint your eyes''
and look at whether the results are similar. More formally, one divides the
vertex sets of both graphs into $q$ blocks, and then averages the adjacency
matrices over the respective blocks, leading to two $q\times q$ matrices
representing the edge densities between various blocks (we call these
matrices $q$-quotients). One might want to call two graphs similar if their
$q$-quotients are close, but we are again faced with an alignment problem,
now of a slightly different kind: different ways of dividing the vertex set
of a graph into blocks produce different quotients. While some of the
quotients of a graph contain useful information about the graph (for example
those corresponding to Szemer\'edi partitions), others might not.
Unfortunately, it is not a priori clear which of the $q$-quotients of a graph
represent its properties well and which do not. We solve this problem by
defining two graphs to be similar if the \emph{sets} of their $q$-quotients
are close, measured in the Hausdorff distance between subsets of the metric
space of weighted graphs on $q$ nodes.

The four notions of convergence describe informally above, namely left
convergence, right convergence, convergence in metric, and convergence of
quotients, were already introduced in \cite{dense1,dense2} in the context of
sequences of dense graphs. But we felt it to be useful to review the
motivation behind these notions, before addressing the extra complications
stemming from the fact that we want to analyze sparse graphs.

In this paper we also discuss a fifth notion of convergence: large deviation
convergence (LD convergence), which was recently introduced \cite{BCG} to
discuss convergence of bounded degree graphs.  Roughly speaking, LD
convergence keeps track of not just the possible quotients of a graph but
also how often they occur.

Figure~\ref{fig:convergence} illustrates the implications among these
concepts.  In the upper half of the figure, we see that LD convergence is the
strongest notion and ground state energy convergence is the weakest.  To
complete the cycle and prove that they are all equivalent to metric
convergence, we require one hypothesis, namely uniform upper regularity. This
notion first arose in \cite{part1}, and we review its definition below;
intuitively, it ensures that subsequential limits are graphons rather than
more subtle objects.  Indeed, it is possible to state our results using just
the fact that limits can be expressed in terms of graphons, without
explicitly referring to upper regularity.  We will chose this approach when
stating our results in Theorem~\ref{thm:limit-expressions}.

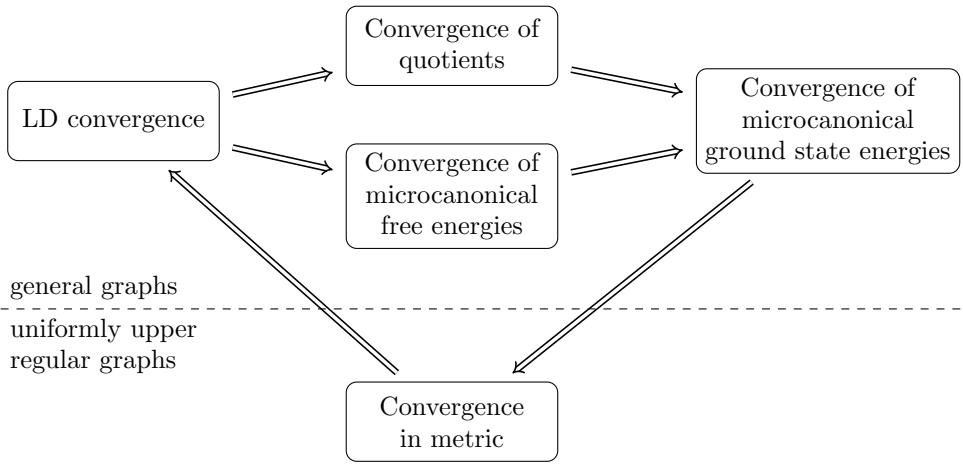
\begin{figure} \centering
\begin{tikzpicture}
  [n/.style={draw,rounded corners,align=center,minimum
  height=3em,minimum width=8em},
imp/.style={-implies,semithick,double equal sign distance,shorten
  >=.5em, shorten <=.5em}]
  \node[n] (LD) at (-4.5,1)
  {LD convergence};
  \node[n] (F) at (0,0)
  {Convergence of \\ microcanonical \\ free energies};
  \node[n] (Q) at (0,2)
  {Convergence of \\ quotients};
  \node[n] (E) at (5,1)
  {Convergence of \\ microcanonical \\ ground state energies};
  \node[n] (D) at (0,-3)
  {Convergence \\ in metric};
  \draw[imp] (LD)--(Q);
  \draw[imp] (LD)--(F);
  \draw[imp] (Q)--(E);
  \draw[imp] (F)--(E);
  \draw[imp] (E)--(D);
  \draw[imp] (D)--(LD);
    \node[anchor=south west,align=left] at (-6,-1.5) {general graphs};
  \draw[dashed] (-6,-1.5) -- (6.8,-1.5);
  \node[anchor=north west,align=left] at (-6,-1.5) {uniformly upper \\
    regular graphs};
\end{tikzpicture}
\caption{Implications between different notions of sparse graph convergence.}
\label{fig:convergence}
\end{figure}

\section{Definitions and main results}
\label{sec:results}

\subsection{Notation}
We begin with some notation. As usual, a weighted graph $G=(V,\alpha,\beta)$
consists of a set $V=V(G)$ of vertices, vertex weights $\alpha_x\geq 0$ for
$x\in V$, and edge weights $\beta_{xy}=\beta_{yx}\in\R$ for $x,y\in V$. We
use $E(G)$ to denote the set of edges of $G$, i.e., the set of pairs
$\{x,y\}$ such that $\beta_{xy}\neq 0$. If we consider several graphs at the
same time, then we make the dependence on $G$ explicit, denoting the edge
weights by $\beta_{xy}(G)$ and the vertex weights by $\alpha_x(G)$. The
maximal node weight of $G$ will be denoted by
\[
\alpha_{\max}(G)=\max_{x\in
V(G)}\alpha_x(G),
\]
and the total node weight of $G$ will be denoted by
\[
\alpha_G=\sum_{x\in V(G)}\alpha_x(G).
\]
We will always assume that $\alpha_G$ is strictly positive. If $U$ is a
subset of $V(G)$, we will use $\alpha_U(G)$ to denote the total weight of
$U$, i.e., $\alpha_U(G)=\sum_{x\in U}\alpha_x(G)$.  We say a sequence
$(G_n)_{n \ge 0}$ of graphs has \emph{no dominant nodes} if
$\alpha_{\max}(G_n)/\alpha_{G_n}\to 0$ as $n \to \infty$. Finally, if $c\in
\R$, we often use $cG$ to denote the weighted graph with vertex weights
identical to those of $G$ and edge weights $\beta_{xy}(cG)=c\beta_{xy}(G)$.
If $G$ is a simple graph with edge set $E(G)$, we often identify it with the
weighted graph with vertex weights $1$ and edge weights
$\beta_{xy}=\one_{xy\in E(G)}$. In this case, $\alpha_G$ is just the number
of vertices in $G$, and $(\beta_{xy}(G))_{x,y \in V(G)}$ is the adjacency
matrix. As usual, we use $[n]$ to denote the set $[n]=\{1,\dots,n\}$ and $\N$
to denote the set of positive integers. Finally, we define the density of a
weighted graph $G$ to be
\[
\|G\|_1=\sum_{x,y\in V(G)}\frac{\alpha_x(G)\alpha_y(G)}{\alpha_G^2}|\beta_{xy}(G)|.
\]
Note that for an unweighted graph without self-loops, $\|G\|_1$ is just the
edge density $2|E(G)|/|V(G)|^2$.

\subsection{Convergence in metric}
\label{sec:metric-conv}

One of the main topics studied in \cite{part1} is the conditions under which
a sequence of sparse graphs contains a subsequence that converges in metric.
This question led us to the notion of $L^p$ upper regularity, and more
generally uniform upper regularity.  Upper regularity plays an important role
in the proofs in the present paper, but it is not essential for stating our
main results.  We therefore defer the discussion of uniform upper regularity
to Section~\ref{sec:upper-reg} and restrict ourselves here to just defining
convergence in metric.  For examples, see Section~\ref{subsec:exandcounter}.

As already discussed, it is convenient to define this distance by embedding
the space of graphs into the set of functions from $[0,1]^2$ into the reals.

\begin{definition}\label{def:GRAPHON} An $L^p$ \emph{graphon} is a
measurable, symmetric function $W\colon [0,1]^2\to \R$ such that
\[
\|W\|_p:=\left(\int |W(x,y)|^p \, dx \, dy\right)^{1/p}<\infty.
\]
Here symmetry means  $W(x,y)=W(y,x)$ for all $(x,y)\in [0,1]^2$. If we do not
specify $p$, we assume that $W$ is in $L^1$ and call it simply a graphon,
rather than an $L^1$ graphon.
\end{definition}

On the set of graphons, one defines the cut norm $\|\cdot\|_\square$ by
\begin{equation}
\label{cutnorm-def}
\|W\|_\square=\sup_{S,T\subseteq [0,1]}
\left|
\int_{S\times T}W(x,y) \,dx\,dy
\right|,
\end{equation}
where the supremum is over measurable sets $S,T\subseteq [0,1]$; this notion
goes back to the classic paper of Frieze and Kannan \cite{FK} on the ``weak
regularity'' lemma. One then defines the \emph{cut distance} between two
graphons $U$ and $W$ by
\[
\delta_\square(U,W)=\inf_{\phi} \|U-W^\phi\|_\square,
\]
where the infimum is over all invertible maps $\phi\colon [0,1]\to[0,1]$ such
that both $\phi$ and its inverse are measure preserving, and $W^\phi$ is
defined by $W^\phi(x,y)=W(\phi(x),\phi(y))$ (see
\cite{dense0,LSz-Anal,dense1}); such a map $\phi$ is called a
\emph{measure-preserving bijection}.  After identifying graphons with cut
distance zero, the space of graphons equipped with the metric
$\delta_\square$ becomes a metric space.

To define the cut distance between two weighted graphs, we assign a graphon
$W^G$ to a weighted graph $G$ as follows: let $n=|V(G)|$, identify $V(G)$
with $[n]$, and let $I_1,\dots,I_n$ be consecutive intervals in $[0,1]$ of
lengths $\alpha_1(G)/\alpha_G$, $\dots$, $\alpha_n(G)/\alpha_G$,
respectively. We then define $W^G$ to be the step function that is constant
on sets of the form $I_u\times I_v$ with
\begin{equation}
\label{W-G-def}
W^G(x,y)=\beta_{uv}(G)\qquad\text{if}\qquad (x,y)\in I_u\times I_v.
\end{equation}
Informally, we consider the adjacency matrix of $G$ and replace each entry
$(u,v)$ by a square of size $\alpha_u(G)\alpha_v(G)/\alpha_G^2$ with the
constant function $\beta_{uv}$ on this square.

With this definition, one easily checks that the density  of a weighted graph
$G$ can be expressed as $\|G\|_1=\|W^G\|_1$. For dense graphs, one can define
a distance $\delta_\square(G,G')$ between two graphs by just considering the
cut distance between $W^G$ and $W^{G'}$. But for sparse graphs, the
inequality
\[
\delta_\square(W^G,W^{G'})\leq \|W^G - W^{G'}\|_\square \leq \|W^G\|_1+
\|W^{G'}\|_1
\]
means the cut distance is not very informative, since under this metric all
sparse graph sequences are Cauchy sequences.

To overcome this problem, we  identify weighted graphs whose edge weights
only differ by a multiplicative factor.\footnote{Of course, this slightly
decreases our ability to distinguish between dense graphs.} Explicitly, we
introduce the distance
\begin{equation}
\label{delcutn-def}
\delcutn(G,G')=\delta_\square\left(\frac{1}{\|G\|_1}W^G, \frac{1}{\|G'\|_1}W^{G'}\right),
\end{equation}
where in the degenerate case of a graph $G$ with $\|G\|_1=0$ we define
$\frac{1}{\|G\|_1}W^G$ to be zero. As an example, with this definition, two
random graphs $G_{n,p}$ for different $p$ can be shown to be close in the
metric $\delcutn$, as are two random graphs with different numbers of nodes,
at least as long as $pn\to\infty$ as $n\to\infty$ (see
Section~\ref{subsec:exandcounter}).

\begin{definition}[\cite{part1}]
\label{def:metric-con} Let $(G_n)_{n \ge 0}$ be a sequence of weighted
graphs, and let $W$ be a graphon. We say that $(G_n)_{n \ge 0}$ is
\emph{convergent in metric} if $(G_n)_{n \ge 0}$ is a Cauchy sequence in the
metric $\delcutn(G,G')$ defined in \eqref{delcutn-def}, and we say that $G_n$
\emph{converges to $W$ in metric} if $\delta_\square\Bigl(\frac
1{\|G_n\|_1}W^{G_n},W\Bigr)\to 0$. (Again, we set $\frac
1{\|G_n\|_1}W^{G_n}=0$ if $\|G_n\|_1=0$.)
\end{definition}

\subsection{Convergence of quotients}
\label{sec:G-quotients}

The next object we define is convergence of quotients. To formalize this,
consider a weighted graph $G$ and a partition $\PP=(V_1,\dots,V_q)$ of $V(G)$
into $q$ parts, some of which could be empty. Equivalently, consider a map
$\phi\colon V(G)\to [q]$ (related to $\PP$ by setting $\phi(x)=i$ iff $x\in
V_i$). We will define a quotient $G/\phi=G/\PP$ as a pair
$(\alpha,\beta)=(\alpha(G/\phi),\beta(G/\phi))$, where $\alpha\in\R^q$ is a
vector encoding the total vertex weights of the classes in $\PP$ and
$\beta\in\R^{q \times q}$ is a matrix encoding the number of edges (weighted
by their edge weights) between different classes. Explicitly,
\begin{equation}
\label{G-phi-alpha-def}
\alpha_i(G/\phi)=\frac {\alpha_{V_i}(G)}{\alpha_G}
\end{equation}
and
\begin{equation}
\label{G-phi-def}
\beta_{ij}(G/\phi)=\frac {1}{\|G\|_1}
\sum_{(u,v)\in V_i\times V_j}
\frac{\alpha_u(G)}{\alpha_G}
\frac{\alpha_v(G)}{\alpha_G}
\beta_{uv}(G).
\end{equation}
(In the degenerate case where $G$ has no edges and $\|G\|_1=0$, we set
$\beta(G/\phi)=0$.) We call $G/\phi$ a \emph{$q$-quotient} of $G$, and we
denote the set of all $q$-quotients of $G$ by $\SS_q(G)$.  Note that without
the normalization factor $\frac 1{\|G\|_1}$  in \eqref{G-phi-def}, the
weights $\beta_{ij}(G/\phi)$ would scale with the density of $G$, which means
that all quotients of a sparse sequence would tend to zero.  We have chosen
this factor in such a way that
\[
\|\beta(G/\phi)\|_1=\sum_{i,j}|\beta_{ij}(G/\phi)|\leq 1,
\]
with equality if and only if $G$ has non-negative weights and density
$\|G\|_1>0$.

We will consider $\SS_q(G)$ as a subset of
\begin{equation}\label{S_q-def}
\SS_q=\Bigl\{(\alpha,\beta) \in [0,1]^q \times [-1,1]^{q \times q} :  \sum_{i\in [q]}\alpha_i=1\text{ and }\sum_{i,j\in [q]}|\beta_{ij}|\leq 1
\Bigr\},
\end{equation}
equipped with the usual $\ell_1$ distance on $\R^{q+q^2}$,
\begin{equation}
\label{d1-general}
d_1((\alpha,\beta),(\alpha',\beta'))=
\sum_{i\in [q]}|\alpha_i-\alpha_i'|+\sum_{i,j\in [q]}|\beta_{ij}-\beta_{ij}'|,
\end{equation}
which turns $\SS_q$ into a compact metric space $(\SS_q,d_1)$, a fact we will
use repeatedly in this paper.  For $\ba\in\PD_q$, we define the subspace
\begin{equation}
\label{eq:Sa-def}
\SS_\ba = \{(\alpha,\beta) \in \SS_q : \alpha = \ba\},
\end{equation}
which is closed and hence also compact.
Note that our normalizations are a little
different from those in \cite{dense2}, in order to ensure
compactness.\footnote{Specifically, the analogue of (2.4) and (2.5)
in~\cite{dense2} would be to use $\beta_{ij}(G/\phi)/\big(\alpha_i(G/\phi)
\alpha_j(G/\phi)\big)$ instead of $\beta_{ij}(G/\phi)$, while modifying the
definition of $d_1$ accordingly. This would encode essentially the same
information, but the analogue of $\SS_q$ would not be compact.}

The quotients of a graph $G$ allow one to express many properties of interest
to combinatorialists and computer scientists in a compact form. For example,
the size of a maximal cut in a simple graph $G$,
\[
\MaxCut(G)=\max_{W\subseteq V(G)} \sum_{(x,y)\in W\times (V(G)\setminus W)} \beta_{xy}(G),
\]
can be expressed as
\[
\MaxCut(G)=|E(G)|
\max_{(\alpha,\beta)\in\SS_2(G)}\bigl(\beta_{12}+\beta_{21}\bigr).
\]
Restricting oneself to the subset of quotients $(\alpha,\beta)\in\SS_2(G)$
such that $\alpha_1=\alpha_2=1/2$, one can express quantities like min- or
max-bisection, and considering $\SS_q(G)$ for $q>2$, one obtains weighted
versions of max-cut for partitions into more than two sets.

To define convergence of quotients, we need the Hausdorff metric on subsets
of a metric space $(X,d)$. As usual, it is a metric $d^\Haus$ on the set of
nonempty compact subsets of $X$, defined by
\[
d^\Haus(S,S')= \max \left\{ \sup_{x\in S} d(x, S') , \sup_{y\in S'}
d(y,S)\right\},
\]
where
\[
d(x,S)=\inf_{y\in S} d(x,y).
\]
If $d$ is a complete metric, then so is $d^\Haus$ (see \cite{GBP}), and the
same holds for total boundedness. Thus, starting from the metric space
$(\SS_q,d_1)$, this gives a metric $d_1^\Haus$ on the space of
nonempty compact subsets of $\SS_q$, and this space inherits compactness from the
compactness of $(\SS_q,d_1)$.

\begin{definition}
\label{def:quotients-con} Let $(G_n)_{n \ge 0}$ be a sequence of weighted
graphs. We say that the sequence $(G_n)_{n \ge 0}$ has \emph{convergent
quotients} if for each $q$, there exists a closed set
$\SS_q^\infty\subseteq\SS_q$ such that $\SS_q(G_n)$ converges to
$\SS_q^\infty$ in the Hausdorff metric.
\end{definition}

\begin{remark}
\label{rem:Sqinfty} Note that the closedness of the set
$\SS_q^\infty\subseteq\SS_q$ can be assumed without loss of generality (every
set has Hausdorff distance zero from its closure, which is why the Hausdorff
metric is restricted to closed sets). Furthermore, because of the compactness
of the Hausdorff metric, convergence of quotients is equivalent to the
statement that the quotients $\SS_q(G_n)$ form a Cauchy sequence. It is then
easy to verify that the limiting set $\SS_q^\infty$ can be expressed
as\footnote{To see why, note that if $(\alpha,\beta) \in \SS_q^\infty$, then
\[
d_1((\alpha,\beta),\SS_q(G_n)) \le d_1^\Haus(\SS_q^\infty,S_q(G_n)) \to 0,
\]
while if $d_1((\alpha,\beta),\SS_q(G_n))\to 0$, then combining this limit with
$d_1^\Haus(\SS_q^\infty,S_q(G_n)) \to 0$ and the fact that $\SS_q^\infty$ is closed
shows that $(\alpha,\beta) \in \SS_q^\infty$.}
\[
\SS_q^\infty=\bigl\{(\alpha,\beta)\in\SS_q : d_1\bigl((\alpha,\beta),\SS_q(G_n)\bigr)\to 0\bigr\}.
\]
\end{remark}

\subsection{Statistical physics and multiway cuts}
\label{sec:stat-phys}

Next we define some notions motivated by concepts from statistical physics
(or, for a different audience, by the concept of graphical models in machine
learning).

Consider a weighted graph $G$. We will randomly color the vertices of $G$
with $q$ colors; i.e., we will consider random maps $\phi\colon V(G)\to [q]$.
We allow for all possible maps, not just proper colorings, and call such a
map a \emph{spin configuration}. To make the model nontrivial, different spin
configurations get different weights, based on a symmetric $q\times q$ matrix
$J$ with entries $J_{ij}\in \R$ called the \emph{coupling matrix}. Given $G$
and $J$, a map $\phi\colon V(G)\to [q]$ then gets an \emph{energy}
\begin{equation}
\label{Ephi(G,J)-def}
E_\phi(G,J)=
-\frac 1{\|G\|_1}\sum_{u,v\in V(G)}\frac{\alpha_u(G)\alpha_v(G)}{\alpha_G^2}\beta_{uv}(G)J_{\phi(v)\phi(u)}.
\end{equation}
(If $G$ has no edges, we set this term equal to zero.) Given a vector
$\ba=(a_1,\dots,a_q)$ of nonnegative real numbers adding up to $1$ (we denote
the set of these vectors by $\PD_q$), we consider configurations $\phi$ such
that the (weighted) fraction of vertices mapped onto a particular color $i\in
[q]$ is near to $a_i$.  More precisely, we consider configurations $\phi$ in
\[
\Omega_{\ba,\eps}(G)=\left\{\phi\colon [q]\to V(G) : \left|
\frac{\alpha_{\phi^{-1}(\{i\})}(G)}{\alpha_G}-a_i\right|
\leq
\eps\text{ for all }i\in [q]
\right\}.
\]
On $\Omega_{\ba,\eps}(G)$ we then define a probability distribution
\[
\mu_{G,J}^{({\ba,\eps})}(\phi)=
\frac 1{Z_{G,J}^{({\ba,\eps})}} e^{-|V(G)|E_\phi(G,J)},
\]
where ${Z_{G,J}^{({\ba,\eps})}}$ is the normalization factor
\begin{equation}
\label{Z(G-ba)-def}
Z_{G,J}^{({\ba,\eps})}=\sum_{\phi\in \Omega_{\ba,\eps}(G)} e^{-|V(G)|E_\phi(G,J)}.
\end{equation}
The distribution $\mu_{G,J}^{({\ba,\eps})}$ is usually called the {\it
microcanonical Gibbs distribution of the model $J$ on $G$}, and
${Z_{G,J}^{({\ba,\eps})}}$ is called the \emph{microcanonical partition
function}.

In this paper, we will not analyze the particular properties of the
distribution $\mu_{G,J}^{({\ba,\eps})}$, but we will be interested in the
normalization factor, or more precisely its normalized logarithm
\begin{equation}
\label{Fa(G)-def}
F_{\ba,\eps}(G,J)=-\frac 1{|V(G)|}\log {Z_{G,J}^{({\ba,\eps})}},
\end{equation}
which is called the \emph{microcanonical free energy}.  We will also be
interested in the dominant term contributing to ${Z_{G,J}^{({\ba,\eps})}}$,
or more precisely its normalized logarithm, the \emph{microcanonical ground
state energy}
\begin{equation}
\label{Ea(G)-def}
E_{\ba,\eps}(G,J)=\min_{\phi\in \Omega_{\ba,\eps}(G) } E_\phi(G,J).
\end{equation}
Note that the energy $E_\phi(G,J)$ has been normalized in such a way that
$|E_\phi(G,J)|\leq \|J\|_\infty$ (where
$\|J\|_\infty=\max_{i,j\in[q]}|J_{ij}|$), and $\Omega_{\ba,\eps}(G)\neq
\emptyset$ as long as $\eps\geq\alpha_{\max}(G)/\alpha_G$, implying that
under this condition, $Z_{G,J}^{({\ba,\eps})}\geq
e^{-|V(G)|E_{\ba,\eps}(G,J)}\geq e^{-|V(G)| \|J\|_\infty}$.  Thus, for fixed
$J$ the microcanonical energies and free energies are of order one.

\begin{definition}
\label{def:right-conv} Let $(G_n)_{n \ge 0}$ be a sequence of weighted
graphs. We say that
\begin{enumerate}
\item $(G_n)_{n \ge 0}$ has \emph{convergent microcanonical ground state
    energies} if the limit
\begin{equation}
\label{Ea-limit-def}
E_\ba(J)=\lim_{\eps\to 0}\limsup_{n\to\infty}E_{\ba,\eps}(G,J)=
\lim_{\eps\to 0}\liminf_{n\to\infty}E_{\ba,\eps}(G,J)
\end{equation}
exists for all $q\in\N$, $\ba\in \PD_q$, and symmetric $J\in\R^{q \times
q}$, and

\item $(G_n)_{n \ge 0}$ has \emph{convergent microcanonical free energies}
    if the limit
\begin{equation}
\label{Fa-limit-def}
F_\ba(J)=\lim_{\eps\to 0}\limsup_{n\to\infty}F_{\ba,\eps}(G,J)=
\lim_{\eps\to 0}\liminf_{n\to\infty}F_{\ba,\eps}(G,J)
\end{equation}
exists for all $q\in\N$, $\ba\in \PD_q$, and symmetric $J\in\R^{q \times
q}$.
\end{enumerate}
\end{definition}

Recall that the microcanonical ground state energy describes the largest term
contributing to the microcanonical partition function
${Z_{G,J}^{({\ba,\eps})}}$. Using the fact that this partition function
contains at least one and at most $q^{|V(G)|}$ terms, we will see that a
scaling argument shows that convergence of the microcanonical free energies
implies convergence of the microcanonical ground state energies.  On the
other hand, the energy of a configuration $\phi$ can be expressed in terms of
the quotient $G/\phi$ as
\begin{equation}
\label{eq:E-phi-quotient}
 E_\phi(G,J)=-\langle\beta(G/\phi),J\rangle,
\end{equation}
where
\[
 \langle\beta,J\rangle=
\sum_{i,j}\beta_{ij}J_{ij}.
\]
Using this identity, we will express the microcanonical ground state energy
as a minimum over quotients, which in turn can be used to show that
convergence of quotients implies convergence of the microcanonical ground
state energies.

The following theorem gives a precise statement of these facts.  We will restate the
theorem as part of Lemma~\ref{lem:FatoFetc} and Theorem~\ref{thm:quotients-to-E}
and prove it in Section~\ref{sec:Gproofs}.

\begin{theorem}
\label{thm:implies-E} Let $q\in\N$ and let $(G_n)_{n \ge 0}$ be a sequence of
weighted graphs.
\begin{enumerate}
\item[(i)] If  $\SS_q(G_n)$ converges to a closed set $\SS_q^\infty$ in the
    Hausdorff metric, then the limit \eqref{Ea-limit-def} exists for all
    $\ba\in \PD_q$ and all symmetric $J\in\R^{q \times q}$
 and can be expressed as
\[
E_\ba(J)=-\max_{\substack{ (\alpha,\beta)\in\SS_q^\infty\cap\SS_\ba}}\langle \beta,J\rangle.
\]

\item[(ii)] Let $\ba\in \PD_q$.  If $|V(G_n)|\to\infty$ and the limit
    \eqref{Fa-limit-def} exists for all symmetric $J\in\R^{q \times q}$,
    then the limit \eqref{Ea-limit-def} exists for all such $J$ and
\[
E_\ba(J)= \lim_{\lambda\to\infty}\frac 1\lambda F_\ba(\lambda J).
\]
\end{enumerate}
\end{theorem}

\begin{remark}
Definition~\ref{def:right-conv} differs from that given in \cite{dense2} for
dense graphs in that we are taking the double limit of first sending
$n\to\infty$ and then sending $\eps\to 0$, rather than a single limit with an
$n$-dependent $\eps=\eps_n$. (In \cite{dense2}, $\eps_n$ was chosen to be
$\alpha_{\max}(G_n)/\alpha_{G_n}$, even though all theorems involving the
microcanonical free energies required the additional assumption that $G_n$
has vertex weights one, corresponding to $\eps_n=1/|V(G_n)|$). While there is
some merit to the simplicity of a single limit, here we decided to follow the
spirit of the definitions from mathematical statistical physics, where the
formulation of a double limit is standard; it is also more consistent with
the double limits usually taken in the theory of large deviations, where an
$n$-dependent $\eps$ usually makes no sense.

However, the two definitions are equivalent if $G_n$ is dense with bounded
edge weights and vertex weights one (this follows from
Theorem~\ref{thm:G-conv-equiv} below, because such graphs are $L^\infty$
upper regular). Thus, as far as the results of \cite{dense2} are concerned,
there is no difference between the two definitions.
\end{remark}

\subsection{Large deviation convergence}
\label{sec:LD-convergence}

As we have seen in the last section, the quotients of a graph $G$ provide
enough information to calculate the microcanonical ground state energies
\eqref{Ea(G)-def}, since the quotients tell us which energies $E_\phi(G,J)$
can be realized. However, to calculate the microcanonical free energies
\eqref{Fa(G)-def} we need to know a little more, namely how often a term with
given energy appears in the sum \eqref{Z(G-ba)-def}.

This leads to the notion of large deviation convergence (LD convergence),
which was first introduced in the context of bounded degree graphs
\cite{BCG}, where it turned out to be strictly stronger than convergence of
quotients.  Roughly speaking, this notion codifies how often a given quotient
$(\alpha,\beta)\in \SS_q(G)$ appears in a sum of the form
\eqref{Z(G-ba)-def}. Or, put differently, it specifies the probability that
for a uniformly random map $\phi\colon V(G)\to [q]$, the quotient $G/\phi$ is
approximately equal to $(\alpha,\beta)$.  The precise definition is as
follows:

\begin{definition}
\label{def:LD-conv} Let $q\in\N$, let $(G_n)_{n \ge 0}$ be a sequence of
weighted graphs, and let $\PP_{q,G_n}$ be the probability distribution of
$G_n/\phi$ when $\phi\colon V(G_n)\to [q]$ is chosen uniformly at random. We
say that $(G_n)_{n \ge 0}$ is \emph{$q$-LD convergent} if $|V(G_n)|\to\infty$
and
\begin{equation}\label{LD-conv-alt}
\begin{split}
&\lim_{\eps\to 0}\liminf_{n\to\infty}\frac{\log
\PP_{q,G_n}\bigl[d_1((\alpha,\beta),G_n/\phi)\leq \eps\bigr]}{|V(G_n)|}\\
&\qquad 
=\lim_{\eps\to
0}\limsup_{n\to\infty}\frac{\log \PP_{q,G_n}\bigl[d_1((\alpha,\beta),G_n/\phi)\leq \eps\bigr]}{|V(G_n)|}
\end{split}
\end{equation}
and say it is \emph{$q$-LD convergent} with rate function $I_q\colon \SS_q\to
[0,\infty]$ if the above limit is equal to $-I_q((\alpha,\beta))$.  We say
that $(G_n)_{n \ge 0}$ is \emph{LD convergent} if it is $q$-LD convergent for
all $q\in\N$.
\end{definition}

The following theorem states that LD convergence is at least as strong as
convergence of quotients and convergence of the microcanonical free energies.
We prove it in Sections~\ref{sec:proof-LD->Sq} and \ref{sec:proof-LD->Fa}.

\begin{theorem}
\label{thm:LD->F-and-Sq} Let  $q\in\N$ and let $(G_n)_{n \ge 0}$ be a
sequence of weighted graphs. If $(G_n)_{n \ge 0}$ is $q$-LD convergent with
rate function $I_q$, then the following hold:
\begin{enumerate}
\item[(i)] The sets of quotients $\SS_q(G_n)$ converge to the closed set
\[
\SS_q(I_q)=\{(\alpha,\beta)\in\SS_q :
I_q((\alpha,\beta))<\infty\}
\]
in the Hausdorff metric.

\item[(ii)] For all $\ba\in \PD_q$ and all symmetric $J\in\R^{q \times q}$,
    the microcanonical free energies converge to
\[
F_\ba(I_q,J)=\inf_{(\alpha,\beta)\in\SS_{\ba}}
\Bigl(-\langle \beta,J\rangle+ I_q((\alpha,\beta))\Bigr)-\log q.
\]
\end{enumerate}
\end{theorem}

\subsection{Limiting expressions for convergent sequences of graphs}
\label{sec:limit-expressions}

The results stated so far, namely Theorems~\ref{thm:implies-E}
and~\ref{thm:LD->F-and-Sq}, raise the question of whether the four notions of
convergence considered in these theorems are equivalent.  They also raise the
question of whether the limits of the quotients, microcanonical ground state
energies, and free energies as well as the rate functions $I_q$ can be
expressed in terms of a limiting graphon. It turns out that the answers to
these two questions are related, and that we have equivalence if we postulate
convergence to a graphon $W\in L^1$.

We need some definitions. All of them rely on the notion of a
\emph{fractional partition} of $[0,1]$ into $q$ classes (briefly, a
\emph{fractional $q$-partition}), which we define as a $q$-tuple of
measurable functions $\rho_1,\dots,\rho_q\colon [0,1]\to[0,1]$ such that
$\rho_1(x)+\dots+\rho_q(x)=1$ for all $x\in [0,1]$. We denote the set of
fractional $q$-partitions by $\FP_q$. To each fractional partition
$\rho\in\FP_q$, we assign a weight vector
$\alpha(\rho)=(\alpha_1(\rho),\dots,\alpha_q(\rho))\in\PD_q$ and an entropy
$\Ent(\rho)\in[0,\log q]$ by setting
\[
\alpha_i(\rho)=\int_0^1\rho_i(x) \,dx
\]
and
\[
\Ent(\rho)=\int_0^1 \Ent_x(\rho)\,dx
\quad\text{with}\quad
\Ent_x(\rho)=- \sum_{i=1}^q
\rho_i(x)\log\rho_i(x)
\]
(with $0 \log 0 = 0$). Let
\[
\iv\SS_q = \Bigl\{ (\alpha,\beta) \in [0,1]^q \times
\R^{q\times q} : \sum_{i\in[q]}\alpha_i=1 \Bigr\}
\]
(in comparison with the definition \eqref{S_q-def} of $\SS_q$, we do not
restrict $\beta$). Given a graphon $W$ and a fractional $q$-partition
$\rho\in\FP_q$, we then define the quotient $W/\rho$ to be the pair
$(\alpha,\beta)\in\iv\SS_q$ where
\[
\alpha_i(W/\rho)
=\alpha_i(\rho)
\]
and
\[
\beta_{ij}(W/\rho)= \int_{[0,1]^2} \rho_i(x)\rho_j(y)W(x,y)\,dx\,dy
\]
for $i,j\in [q]$. We call $W/\rho$ a \emph{fractional $q$-quotient} of $W$.
Let $\iv\SS_q(W)$ denote the set of all fractional $q$-quotients of $W$, and
for $\ba \in \PD_q$, let $\iv\SS_\ba(W)$ denote the set of pairs in
$\iv\SS_q(W)$ whose first coordinate equals $\ba$. It will be shown in
Proposition~\ref{prop:S-Closed} that $\iv\SS_q(W)$ is compact.

Next we define the microcanonical ground state energies and free energies of
a graphon $W$. Given an integer $q\geq 1$ and a symmetric matrix $J\in\R^{q
\times q}$, we define the \emph{energy of a fractional partition}
$\rho\in\FP_q$ to be
\[
\EE_\rho(W,J)=
-\sum_{i,j}J_{ij}\ \int_{[0,1]^2}\rho_i(x)\rho_j(y)W(x,y)\,dx\,dy.
\]
For $\bf a\in \PD_q$, the \emph{microcanonical ground state energy} is
defined as
\begin{equation}
\label{EaW-def}
\EE_\ba(W,J)=\inf_{\rho : \alpha(\rho)=\ba}\EE_\rho(W,J),
\end{equation}
while the \emph{microcanonical free energy} is defined as
\begin{equation}
\label{FaW-def}
\FF_\ba(W,J)
=\inf_{\rho:\alpha(\rho)=\ba}
\biggl(
\EE_\rho(W,J)-\Ent(\rho)
\biggr).
\end{equation}
The infima in these equations are over all fractional $q$-partitions of
$[0,1]$ such that $\alpha(\rho)=\ba$. Note that all these quantities are well
defined because $0\leq\Ent_x(\rho)\leq \log q$ and
\[
|\EE_\rho(W,J)|
\leq \|J\|_\infty\|W\|_1
.
\]

Finally, the LD rate function $I_q(F,W)$ is defined as
\begin{equation}\label{IW-def}
I_q((\alpha,\beta),W) =\inf_{\rho \in \FP_q: W/\rho = (\alpha,\beta)} (\log q- \Ent(\rho)),
\end{equation}
Note that $I_q((\alpha,\beta),W) \in [0,\log q]$ if $(\alpha,\beta) \in
\iv\SS_q(W)$ and $I_q((\alpha,\beta),W) = \infty$ if $(\alpha,\beta) \notin
\iv\SS_q(W)$.

We are now  ready to state the main theorem of this paper.  Recall that
graphons are assumed to be $L^1$ (and not necessarily $L^\infty$, as in some
papers in the literature).

\begin{theorem}
\label{thm:limit-expressions} Let $W$ be a graphon, and let $(G_n)_{n \ge 0}$
be a sequence of weighted graphs with no dominant nodes, in the sense that
$\alpha_{\max}(G_n)/\alpha_{G_n}\to 0$. Then the following statements are
equivalent:
\begin{enumerate}
\item[(i)] $(G_n)_{n \ge 0}$ converges to $W$ in metric.

\item[(ii)] For all $q\in\N$, $\SS_q(G_n)\to\iv\SS_q(W)$ in the Hausdorff
    metric $d_1^\Haus$.

\item[(iii)]The  microcanonical ground state energies of $(G_n)_{n \ge 0}$
    converge to those of $W$.
\end{enumerate}
If all the vertices of $G_n$ have weight one, then the following two
statements are also equivalent to (i):
\begin{enumerate}
\item[(iv)] $(G_n)_{n \ge 0}$ is LD convergent with rate function
    $I_q=I_q(\cdot,W)$.
\item[(v)] The microcanonical free energies of $(G_n)_{n \ge 0}$ converge
    to those of $W$.
\end{enumerate}
\end{theorem}

We prove this theorem in Section~\ref{sec:graphs->graphons}.

\subsection{Uniform upper regularity}
\label{sec:upper-reg}

It is natural to ask whether one can state
Theorem~\ref{thm:limit-expressions} without reference to the limiting graphon
$W$. It turns out that the answer is yes, and in fact this reformulation
(Theorem~\ref{thm:G-conv-equiv}) will play a key role in the proof.  To state
this theorem, we need the notion of upper regularity, which first arose in
our study of subsequential metric convergence in \cite{part1} and plays a key
role both in that paper and in this one.

To define this concept, we define the $L^p$ norm of a weighted graph $G$ to
be
\[
\|G\|_p=\left(\sum_{x,y\in V(G)}\frac{\alpha_x(G)\alpha_y(G)}{\alpha_G^2}|\beta_{xy}(G)|^p\right)^{1/p},
\]
and for $p=\infty$ we set
\[
\|G\|_\infty=\max_{\substack{x,y\in V(G)\\\alpha_x(G),\alpha_y(G) > 0}}|\beta_{xy}(G)|.
\]
As we already have seen in Section~\ref{sec:metric-conv}, when studying graph
convergence for sparse graphs, it is natural to reweight the edge weights by
$\frac 1{\|G\|_1}$ to obtain a weighted graph which does not go to zero for
trivial reasons. In order to control the now possibly large entries of the
adjacency matrix of the weighted graph $\frac 1{\|G\|_1} G$, one might want
to require the $L^p$ norm of $\frac 1{\|G\|_1} G$ to be bounded, but this
turns out to be too restrictive. Instead, we will use a weaker condition,
which requires the $L^p$ norm of $\frac 1{\|G\|_1} G$ to be bounded ``on
average,'' at least when the averages are taken over sufficiently large
blocks. To make this precise, we need some additional notation.

Given a weighted graph $G$ and a partition $\PP=\{V_1,\dots,V_q\}$ of $V(G)$
into disjoint sets $V_1,\dots,V_q$, we define $G_\PP$ to be the weighted
graph with the same vertex weights as $G$ and edge weights which are defined
by averaging over the blocks $V_i\times V_j$, suitably weighted by the vertex
weights:
\begin{equation}
\label{G_P-def}
\beta_{xy}(G_\PP)=\frac 1{\alpha_{V_i}(G)\alpha_{V_j}(G)}\sum_{(u,v)\in V_i\times V_j}\alpha_u(G)\alpha_v(G)\beta_{uv}(G)
\end{equation}
if $(x,y)\in V_i\times V_j$ and $\alpha_{V_i}(G)\alpha_{V_j}(G)>0$, while we
set $\beta_{xy}(G_\PP)=0$ if either $x$ or $y$ lie in a block $V_k(G)$ with
total node weight $\alpha_{V_k}(G)=0$.

\begin{definition}\label{def:up-reg}
Let $G$ be a weighted graph, let $C,\eta>0$, and let $p>1$. We say that $G$
is \emph{$(C,\eta)$-upper $L^p$ regular} if
$\alpha_{\max}(G)\leq\eta\alpha_G$ and
\[
\|G_\PP\|_p\leq C \|G\|_1
\]
for all partitions $\PP=\{V_1,\dots,V_q\}$ for which
$\min_i\alpha_{V_i}\geq\eta\alpha_G$. We say that a sequence of graphs
$(G_n)_{n \ge 0}$ is \emph{$C$-upper $L^p$ regular} if there exists a
sequence $\eta_n \to 0$ such that $G_n$ is $(C,\eta_n)$-upper regular, and we
say that $(G_n)_{n \ge 0}$ is \emph{$L^p$ upper regular} if there exists a
$C<\infty$ such that $(G_n)_{n \ge 0}$ is $C$-upper $L^p$ regular.
\end{definition}

The definition of $L^1$ upper regularity always holds vacuously, but the
following definition of uniform upper regularity turns out to be the correct
$L^1$ analogue, as described in Appendix~C of \cite{part1}.  It is closely
related to the notion of uniform integrability of a set of graphons (see
Section~\ref{sec:Graphon-Results}), and it is the notion we will need in this
paper.

\begin{definition}\label{def:unif-up-reg}
Let $\eta>0$ and let $K \colon (0,\infty)\to (0,\infty)$ be any function.  We
say that a weighted graph $G$ is \emph{$(K,\eta)$-upper regular} if
$\alpha_{\max}(G)\leq\eta\alpha_G$ and
\begin{equation}
\label{K-eta-upreg}
\sum_{x,y\in V(G)}\frac{\alpha_x(G)\alpha_y(G)}{\alpha_G^2}
\frac{|\beta_{xy}(G_\PP)|}{\|G\|_1}\one_{|\beta_{xy}(G_\PP)|\geq K(\eps)\|G\|_1}
\le \varepsilon
\end{equation}
for all $\eps>0$ and all partitions $\PP=\{V_1,\dots,V_q\}$ for which
$\min_i\alpha_{V_i}\geq\eta\alpha_G$. We say that a sequence of graphs
$(G_n)_{n \ge 0}$ is \emph{$K$-upper  regular} if there exists a sequence
$\eta_n \to 0$ such that $G_n$ is $(K,\eta_n)$-upper regular, and we say that
$(G_n)_{n \ge 0}$ is \emph{uniformly upper regular} if there exists a
function $K \colon (0,\infty)\to (0,\infty)$ such that $(G_n)_{n \ge 0}$ is
$K$-upper regular.
\end{definition}

Note that the properties of $L^p$ upper regularity and uniform upper
regularity require $(G_n)_{n \ge 0}$ to have \emph{no dominant nodes}, a
property we already encountered in Theorem~\ref{thm:limit-expressions}. One
of the main results of \cite{part1} is the following theorem.

\begin{theorem}[Theorem~C.7 in \cite{part1}] 
\label{thm:compact} Let  $(G_n)_{n \ge 0}$ be a uniformly upper regular
sequence of weighted graphs. Then $(G_n)_{n \ge 0}$ contains a subsequence
that is convergent in metric. Furthermore, if $(G_n)_{n \ge 0}$ is convergent
in metric, then there exists a graphon $W$ such that $G_n$ converges to $W$
in metric.
\end{theorem}

Conversely, it was shown in  \cite{part1} that every sequence of weighted
graphs which converges in metric to a graphon and has no dominant nodes must
be upper regular. The precise statement is given by the following theorem,
which follows immediately from Corollary~2.11 and Proposition~C.5 in
\cite{part1}.

\begin{theorem}[\cite{part1}]
\label{thm:metric>upper-regular} Let $(G_n)_{n \ge 0}$ be a sequence of
weighted graphs without dominant nodes, and assume that $G_n$ converges to
some graphon $W$ in metric. Then $(G_n)_{n \ge 0}$ is uniformly upper
regular.  If $W$ is in $L^p$, then $(G_n)_{n \ge 0}$ is $L^p$-upper regular.
\end{theorem}

A uniformly upper regular sequence of simple graphs must have unbounded
average degree, by Proposition~C.15 in \cite{part1}.  This corresponds to the
fact that graphons are not the appropriate limiting objects for graphs with
bounded average degree (although they apply to all other sparse graphs).

Returning to the subject of this paper, the question of whether the five
versions of convergence defined in Sections~\ref{sec:metric-conv} through
\ref{sec:LD-convergence} are equivalent, we are now ready to state our
results without reference to a limiting graphon.

\begin{theorem}
\label{thm:G-conv-equiv} Let $(G_n)_{n \ge 0}$ be a uniformly upper regular
sequence of weighted graphs. Then the following three statements are
equivalent:
\begin{enumerate}
\item[(i)] $(G_n)_{n \ge 0}$ is convergent in metric.

\item[(ii)]  $(G_n)_{n \ge 0}$ has convergent quotients.

\item[(iii)] $(G_n)_{n \ge 0}$ has convergent microcanonical ground state
    energies.
\end{enumerate}
If all the vertices of $G_n$ have weight one, then the following two
statements are also equivalent to (i):
\begin{enumerate}
\item[(iv)] $(G_n)_{n \ge 0}$ is LD convergent.

\item[(v)] $(G_n)_{n \ge 0}$ has convergent microcanonical free energies.
\end{enumerate}
\end{theorem}

Note that by Theorems~\ref{thm:implies-E} and \ref{thm:LD->F-and-Sq}, we
already know that (iv) implies both (v) and (ii), and that both (v) and (ii)
imply (iii); in fact, we need neither node weights one, nor the assumption of
upper regularity.  So the important part of this theorem is that under the
assumption of uniform upper regularity, convergence in metric implies
convergence of quotients (and LD convergence, if we assume node weights one),
and convergence of the microcanonical ground state energies implies
convergence in metric.  We prove Theorem~\ref{thm:G-conv-equiv} in
Section~\ref{sec:graphs->graphons}.

One may want to know whether the assumption of upper regularity is actually
necessary for these conclusions to hold.  The answer is yes, by the following
example.

\begin{example}
\label{ex:1} Let $c_n\in\N$ be such that $c_n \to \infty$ and $c_n/n\to 0$
as $n\to\infty$, and let $G_n$ be the disjoint union of a complete graph on
$c_n$ nodes with $n-c_n$ isolated nodes.  Then $(G_n)_{n \ge 0}$ is LD
convergent (and hence has convergent quotients, microcanonical free energies
and microcanonical ground state energies); see Section~\ref{sec:ex5} below.
However, $(G_n)_{n \ge 0}$ is not a Cauchy sequence in the normalized cut
metric $\delcutn$ from \eqref{delcutn-def} and hence does not converge to
any graphon in metric (see the proof of Proposition~2.12(a) in
\cite{part1}).
\end{example}

The following theorem states that convergence of the quotients,
microcanonical ground state energies, or microcanonical free energies to
those of a graphon $W$ all imply upper regularity, as does LD convergence
with a rate function $I_q(\cdot,W)$ given in terms of a graphon $W$.  It is
the analogue of Theorem~\ref{thm:metric>upper-regular} for these notions of
convergence.

\begin{theorem}
\label{thm:inverse-limit} Let $(G_n)_{n \ge 0}$ be a sequence of weighted
graphs with no dominant nodes, and let $W$ be a graphon.  Then any of the
following conditions implies that $(G_n)_{n \ge 0}$ is uniformly upper
regular.
\begin{enumerate}
\item[(i)] The microcanonical ground state energies of $(G_n)_{n \ge 0}$
    converge to those of $W$.
\item[(ii)] For all $q\in\N$, $\SS_q(G_n)\to\iv\SS_q(W)$ in the Hausdorff
    metric $d_1^\Haus$.
\item[(iii)] The  microcanonical free energies of $(G_n)_{n \ge 0}$
    converge to those of $W$.
\item[(iv)] $(G_n)_{n \ge 0}$ is LD convergent with rate function
    $I_q=I_q(\cdot,W)$.
\end{enumerate}
\end{theorem}

Note that the first two assertions in this theorem already follow by
combining Theorem~\ref{thm:metric>upper-regular} with
Theorem~\ref{thm:limit-expressions}(i)--(iii). However, this is not how our
proofs of Theorems~\ref{thm:limit-expressions} and \ref{thm:inverse-limit}
proceed.  Instead of proving Theorem~\ref{thm:limit-expressions} directly, we
use uniform upper regularity to prove Theorem~\ref{thm:G-conv-equiv} in
Section~\ref{sec:graphs->graphons}. Then Theorem~\ref{thm:inverse-limit} is
exactly what we need to deduce Theorem~\ref{thm:limit-expressions} from
Theorem~\ref{thm:G-conv-equiv}, and we prove Theorem~\ref{thm:inverse-limit}
in Section~\ref{sec:inferring-upper-regularity}.

\section{Further definitions, remarks, and examples}
\label{sec:remarks}

\subsection{Convergence of free energies and ground state energies}

In addition to the microcanonical quantities introduced in
Section~\ref{sec:stat-phys}, statistical physicists often analyze the
unrestricted probability measure
\[
\mu_{G,J,h}(\phi)=
\frac 1{Z_{G,J,h}} e^{-|V(G)| E_\phi(G,J)+|V(G)|\langle h,\alpha(G /\phi)\rangle},
\]
where $h$ is a vector in $\R^q$ called the \emph{magnetic field},
\[
\langle h,\alpha\rangle=\sum_{i\in [q]}h_i\alpha_i,
\]
and $Z_{G,J,h}$ is the normalization factor
\begin{equation}
\label{Z(G)-def}
Z_{G,J,h}=\sum_{\phi\colon V(G)\to [q]} e^{-|V(G)|E_\phi(G,J)+|V(G)|\langle h,\alpha(G /\phi)\rangle},
\end{equation}
usually called the \emph{partition function}. The normalized logarithm of the
partition function is the \emph{free energy}
\[
F(G,J,h)=-\frac 1{|V(G)|}\log {Z_{G,J,h}}
\]
of the model $(J,h)$ on $G$, and the maximizer in the sum \eqref{Z(G)-def},
or more precisely its normalized logarithm, is the \emph{ground state energy}
\[
E(G,J,h)=\min_{\phi\colon [q]\to V(G)} \Bigl(E_\phi(G,J,h)-\langle h,\alpha(G/\phi)\rangle\Bigr).
\]

\begin{definition} \leavevmode 
\begin{enumerate}
\item $(G_n)_{n \ge 0}$ has \emph{convergent ground state energies} if the
    limit
\begin{equation}
\label{E-limit-def}
E(J,h)=\lim_{n\to\infty}
E(G_n,J,h)
\end{equation}
exists for all $q\in\N$, symmetric $J\in\R^{q \times q}$, and $h\in\R^q$.

\item $(G_n)_{n \ge 0}$ has \emph{convergent free energies} if the limit
\begin{equation}
\label{F-limit-def}
F(J,h)=\lim_{n\to\infty}
F(G_n,J,h)
\end{equation}
exists for all $q\in\N$, symmetric $J\in\R^{q \times q}$, and $h\in\R^q$.
\end{enumerate}
\end{definition}

These notions are implied by the microcanonical versions, and convergence of
free energies implies convergence of ground state energies.  This is the
content of the following lemma, which we will prove in
Section~\ref{sec:FatoFetc-proof}. Note that part (iii) is a restatement of
Theorem~\ref{thm:implies-E}(ii).

\begin{lemma}
\label{lem:FatoFetc} Let $(G_n)_{n \ge 0}$ be a sequence of weighted graphs
with $|V(G_n)|\to\infty$, and let $q\in\N$. Then the following hold:
\begin{enumerate}
\item[(i)] Let $J$ be a symmetric matrix in $\R^{q \times q}$, and assume
    that the limit \eqref{Fa-limit-def} exists for all $\ba\in \PD_q$. Then
    the limit \eqref{F-limit-def} exists for all $h\in \R^q$, and
\[
F(J,h)=\inf_{\ba\in \PD_q} \bigl(F_\ba(J)-\langle \ba, h \rangle\bigr).
\]

\item[(ii)] Let $J$ be a symmetric matrix in $\R^{q \times q}$, and assume
    that the limit \eqref{Ea-limit-def} exists for all $\ba\in \PD_q$. Then
    the limit \eqref{E-limit-def} exists for all $h\in \R^q$, and
\[
E(J,h)=\inf_{\ba\in \PD_q} \bigl(E_\ba(J)-\langle \ba, h \rangle\bigr).
\]

\item[(iii)] Let $\ba\in \PD_q$, and assume that the limit
    \eqref{Fa-limit-def} exists for all symmetric $J\in\R^{q \times q}$.
    Then the limit \eqref{Ea-limit-def} exists for all such $J$, and
\[
E_\ba(J)= \lim_{\lambda\to\infty}\frac 1\lambda F_\ba(\lambda J).
\]

\item[(iv)] Assume that the limit \eqref{F-limit-def} exists for all $h\in
    \R^d$ and all symmetric $J\in \R^{q \times q}$. Then the limit
    \eqref{E-limit-def} exists for all $h\in \R^d$ and all symmetric $J\in
    \R^{q \times q}$, and
\[
E(J,h)=\lim_{\lambda\to\infty}\frac 1\lambda F(\lambda J,\lambda h).
\]
\end{enumerate}
\end{lemma}

Convergence of the ground state and free energies is strictly weaker than
that of the microcanonical versions.  See Section~\ref{sec:ex4} for an
example.

On the other hand, we can use \eqref{eq:E-phi-quotient} to express both the
microcanonical ground state energies $E_{\ba,\eps}(G,J)$ and the unrestricted
ground state energies $E(G,J,h)$ as minima over quotients.  Using this fact,
it is not hard to show that convergence of quotients implies convergence of
the ground state energies as well as the microcanonical ground state
energies. This is the content of the following theorem, which again holds for
an arbitrary sequence, with no assumption about upper regularity. We prove
the theorem (which encompasses Theorem~\ref{thm:implies-E}(i)) in
Section~\ref{sec:proof-quotients-to-E}.

\begin{theorem}
\label{thm:quotients-to-E} Let $q\in\N$ and let $(G_n)_{n \ge 0}$ be a
sequence of weighted graphs such that $\SS_q(G_n)$ converges to a closed set
$\SS_q^\infty$ in the Hausdorff metric. Then the limit \eqref{Ea-limit-def}
exists for all $\ba\in\PD_q$ and all symmetric $J\in\R^{q \times q}$ and can
be expressed as
\[
E_\ba(J)=-\max_{\substack{ (\alpha,\beta)\in\SS_q^\infty\cap\SS_\ba}}\langle \beta,J\rangle.
\]
and the limit \eqref{E-limit-def} exists for all symmetric $J\in\R^{q \times
q}$ and all $h\in \R^q$ and can be expressed as
\[
E(J,h)=\min_{(\alpha,\beta)\in\SS_q^\infty}\bigl(- \langle \beta,J\rangle-\langle \alpha, h \rangle\bigr),
\]
\end{theorem}

Much as in Section~\ref{sec:limit-expressions}, we can write down limiting
expressions for a graphon $W$. The \emph{ground state energy of the model
$(J,h)$ on $W$} is
\[
\EE(W,J,h)=\inf_{\rho\in\FP_q}\biggl(
\EE_\rho(W,J)-\sum_ih_i\int_{[0,1]}\rho_i(x)\,dx\biggr)
\]
and its \emph{free energy} is defined as
\[
\FF(W,J,h) =\inf_{\rho\in\FP_q} \biggl( \EE_\rho(W,J)-\sum_ih_i\int_{[0,1]}\rho_i(x)\,dx-\Ent(\rho) \biggr).
\]
It follows from Lemma~\ref{lem:FatoFetc} and
Theorem~\ref{thm:limit-expressions} that if $(G_n)_{n \ge 0}$ has no dominant
nodes and converges to $W$ in metric, then its ground state energies converge
to those of $W$, and if all the vertices of $G_n$ have weight one, then the
free energies also converge to those of $W$.

\subsection{LD convergence}

\begin{remark}
\label{rem:LD-conv} It is not hard to see that $(G_n)_{n \ge 0}$ is $q$-LD
convergent if and only if $\PP_{q,G_n}$ obeys a large deviation principle
with speed $|V(G_n)|$, i.e., if there exists a lower semicontinuous function
$I_q\colon \SS_q\to [0,\infty]$ such that
\begin{equation}
\label{LD-conv-def}
\begin{split}
-\inf_{(\alpha,\beta)\in \mathring{S}}I_q((\alpha,\beta))
& \leq
\liminf_{n\to\infty}\frac{\log
\PP_{q,G_n}\bigl[ G_n/\phi\in \mathring{S} \bigr]}{|V(G_n)|}\\
&\leq
\limsup_{n\to\infty}\frac{\log
\PP_{q,G_n}\bigl[ G_n/\phi\in \bar S \bigr]}{|V(G_n)|}
\leq
-\inf_{(\alpha,\beta)\in \bar S}I_q((\alpha,\beta))
\end{split}
\end{equation}
for all sets $S\subseteq \SS_q$. Here $\bar S$ denotes the closure of $S$ and
$\mathring{S}$ its interior.

Indeed, assume that \eqref{LD-conv-def} holds for some lower semicontinuous
function $I_q\colon \SS_q\to [0,\infty]$.  By the lower semicontinuity of
$I_q$,
\[
I_q((\alpha,\beta))=\lim_{\eps\to 0}\inf\{I_q((\alpha',\beta')) :
d_1((\alpha,\beta),(\alpha',\beta'))<\eps\},
\]
which implies \eqref{LD-conv-alt} when inserted into \eqref{LD-conv-def}. It
turns out that \eqref{LD-conv-alt} is also sufficient for \eqref{LD-conv-def}
to hold. Indeed, under the assumption that the underlying metric space is
compact (which is the case here), the equality of the two limits in
\eqref{LD-conv-alt} implies that $\PP_{q,G_n}$ obeys a large deviation
principle with rate function given by $I_q$; see, for example, Theorem~4.1.11
in \cite{DZ} for the proof.
\end{remark}

\subsection{Examples}
\label{subsec:exandcounter}

In this section we give some examples of convergent graph sequences, as well
as a few counterexamples in which the equivalences in
Theorem~\ref{thm:G-conv-equiv} fail (of course because uniform upper
regularity does not hold).

\subsubsection{Erd\H os-R\'enyi random graphs}
The simplest example of a uniformly upper regular sequence---in fact an
$L^\infty$-upper regular sequence---is the standard Erd\H os-R\'enyi random
graphs $G_{n,p}$ obtained by connecting each pair of distinct vertices in
$[n]$ independently with probability $p$. Here $p$ can depend on $n$, as long
as $pn\to\infty$ as $n\to \infty$. Under this condition, $G_{n,p}$ converges
with probability one to the constant graphon $W=1$. This can proved in
several ways, for example by showing that in expectation  all the quotients
in $\SS_q(G_{n,p})$ converge to the corresponding quotients in $\iv\SS_q(W)$
and proving concentration with the help of Azuma's inequality.

\subsubsection{Stochastic block models}
Next we consider the block models obtained as follows.  Fix $k\in\N$, a
symmetric matrix $B=(b_{ij})_{i,j\in[k]}$ with entries $b_{ij}\geq 0$
satisfying $k^{-2}\sum_{i,j}b_{ij}=1$, and a target density $\rho_n\leq
1/\max b_{ij}$.  Divide $[n]$ into $k$ blocks $V_1,\dots,V_k$ of equal size
(or, in the case where $n$ is not divisible by $k$ of sizes differing by at
most $1$) and define $p_{uv}=\rho_n b_{ij}$ if $(u,v)\in V_i\times V_j$. Then
we connect vertices $u$ and $v$ with probability $p_{uv}$. If
$n\rho_n\to\infty$ as $n\to\infty$, then the resulting graph converges with
probability one to the step function $W$ that is equal to $b_{ij}$ on the
block $(\frac{i-1}n,\frac in]\times (\frac {j-1}n,\frac jn]$.  The proof can
again be obtained by proving that the quotients converge in expectation,
followed by a concentration argument.

\subsubsection{Power law graphs}
Starting again with the vertex set $[n]$, connect $i\neq j$ with probability
$\min(1,n^\beta (ij)^{-\alpha})$, where $0 < \alpha < 1$ and $0 \le \beta <
2\alpha$.  In other words, the expected degree distribution follows an
inverse power law with exponent $\alpha$, while the $n^\beta$ scaling factor
ensures that the probabilities do not become too small.  If $\beta >
2\alpha-1$, then the expected number of edges is superlinear, and a similar
argument to the one used in the above two examples shows convergence, this
time to a graphon that is not in $L^\infty$, namely $W(x,y) =
(1-\alpha)^2(xy)^{-\alpha}$.

\subsubsection{$W$-random graphs}
Our fourth example provides a construction of a sequence $(G_n)_{n \ge 0}$ of simple graphs
that converge to a given graphon $W$ with non-negative entries $W(x,y)\geq 0$.
Normalizing $W$ so that $\int_{[0,1]^2} W=1$ and fixing a target
density $\rho_n$, we proceed by first choosing $n$ i.i.d.\ variables $x_1,\dots x_n$ uniformly
in $[0,1]$, and then defining a random graph $G_n(W,\rho_n)$ on $\{1,\dots,n\}$ by
connecting each pair $\{i,j\}\in {[n]\choose 2}$ independently with probability
$\min\{1,\rho_nW(x_i,y_i)\}$.  Assuming that $\rho_n \to 0$ and $n\rho_n\to\infty$, the
graphs $G_n$ converge to $W$ under the normalized cut metric with probability one, by Theorem~2.14
in \cite{part1}.  If $W$ is a step function, this is more or less
equivalent to the convergence of stochastic block models, while for general graphons $W$,
one can proceed by first approximating $W$ by a step function.

\subsubsection{Convergence of free energies without convergence of microcanonical free energies}
\label{sec:ex4} Our next example is a generalization of Example~6.3 from
\cite{dense2} to the sparse setting, and is based on the observation that for
an arbitrary sequence of graphs $G_n$, the free energies of $G_n$ and a
disjoint union of $G_n$ with itself are identical (this follows from the fact
that for two disjoint graphs $G$ and $G'$, the partition function on $G\cup
G'$ factors into that of $G$ times that of $G'$).  If we take $G_n$ to be
equal to $G_{n,p}$ if $n$ is odd, and equal to a disjoint union of two copies
of $G_{n,p}$ if $n$ is even, then we get convergence of the free energies. By
contrast, in the notions of convergence from
Theorem~\ref{thm:limit-expressions}, the odd subsequence converges to $W=1$,
while the even one converges to the block graphon $W'$ that is equal to $2$
on $[0,1/2]^2 \cup [1/2,1]^2$ and $0$ elsewhere.  In particular, the
min-bisection of the even subsequence converges to zero, while the
min-bisection of the odd sequence converges to $1/2$.  This shows that the
microcanonical ground state energies are not convergent, which implies that
the microcanonical free energies don't converge either.

\subsubsection{LD convergence without metric convergence}
\label{sec:ex5} This is Example~\ref{ex:1} from Section~\ref{sec:upper-reg},
consisting of a graph $G_n$ that is the disjoint union of a complete graph on
$c_n$ nodes with $n-c_n$ isolated nodes.  A random $q$-quotient is then
determined by how many elements of the clique there are in each part and how
many elements of the non-clique. Calling these numbers $b_1,\dots,b_q$ and
$a_1,\dots, a_q$, we have $b_1+\dots+b_q=c_n$ and $a_1+\dots+a_q=n-c_n$, and
this occurs with probability
\[
q^{-n}
{\binom{c_n}{b_1,\dots,b_q}}
{\binom{n-c_n}{a_1,\dots,a_q}}.
\]
Everything else is determined from this data: $\alpha_i = (a_i+b_i)/n$,
$\beta_{ij} = b_i b_j/(c_n(c_n-1))$ if $i \neq j$, and $\beta_{ii} =
b_i(b_i-1)/(c_n(c_n-1))$. If $c_n\in\N$ is such that $c_n \to \infty$ and
$c_n/n\to 0$ as $n\to\infty$, then in the rate function, the choice of
$b_1,\dots,b_q$  gets wiped out by the choice of $a_1,\dots,a_q$, leading to
 LD convergence with rate function
\[
I_q((\alpha,\beta)) = \log q + \sum_{i=1}^q \alpha_i \log \alpha_i
\]
as long as $\beta \in \R^{q \times q}$ satisfies $\beta_{ii} \ge 0$,
$\beta_{ij} = \sqrt{\beta_{ii} \beta_{jj}}$, and $\sum_i \sqrt{\beta_{ii}} =
1$ (while $I_q((\alpha,\beta)) = \infty$ otherwise).  On the other hand,
$(G_n)_{n \ge 0}$ is not a Cauchy sequence in the normalized cut metric
$\delcutn$ from \eqref{delcutn-def} and hence does not converge to any
graphon in metric (see the proof of Proposition~2.12(a) in \cite{part1}).

\subsubsection{Convergence of quotients without convergence of the microcanonical free energies}
We close our example section with an example from \cite{BCG} (Example~5 from
that paper) which shows that without the assumption of upper regularity,
convergence of quotients does not imply convergence of the microcanonical
free energies, and hence does not imply LD-convergence either.  Before
stating this example, we note that whenever $H_n$ is a sequence of regular
bipartite graphs, $c_n\to\infty$, and $G_n$ is the union of $c_n$ disjoint
copies of $H_n$, then the quotients of $G_n$ converge to the convex hull of
the quotients of a graph consisting of a single edge. To see why, consider a
map from the vertex set of $G_n$ into $[q]$. Since $G_n$ is regular, the
corresponding quotient does not change if we replace $G_n$ by a disjoint
union of $|E(G_n)|$ edges (and map each of the split vertices to the same
element of $[q]$ as its original vertex in $G_n$).  Thus, the quotient is in
the convex hull of the quotients of a single edge.  On the other hand, each
quotient of a single edge can be realized in the bipartite graph $H_n$,
showing that each quotient in the convex hull can be arbitrarily well
approximated in $G_n$ if $c_n\to\infty$.

To get a sequence $G_n$ without convergent microcanonical free energies we
specialize to the case where $H_n$ consists of a $4$-cycle when $n$ is even
and a $6$-cycle when $n$ is odd.  The free energies of $G_n$ are then equal
to the free energies of the $4$-cycle when $n$ is even  and those of the
$6$-cycle when $n$ is odd.  But it is easy to check that the $4$-cycle has
different free energies from  a $6$-cycle, implying that $G_n$ does not have
convergent free energies, and hence does not have convergent microcanonical
free energies either (an alternative proof was given in \cite{BCG}, where it
was used that $G_n$ does not converge in the sense of Benjamini and Schramm
\cite{BS}, which in turn is necessary for convergence of the free energies,
as proved in \cite{bounded}).

\section{Convergence without the assumption of upper regularity}
\label{sec:Gproofs}

In this section, we consider general sequences of weighted graphs $G_n$
without any additional assumptions (except that $G_n$ has at least one edge
with nonzero edge weight). We will prove Lemma~\ref{lem:FatoFetc},
Theorem~\ref{thm:quotients-to-E}, and Theorem~\ref{thm:LD->F-and-Sq}.

\subsection{Free energies and ground state energies}
\label{sec:FatoFetc-proof}

In this section, we prove Lemma~\ref{lem:FatoFetc}. We start with the proof
of (i).  To this end, we note that for all $\ba\in\PD_q$ we have the lower
bound
\[
Z(G_n,J,h)^{1/|V(G_n)|}\geq e^{\langle \ba, h \rangle-\eps \|h\|_1}{\big(Z_{G,J}^{({\ba,\eps})}\big)}^{1/|V(G_n)|},
\]
from which we conclude that
\[ \limsup_{n\to\infty}F(G_n,J,h)
\leq F_{\ba}(J)- \langle \ba, h \rangle.
\]
Since $\ba\in\PD_q$ was arbitrary, this gives
\[ \limsup_{n\to\infty}F(G_n,J,h)
\leq
\inf_{\ba\in \PD_q} \bigl(F_\ba(J)-\langle \ba, h \rangle\bigr).
\]

To get a matching lower bound, we use the fact that $\PD_q$ can be covered by
$\lceil 1/(2\eps)\rceil^{q}\leq \eps^{-q}$ cubes of the form $\prod_{i=1}^q
[a_i-\eps,a_i+\eps]$. Explicitly, let $\PD_q^{(\eps)}$ be the set of points
$\ba$ where each coordinate is an odd multiple of $\eps$. Then
\[
Z(G_n,J,h)^{1/|V(G_n)|}\leq
\eps^{-q/|V(G_n)|}\max_{\ba\in\PD_q^{(\eps)}}e^{\langle \ba, h \rangle+\eps \|h\|_1}
Z_{\ba,\eps}(G_n,J)^{1/|V(G_n)|},
\]
implying that
\begin{align*}
\liminf_{n\to\infty}F(G_n,J,h)
&\geq -\eps \|h\|_1
+\liminf_{n\to\infty}\min_{\ba\in\PD_q^{(\eps)}}\bigl(F_{\ba,\eps}(G_n,J)-\langle \ba, h \rangle\bigr)
\\
&=-\eps \|h\|_1
+\min_{\ba\in\PD_q^{(\eps)}}\liminf_{n\to\infty}\bigl(F_{\ba,\eps}(G_n,J)-\langle \ba, h \rangle\bigr)
\end{align*}
where in the second step, we used that the minimum is over a finite set. Let
$\eps_k$ be a sequence going to zero, let $\ba_k$ be the minimizer on the
right hand side, and assume (by taking a subsequence, if necessary) that
$\ba_k$ converges to some $\ba$. Let
$\tilde\eps_k=\eps_k+\|\ba-\ba_k\|_\infty$. Since
$F_{\ba_k,\eps_k}(G_n,J)\geq F_{\ba,\tilde\eps_k}(G_n,J)$,
\begin{align*}
\liminf_{n\to\infty}F(G_n,J,h)
\geq -\tilde\eps_k \|h\|_1
+\liminf_{n\to\infty}\bigl(F_{\ba,\tilde\eps_k}(G_n,J)-\langle \ba, h \rangle\bigr).
\end{align*}
Sending $k\to\infty$, we conclude that
\begin{align*}
\liminf_{n\to\infty}F(G_n,J,h)
&\geq \lim_{\eps\to 0}
\liminf_{n\to\infty}\bigl(F_{\ba,\tilde\eps}(G_n,J)-\langle \ba, h \rangle\bigr)
\\
&=
\bigl(F_\ba(J)-\langle \ba, h \rangle\bigr)\geq \min_{\ba\in \PD_q}\bigl(F_\ba(J)-\langle \ba, h \rangle\bigr)
\end{align*}
as desired.

The proof of (ii) starts from the observations that
\begin{align*}
E_{\ba',\eps}(G_n,h) -\langle \ba', h \rangle  + \eps \norm{h}_1
&\geq E(G_n,J,h)\\
&\geq
\min_{\ba\in\PD_q^{(\eps)}}\bigl(E_{\ba,\eps}(G_n,J)-\langle \ba, h \rangle
\bigr) - \eps \norm{h}_1
\end{align*}
for all $\ba'\in\PD$ and all $\eps>0$.  Using these two bounds, the proof of
(ii) is now identical to the proof of (i).

To prove (iii) and (iv), we note that the number of terms in \eqref{Z(G)-def}
and \eqref{Z(G-ba)-def} is at most $q^{|V(G)|}$, implying that
\[
F(G,J,h)\leq E(G,J,h)\leq F(G,J,h) +\log q
\]
and
\[
F_{\ba,\eps}(G,J)\leq E_{\ba,\eps}(G,J)\leq F_{\ba,\eps}(G,J) +\log q.
\]
Rescaling $J$ and $h$ by a factor $\lambda \to \infty$, and using that both
the energies and microcanonical energies are linear in $\lambda$, we obtain
the claimed implications. \hfill\qed

\subsection{Convergence of quotients implies convergence of microcanonical ground state energies}
\label{sec:proof-quotients-to-E} In this section, we prove
Theorem~\ref{thm:quotients-to-E}.

To this end, we use \eqref{eq:E-phi-quotient} to express the microcanonical ground
state energies  as
\[
E_{\ba,\eps}(G,J)=- \max_{(\alpha,\beta)\in \SS_{\ba,\eps}(G)} \langle \beta,J\rangle.
\]
where
\[
\SS_{\ba,\eps}(G)=\SS_q(G)\cap\SS_{\ba,\eps}
\quad\text{with}\quad\SS_{\ba,\eps}=\left\{(\alpha,\beta)\in\SS_q : \|\alpha-\ba\|_\infty\leq
  \eps\right\}.
\]

\begin{proof}[Proof of Theorem~\ref{thm:quotients-to-E}]
In view of Lemma~\ref{lem:FatoFetc} it is enough to prove convergence of the
microcanonical ground state energies.

Let $\eps>0$. Since $\SS_q(G_n)$ is
assumed to converge to $\SS_q^{\infty}$, we can find an $n_0\in\N$ such that
\[
d_1^\Haus(\SS_q(G_n),\SS_q^\infty)\leq \eps\quad\text{for all} \quad n\geq n_0.
\]
For $n\geq n_0$, choose
$(\alpha^{(n)},\beta^{(n)})\in\SS_{\ba,\eps}(G_n)\subseteq \SS_q(G_n)$ such
that
\[
E_{\ba,\eps}(G_n,J)=-\langle \beta^{(n)},J\rangle,
\]
and choose $(\tilde\alpha^{(n)},\tilde\beta^{(n)}) \in\SS_q^\infty$ such that
$d_1((\tilde\alpha^{(n)},\tilde\beta^{(n)}),(\alpha^{(n)},\beta^{(n)}))\leq\eps$.
Then
\[E_{\ba,\eps}(G_n,J)\geq-\langle \tilde\beta^{(n)},J\rangle -\eps\|J\|_\infty.
\]
Since $|\tilde\alpha_i^{(n)}-a_i|\leq |\alpha_i^{(n)}-a_i|+
d_1((\tilde\alpha^{(n)},\tilde\beta^{(n)}),(\alpha^{(n)},\beta^{(n)})) \leq 2
\eps$, we have that
$(\tilde\alpha^{(n)},\tilde\beta^{(n)})\in\SS_{\ba,2\eps}$, proving in
particular that
\[
E_{\ba,\eps}(G_n,J)\geq-\langle \tilde\beta^{(n)},J\rangle -\eps\|J\|_1
\geq
 -
\sup_{ (\alpha,\beta)\in\SS_q^\infty \cap\SS_{\ba,2 \eps}}\langle \beta,J\rangle -\eps\|J\|_\infty.
\]
Taking first the $\liminf$ as $n\to\infty$ and then the limit $\eps\to 0$,
this shows that
\[
\lim_{\eps\to 0}\liminf_{n\to\infty}E_{\ba,\eps}(G_n,J)\geq-\lim_{\eps\to 0}
\sup_{ (\alpha,\beta)\in\SS_q^\infty\cap\SS_{\ba, \eps}}\langle
\beta,J\rangle
= - \max_{(\alpha,\beta) \in \SS_q^\infty
  \cap \SS_\ba} \langle \beta, J \rangle,
\]
where the final step is due to compactness. The proof of the matching upper
bound
\[
\lim_{\eps\to 0}\limsup_{n\to\infty}E_{\ba,\eps}(G_n,J)\leq-\lim_{\eps\to 0}
\sup_{ (\alpha,\beta)\in\SS_q^\infty\cap\SS_{\ba,\eps}}\langle
\beta,J\rangle
= - \max_{(\alpha,\beta) \in \SS_q^\infty
  \cap \SS_\ba} \langle \beta, J \rangle,
\]
proceeds along the same lines, now using that for any $(\alpha,\beta)\in
\SS_q^\infty$ with $\| \alpha - \ba\|_\infty\leq\eps$ we can find
$(\alpha^{(n)},\beta^{(n)})\in\SS_{\ba,2\eps}(G_n)$ with
$d_1((\alpha,\beta),(\alpha^{(n)},\beta^{(n)}))\leq\eps$.
\end{proof}

\subsection{LD convergence implies convergence of quotients}
\label{sec:proof-LD->Sq}

In this section, we prove part (i) of Theorem~\ref{thm:LD->F-and-Sq}, which
is statement (iii) of the following lemma.

\begin{lemma}
\label{lem:LD_prelim} Let $q\in \N$, assume that $(G_n)_{n \ge 0}$ is a
$q$-LD convergent sequence of weighted graphs with rate function $I_q$, and
let $\SS_q(I_q)=\{(\alpha,\beta)\in\SS_q : I_q((\alpha,\beta))<\infty\}$.
Then the following are true:
\begin{enumerate}
\item[(i)] The set $\SS_q(I_q)$ is closed with respect to the metric $d_1$.

\item[(ii)] The set $\SS_q(I_q)$ is equal to the set
    $\SS_q^\infty=\bigl\{(\alpha,\beta) :
    d_1\Bigl((\alpha,\beta),\SS_q(G_n)\Bigr)\to 0\bigr\}$.

\item[(iii)]  $\SS_q(G_n)$ converges to $\SS_q(I_q)$ in the Hausdorff
    distance.
\end{enumerate}
\end{lemma}

\begin{proof}
(i) For each $a\in\R$, the set $\{(\alpha,\beta)\in\SS_q :
I_q((\alpha,\beta))\leq a\}$ is closed by the lower semicontinuity of $I_q$.
To prove closedness of the set $\SS_q(I_q)$, we observe that
\[
I_{q,\eps,n}((\alpha,\beta))=-\frac{\log
\PP_{q,G_n}\bigl[d_1((\alpha,\beta),G_n/\phi)\leq \eps\bigr]}{|V(G_n)|}
\]
takes values in $[0,\log q]\cup\{\infty\}$, which in turn implies that $I_q$
takes values in $[0,\log q]\cup\{\infty\}$ and shows that
$\SS_q(I_q)=\{(\alpha,\beta) : I_q((\alpha,\beta))\leq \log q\}$.

(ii) Let us first assume that $(\alpha,\beta)\in \SS_q(I_q)$. Then
\[
\limsup_{n\to\infty} I_{q,\eps,n}((\alpha,\beta))\leq I_q((\alpha,\beta)) \leq \log q \quad\text{for all}\quad \eps>0,
\]
because $ I_{q,\eps,n}((\alpha,\beta))$ is non-increasing in $\eps$. Since
$I_{q,\eps,n}$ takes values in $[0,\log q]\cup\{\infty\}$, this implies that
for all $\eps>0$ we can find an $n_0$ such that
\[ I_{q,\eps,n}((\alpha,\beta))\leq \log q\quad\text{if}\quad n\geq n_0,
\]
which in turn implies that
\[d_1\Bigl((\alpha,\beta),\SS_q(G_n)\Bigr)\leq \eps\quad\text{if}\quad n\geq n_0.
\]
This proves that $(\alpha,\beta)\in \SS_q(I_q)$ implies
$d_1\Bigl((\alpha,\beta),\SS_q(G_n)\Bigr)\to 0$.

Assume on the other hand that $d_1\Bigl((\alpha,\beta),\SS_q(G_n)\Bigr)\to
0$, and by contradiction, assume further that
$(\alpha,\beta)\notin\SS_q(I_q)$, i.e., assume that
$I_q((\alpha,\beta))=\infty$. Since $I_{q,\eps,n}((\alpha,\beta))$ takes
values in $[0,\log q]\cup\{\infty\}$, this implies that there exists an
$\eps>0$ such that
\[
\liminf_{n\to\infty} I_{q,\eps,n}((\alpha,\beta))=\infty.
\]
which in turn implies that there exists an $n_0<\infty$ such that $
I_{q,\eps,n}((\alpha,\beta))=\infty$ for all $n\geq n_0$. As a consequence,
\[d_1\Bigl((\alpha,\beta),\SS_q(G_n)\Bigr) >\eps\quad\text{if}\quad n\geq n_0,
\]
contradicting the assumption that
$d_1\Bigl((\alpha,\beta),\SS_q(G_n)\Bigr)\to 0$.

(iii) Using the fact that $\SS_q(I_q)$ is compact, one easily transforms the
statement that $d_1\Bigl((\alpha,\beta),\SS_q(G_n)\Bigr)\to 0$ for all
$(\alpha,\beta)\in \SS_q(I_q)$ into the uniform statement that
\[
\sup_{(\alpha,\beta)\in \SS_q(I_q)}d_1\Bigl((\alpha,\beta),\SS_q(G_n)\Bigr)\to 0.
\]
To prove convergence in the Hausdorff distance we have to prove the matching
bound
\[
\sup_{(\alpha,\beta)\in \SS_q(G_n)}d_1\Bigl((\alpha,\beta),\SS_q(I_q)\Bigr)\to 0.
\]
Fix $\eps>0$, and let $S$ be the set $S=\{(\alpha,\beta)\in \SS_q :
d_1((\alpha,\beta),\SS_q(I_q))\geq \eps \}$.  Since $S$ is closed, we may use
\eqref{LD-conv-def} to conclude that
\[
\limsup_{n\to\infty}\frac{\log \PP_{q,G_n}[G_n/\phi\in S]}{|V(G_n)|}
\leq -\inf_{(\alpha,\beta)\in S} I_q((\alpha,\beta))= -\infty.
\]
Since the probability on the left hand side takes values in
$\{0\}\cup[q^{-|V(G_n)|},1]$ this shows there must exists an
$n_0=n_0(\eps,q)$ such that $ \PP_{q,G_n}[G_n/\phi\in S]=0 $ if $n\geq n_0$,
showing that $\SS_q(G_n)\cap S=\emptyset$ when $n\geq n_0$.  Expressed
differently, for all $\eps>0$ we can find an $n_0$ such that for $n\geq n_0$,
\[
\SS_q(G_n)\subseteq\{(\alpha,\beta)\in\SS_q : d_1((\alpha,\beta),\SS_q(I_q))<\eps\}.
\]
Or still expressed differently, we can find a sequence $\eps_n\to 0$ as
$n\to\infty$ such that
\[
d_1((\alpha,\beta),\SS_q(I_q))\leq \eps_n \quad\text{for all}\quad (\alpha,\beta)\in \SS_q(G_n). \qedhere
\]
\end{proof}

\subsection{LD convergence implies convergence of free energies}
\label{sec:proof-LD->Fa}

In this section, we prove part (ii) of Theorem~\ref{thm:LD->F-and-Sq}.

(ii) Given $\delta,\eps>0$, chose an arbitrary  $(\alpha,\beta)\in
\SS_{\ba,\delta}$, and let $\Omega_{(\alpha,\beta),\eps}$ be the set of
configurations $\phi\colon V(G_n)\to [q]$ such that
$d_1((\alpha,\beta),G_n/\phi)\leq \eps$. If
$\phi\in\Omega_{(\alpha,\beta),\eps}$, then $|a_i-\alpha_i(G_n/\phi)|\leq
\eps+\delta$, implying that $\Omega_{(\alpha,\beta),\eps}\subseteq
\Omega_{\ba,\eps+\delta}(G_n)$. Using further that
$d_1((\alpha,\beta),G_n/\phi)\leq \eps$ implies that $|\langle \beta,J\rangle
-\langle \beta(G_n/\phi),J\rangle|\leq \eps\|J\|_\infty$, and we then bound
\[
\begin{aligned} 
q^{|V(G_n)|}e^{\langle \beta,J\rangle |V(G_n)|}
&\PP_{q,G_n}\bigl[d_1((\alpha,\beta),G_n/\phi)\leq \eps\bigr]\\
&\leq
\sum_{\phi\in\Omega_{\ba,\eps+\delta}(G_n)}
e^{\bigl(\langle \beta(G_n/\phi),J\rangle+ \eps\|J\|_\infty\bigr)|V(G_n)|}
\\ & =
Z_{G_n,J}^{({\ba,\eps+\delta})} e^{\eps\|J\|_\infty |V(G_n)|},
\end{aligned}
\]
where the last step follows from the definition \eqref{Z(G-ba)-def} and the
fact that $-E_\phi(G_n,J)$ can be expressed as $\langle
\beta(G_n/\phi),J\rangle$. Using \eqref{LD-conv-alt} plus monotonicity in
$\eps$ to guarantee the existence of the limit $\eps\to 0$, this implies that
\[\lim_{\eps\to 0}\liminf_{n\to\infty}
\frac{\log Z_{G_n,J}^{({\ba,\eps+\delta})} }{|V(G_n)|}
\geq\log q+ \langle \beta ,J\rangle - I_q((\alpha,\beta)).
\]
Since $\delta>0$ and $(\alpha,\beta)\in\SS_{\ba,\delta}$ were arbitrary, this
shows that
\[
\begin{aligned}
\lim_{\eps\to 0}\liminf_{n\to\infty}
\frac{\log Z_{G_n,J}^{({\ba,\eps})} }{|V(G_n)|}
&\geq
\lim_{\delta\to 0}\sup_{(\alpha,\beta)\in\SS_{\ba,\delta}}
\Bigl(\log q+\langle \beta,J\rangle- I_q((\alpha,\beta))\Bigr)
\\
&\geq
\sup_{(\alpha,\beta)\in\SS_{\ba}}
\Bigl(\log q+\langle \beta,J\rangle- I_q((\alpha,\beta))\Bigr)
= -F_\ba(I_q,J).
\end{aligned}
\]

To get a matching upper bound we again fix $\ba$ and $\eps,\delta>0$. Since
$\SS_{\ba,\eps}$ is closed and hence compact, we can find a finite set
$S_{\delta}\subset\SS_{\ba,\eps}$ such that
$d_1^\Haus(S_\delta,\SS_{\ba,\eps})\leq \delta$. For $s\in S_{\delta}$, let
$B_{\delta}(s)$ be the set of pairs $(\alpha,\beta)\in \SS_{\ba,\eps}$ such
that $d_1(s,(\alpha,\beta))\leq \delta$.  Then $\SS_{\ba,\eps}=\bigcup_{s\in
S_{\delta}} B_{\delta}(s)$.  As a consequence,
\[
\begin{aligned}
Z_{G_n,J}^{({\ba,\eps})}
&=\sum_{\phi:V(G_n)\to [q]}
e^{\langle G_n/\phi,J\rangle |V(G_n)|}\one_{G_n/\phi\in\SS_{\ba,\eps}}
\\
&\leq q^{|V(G_n)|}
\sum_{s\in S_\delta}
\Biggl(\sup_{(\alpha,\beta)\in B_{\delta}(s)}e^{\langle\beta/\phi,J\rangle |V(G_n)|}
\PP_{q,G_n}\bigl[G_n/\phi\in B_{\delta}(s)\bigr]\Biggr).
\end{aligned}
\]
Since $S_{\delta}$ is finite and does not depend on $n$, we find that
\[
\begin{aligned}
\limsup_{n\to\infty}&\frac{ \log Z_{G_n,J}^{({\ba,\eps})}}{|V(G_n)|}
\\
&\leq\log q
+
\limsup_{n\to\infty}\max_{s\in S_{\delta}}
\left(\sup_{(\alpha,\beta)\in B_\delta(s)}\langle \beta,J\rangle
+
\frac{\log\PP_{q,G_n}\bigl[G_n/\phi\in B_\delta(s)]}{|V(G_n)|}\right)
\\
&=\log q+
\max_{s\in S_\delta}
\left(\sup_{(\alpha,\beta)\in B_\delta(s)}
\langle \beta,J\rangle
+\limsup_{n\to\infty}\frac{\log\PP_{q,G_n}\bigl[G_n/\phi\in B_\delta(s)]}{|V(G_n)|}\right)
\\
&\leq
\log q+
\max_{s\in S_\delta}
\left(\sup_{(\alpha,\beta)\in B_\delta(s)}
\langle \beta,J\rangle
-\inf_{(\alpha,\beta)\in B_\delta(s)}I_q((\alpha,\beta))\right),
\\
\end{aligned}
\]
where we used \eqref{LD-conv-def} in the last step.

Since $ \sup_{(\alpha,\beta)\in B_\delta(s)}\langle \beta, J\rangle\leq \langle
\beta',J\rangle +2 \|J\|_\infty\delta$ for all $(\alpha',\beta')\in B_\delta(s)$, we conclude
that
\[
\begin{aligned}
\limsup_{n\to\infty}\frac{ \log Z_{G_n,J}^{({\ba,\eps})}}{|V(G_n)|}
&\leq
\log q+\max_{s\in S_\delta}
\sup_{(\alpha,\beta)\in B_\delta(s)}
\Bigl(\langle \beta,J\rangle- I_q((\alpha,\beta))\Bigr)
+2\delta\|J\|_\infty
\\
&=
\log q+
\sup_{(\alpha,\beta)\in \SS_{\ba,\eps}}
\Bigl(\langle \beta,J\rangle- I_q((\alpha,\beta))\Bigr)
+2\delta\|J\|_\infty.
\end{aligned}
\]
Sending $\delta\to 0$ and using the fact that $Z_{G_n,J}^{({\ba,\eps})}$ is monotone in
$\eps$, this gives
\[
\lim_{\eps\to 0}
\limsup_{n\to\infty}\frac{\log Z_{G_n,J}^{({\ba,\eps})}}{|V(G_n)|}
\leq
\sup_{(\alpha,\beta)\in\SS_{\ba,\eps}}
\Bigl(\log q+\langle \beta,J\rangle- I_q((\alpha,\beta))\Bigr).
\]
Choose an arbitrary sequence $\eps_k$ going to zero, and choose
$(\alpha_k,\beta_k)\in \SS_{\ba,\eps}$ such that the supremum on the right
hand side is bounded by $\log q+\langle \beta_k,J\rangle-
I_q((\alpha_k,\beta_k))+\eps_k$. Going to a subsequence if needed, assume
that $(\alpha_k,\beta_k)$ converges to some $(\alpha,\beta)\in \SS_\ba$ in
the $d_1$ distance. Then
\[
\begin{aligned}
\lim_{\eps\to 0}
\limsup_{n\to\infty}\frac{\log Z_{G_n,J}^{({\ba,\eps})}}{|V(G_n)|}
&\leq
\log q+\langle \beta,J\rangle- \liminf_{k\to\infty}I_q((\alpha_k,\beta_k))
\\
&\leq
\log q+\langle \beta,J\rangle- I_q((\alpha,\beta))
\end{aligned}
\]
where in the last step we used that $I_q$ is lower semi-continuous.  Since
$(\alpha,\beta)\in\SS_\ba$, the right hand side is bounded by
\[\sup_{(\alpha,\beta)\in\SS_{\ba}}
\Bigl(\log q+\langle \beta,J\rangle- I_q((\alpha,\beta))\Bigr)
= -F_\ba(I_q,J),
\]
as desired. \qed

\section{Convergent sequences of graphons}
\label{sec:W-sequences}

In this section, we formulate and prove our main results in the language of
graphons. Several of these results are generalizations of the corresponding
results for $L^\infty$ graphons proved in \cite{dense2}; the exceptions are
those involving LD convergence, which was not considered in \cite{dense2}. It
turns out, however, that most of our proofs are quite different from those of
\cite{dense2}, most notably the proof that convergence of ground state
energies implies convergence in metric (Section~\ref{sec:Proof-E2Dist}),
which involves some new ideas not present in \cite{dense2} such as the use of
rearrangement inequalities.

\subsection{Upper regularity for graphons}

First we review the notion of upper regularity for graphons from
\cite{part1}.

Given a graphon $W$ and a partition $\PP = (Y_1, \dots, Y_m)$ of the interval
$[0,1]$ into finitely many measurable sets, we define $W_\PP$ to be the step
function whose value on $Y_i \times Y_j$ equals to the average of $W$ over
$Y_i \times Y_j$, i.e.,
\[
W_\PP = \frac{1}{\lambda(Y_i)\lambda(Y_j)}
\int_{Y_i\times Y_j} W(x,y)\,dx\,dy \quad \text{on $Y_i\times Y_j$},
\]
where $\lambda$ denotes the Lebesgue measure. An easy fact is that $W\mapsto
W_\PP$ is contractive with respect to the $L^p$ norms $\|\cdot\|_p$ and the
cut norm $\|\cdot\|_\square$, i.e.,
\begin{equation}\label{CONTRACT}
 \|W_\PP\|_\square\le\|W\|_\square
 \qquad\text{and}\qquad
 \|W_\PP\|_p\le \|W\|_p
\qquad\text{for all $p\geq 1$}.
\end{equation}
Another standard fact is that up to a factor of $2$, $W_\PP$ is the best step
function approximation to $W$ with steps in $\PP$, in the sense that
\[
\|W-W_\PP\|_\square\leq 2\|W-U_\PP\|_\square
\]
for all graphons $U$.  To see why, note that
\[
\begin{aligned}
\|W-W_\PP\|_\square &\leq \|W-U_\PP\|_\square + \|U_\PP-W_\PP\|_\square\\
&= \|W-U_\PP\|_\square + \|(W-U_\PP)_\PP\|_\square\\
&\leq 2\|W-U_\PP\|_\square.
\end{aligned}
\]

\begin{definition} \label{def:W-upp-reg}
Let $K \colon (0,\infty)\to (0,\infty)$ be any function. We say that a
graphon $W$ has \emph{$K$-bounded tails} if for each $\eps>0$,
\begin{equation}
\label{K-bd-tails}
\int_{[0,1]^2} |W(x,y)| \one_{|W(x,y)| \ge K(\varepsilon)}\,dx\,dy \leq \eps.
\end{equation}
A graphon $W$ is \emph{$(K,\eta)$-upper regular} if $W_\PP$ has $K$-bounded
tails whenever $\PP$ is a partition of the interval $[0,1]$ into sets of
measure at least $\eta$. A sequence $(W_n)_{n \ge 0}$ of graphons is
\emph{uniformly upper regular} if there exist $K \colon (0,\infty) \to
(0,\infty)$ and a sequence $\eta_n \to 0$ such that $W_n$ is
$(K,\eta_n)$-upper regular for all $n$.
\end{definition}

A key result from \cite{part1} is that every uniformly upper regular sequence
of graphons contains a subsequent that converges in cut distance to some
graphon. This is stated below.

\begin{theorem}[Theorem~C.7 in \cite{part1}]
\label{thm:UI-to-subsequence} If $(W_n)_{n \ge 0}$ is a sequence of uniformly
upper regular graphons, then there exists a graphon $W$ and a subsequence
$(W'_n)_{n \ge 0}$ of $(W_n)_{n \ge 0}$ such that $\delta_\square(W_n',W)\to
0$.
\end{theorem}

\subsection{Equivalent notions of convergence for graphons}
\label{sec:Graphon-Results}

The main theorem of this section, Theorem~\ref{thm:W-conv-equiv} below, is
the analogue of the first four statements of Theorem~\ref{thm:G-conv-equiv}.
To state it, we need the analogue of the microcanonical ground state energies
and microcanonical free energies define in \eqref{Ea(G)-def} and
\eqref{Fa(G)-def}, namely the quantities
\[
\EE_{\ba,\eps}(W,J)=\inf_{\substack{\rho\in\FP_q\\ \|\alpha(\rho)-\ba\|_\infty\leq\eps}}
\EE_\rho(W,J),
\]
and
\begin{equation}
\label{FaepsW-def}
\FF_{\ba,\eps}(W,J)=\inf_{\substack{\rho\in\FP_q\\ \|\alpha(\rho)-\ba\|_\infty\leq\eps}}
\biggl(
\EE_\rho(W,J)-\Ent(\rho)
\biggr).
\end{equation}

The following theorem is the main theorem of this section, and will be proved
in Sections~\ref{sec:Wpf-of-i>ii>iii}, \ref{sec:Proof-E2Dist}, and
\ref{sec:WProof-DtoF} below.

\begin{theorem}
\label{thm:W-conv-equiv} Let $(W_n)_{n \ge 0}$ be a sequence of uniformly
upper regular graphons. Then the following statements are equivalent:
\begin{enumerate}
\item[(i)] $(W_n)_{n \ge 0}$ is a Cauchy sequence in the cut metric
    $\delta_\square$.

\item[(ii)] For every $q \in \N$, the sequence $(\iv\SS_q(W_n))_{n \ge 0}$
    is a Cauchy sequence under the Hausdorff distance $d_1^\Haus$.

  \item[(iii)] For every $q\in\N$, $\ba\in\PD_q $, and symmetric matrix $J
      \in \R^{q \times q}$,
    \[
    \lim_{\eps\to 0}\liminf_{n\to\infty}
    \EE_{\ba,\eps}(W_n,J)
    =
    \lim_{\eps\to 0}\limsup_{n\to\infty}
    \EE_{\ba,\eps}(W_n,J).
    \]

  \item[(iv)] For every $q\in\N$, $\ba\in\PD_q $, and symmetric matrix $J
      \in \R^{q \times q}$,
    \[
    \lim_{\eps\to 0}\liminf_{n\to\infty}
    \FF_{\ba,\eps}(W_n,J)
    =
    \lim_{\eps\to 0}\limsup_{n\to\infty}
    \FF_{\ba,\eps}(W_n,J).
\]
\end{enumerate}
\end{theorem}

\begin{remark}
\label{rem:W-equiv-conv} \leavevmode
\begin{enumerate}
\item[(i)] We will prove the theorem by showing that (i) $\Rightarrow$ (ii)
    $\Rightarrow$ (iii) $\Rightarrow$ (i) and that (i) $\Rightarrow$ (iv) $\Rightarrow$
    (iii).  It turns out that the assumption of uniform upper regularity is only needed for
    the proof that (iii) $\Rightarrow$ (i). All other implications hold for arbitrary
    sequences of graphons.
\item[(ii)] Under the assumption (ii) of the above theorem, $\iv\SS_q(W_n)$ converges
    to the compact set\footnote{The nonempty compact subsets of $\iv\SS_q$ form a complete
    metric space under $d_1^\Haus$ (see \cite{GBP}).} $\SS_q^\infty=\{(\alpha,\beta)\colon
    d_1((\alpha,\beta),\iv\SS_q(W_n))\to 0\}$. Proceeding as in the proof of
    Theorem~\ref{thm:quotients-to-E}, this in turn implies that (iii) holds with the limit
    given as
 \[
E_\ba(J)=-\max_{\substack{(\alpha,\beta)\in\SS_q^\infty\\ \alpha=\ba}}\langle \beta,J\rangle,
\]
again without the assumption of uniform upper regularity.
\item[(iii)] Under the assumption (i) of the above theorem, the sequences
    $(\EE_\ba(W_n,J))_{n \ge 0}$ and $(\FF_\ba(W_n,J))_{n \ge 0}$ are convergent for
    all $q, \ba, J$. (In particular, the use of $\eps$ in the theorem statement is just for
    comparison with the case of graphs, and not because it is truly needed.) Finally,
    under the assumption of uniform upper regularity, each of these two statements is
    not only necessary but also sufficient for convergence in metric to hold, as we will
    show in Sections~\ref{sec:Proof-E2Dist} and \ref{sec:WProof-DtoF}.
\end{enumerate}
\end{remark}

\subsection{Compactness of quotient space}
\label{sec:compact-quotient}

Before jumping into the proof of Theorem~\ref{thm:W-conv-equiv}, we prove
some compactness results about quotients of $W$, thereby shedding light on
the quantities $\EE_\ba(W, J)$, $\FF_\ba(W, J)$, and $I_q((\alpha,\beta),W)$.

Recall the $\ell_1$ distance $d_1$ from \eqref{d1-general} as well as the
definitions of fractional graphon quotients from
Section~\ref{sec:limit-expressions}.

\begin{prop}\label{prop:S-Closed}
Let $W$ be a graphon, let $q\in\N$, and let $\ba\in\PD_q$. Then $\iv\SS_q(W)$
and $\iv\SS_\ba(W)$ are compact under the metric $d_1$.
\end{prop}

We will prove this proposition after we develop a few preliminaries.

In \eqref{EaW-def}--\eqref{IW-def}, $\EE_\ba(W, J)$, $\FF_\ba(W, J)$, and
$I_q((\alpha,\beta),W)$ were originally defined as infima over some subset of
fractional partitions. We will see that the infima are attained by some
fractional partitions, so that the ``inf'' can be replaced by ``min''.

Furthermore, we will see that
\begin{align}
\label{eq:E-W-eps->0} \EE_\ba(W, J) &= \lim_{\eps \to 0} \EE_{\ba,\eps}(W,J), \\
\label{eq:F-W-eps->0} \FF_\ba(W, J) &= \lim_{\eps \to 0} \FF_{\ba,\eps}(W,J), \qquad \text{and} \\
\label{eq:I-W-eps->0}  I_q((\alpha,\beta), W) &= \lim_{\eps \to 0} \inf_{\substack{\rho
    \in \FP_q \\ d_1(W/\rho, (\alpha,\beta)) \leq \eps}} (\log q - \Ent(\rho)).
\end{align}
The quantities $\EE_\ba(W,J)$ and $\FF_\ba(W,J)$ are both continuous with
respect to $\ba, W, J$. On the other hand, $I_q((\alpha,\beta), W)$ is lower
semicontinuous in its arguments
(Proposition~\ref{prop:I_q-lower-semicontinuous}), and it is not continuous
(it takes values in $[0, \log q]$ when $(\alpha,\beta) \in \iv\SS_q(W)$ and
is infinite otherwise).

It follows as an immediate corollary of Proposition~\ref{prop:S-Closed} that
\[
\EE(W,J,h)
= -\max_{(\alpha,\beta)\in\iv\SS_q(W)} \bigl(\langle \beta,J\rangle +
\langle \alpha, h \rangle \bigr)
\]
and
\begin{equation}\label{Ea-frac-def}
\EE_\ba(W,J)
= -\max_{(\alpha,\beta)\in\iv\SS_\ba(W)}\langle \beta,J\rangle,
\end{equation}
since $(\alpha,\beta)\mapsto \langle \beta,J\rangle$ and
$(\alpha,\beta)\mapsto \langle \alpha,h \rangle$ are continuous in the $d_1$
metric. This gives an alternate representation of ground state energies of
$W$ in terms of its quotients.

\subsubsection{Approximations by step functions}

One way to approximate a graphon $W$ by step functions is given by the
following lemma, which is an immediate consequence of the almost everywhere
differentiability of the integral function.

\begin{lemma}
\label{lem:as-approx} Let $p\geq 1$. For a positive integer $n$, let $\PP_n$
be the partition of $[0,1]$ into consecutive intervals of length $1/n$. If
$W$ is a graphon, then $W_{\PP_n}\to W$ almost everywhere. In addition,
$W_{\PP_n}\to W$ in $L^p$ whenever $W$ is an $L^p$ graphon.
\end{lemma}

\begin{proof}
Almost everywhere convergence follows the Lebesgue differentiation theorem.
To get convergence in $L^p$ we approximate $W$ by the bounded graphon
$W_M=W\one_{|W|\leq M}$, where $M > 0$. By triangle inequality and
\eqref{CONTRACT},
\[
\begin{aligned}
\|W-W_{\PP_n}\|_p&\leq \|W_M-(W_M)_{\PP_n}\|_p
+\|W-W_M\|_p+\|(W-W_M)_{\PP_n}\|_p
\\
&\leq \|W_M-(W_M)_{\PP_n}\|_p
+2\|W-W_M\|_p.
\end{aligned}
\]
The second term on the right can be made arbitrarily small by setting $M$ to
be sufficiently large, and for any fixed $M$, the first term on the right
goes to zero as $n\to \infty$. This shows that $\|W-W_{\PP_n}\|_p \to 0$.
\end{proof}

The lemma does not, however, give any information on the speed of
convergence. If instead of almost everywhere convergence we content ourselves
with convergence in the cut metric, the situation is different, as is well
known in the case of $L^2$ graphons, where one can apply the weak version of
the regularity lemma first established in \cite{FK}. This lemma can be
generalized to $L^p$ graphons for $p>1$ and more generally to any graphon
with $K$-bounded tails (see \cite{part1}), but we will not need this here,
where we use only the corresponding version for uniformly upper regular
sequences of graphs (see Theorem~\ref{thm:weak-regular-for-graphs} in
Section~\ref{sec:graphs->graphons}).

\subsubsection{Limits of fractional partitions}

We say that a sequence $\rho^{(n)} \in \FP_q$ of fractional partitions
converges to $\rho \in \FP_q$ \emph{over rational intervals} if
$\int_D\rho_i^{(n)}(x)\,dx\to \int_D\rho_i(x)\,dx$ for every interval
$D\subseteq [0,1]$ with rational endpoints.

\begin{lemma} \label{lem:FP-limit}
Fix $q \in \N$. Every sequence $\rho^{(n)} \in \FP_q$ of fractional
partitions contains a subsequence that converges to some $\rho \in \FP_q$
over rational intervals.
\end{lemma}

\begin{proof}
By restricting to a subsequence, we may assume that $\int_D\rho_i^{(n)}(x)
\,dx$ converges for every interval $D \subseteq [0,1]$ with rational
endpoints, and let us denote this limiting value by $\mu_i(D)$. Then, by the
extension theorem for measures (Proposition~5.5 in \cite{Knapp}), $\mu_i$ can
be extended to a measure on $[0,1]$ such that $\sum_{i \in
[q]}\mu_i(D)=\lambda(D)$ for every measurable $D\subseteq [0,1]$, where
$\lambda$ is the Lebesgue measure. It is easy to see that $\mu_i$ is
absolutely continuous with respect to $\lambda$. Defining $\rho_i$ to be the
density of $\mu_i$ with respect to $\lambda$ and changing $\rho_i$ on a set
of measure zero, we obtain the desired fractional partition
$\rho=(\rho_i)_{i\in [q]}$.
\end{proof}

\begin{lemma} \label{lem:FP-quotients-converge}
  If $\rho^{(n)} \in \FP_q$ converges to $\rho \in \FP_q$ over rational
  intervals and $\norm{W_n - W}_\square\to 0$, then $d_1(W_n/\rho^{(n)}, W/\rho)
  \to 0$.
\end{lemma}

\begin{proof}
  We have $\alpha_i(\rho^{(n)}) = \int_{[0,1]}\rho_i^{(n)} \to
  \int_{[0,1]} \rho_i = \alpha_i(\rho)$. Next we have
\begin{align*}
\abs{\beta_{ij}(W_n/\rho^{(n)})  - \beta_{ij}(W/\rho^{(n)})}
&=\abs{\int \rho_i^{(n)}(x)\rho_j^{(n)}(y) (W_n(x,y) - W(x,y)) \,dx\,dy}
\\
& \leq \norm{W_n - W}_\square.
\end{align*}
It remains to show that $\beta_{ij}(W/\rho^{(n)}) \to \beta_{ij}(W/\rho)$,
i.e.,
 \[
 \int \rho_i^{(n)}(x)\rho_j^{(n)}(y) W(x,y) \,dx\,dy \to
  \int \rho_i(x)\rho_j(y) W(x,y) \,dx\,dy,
 \]
 which follows from Lemma~\ref{lem:as-approx} as we can approximate $W$ arbitrarily well
in $L^1$ using step functions with rational steps, and $\rho^{(n)}$ converges
to $\rho$ over rational intervals.
\end{proof}

Now we can prove the compactness of the set of quotients.

\begin{proof}[Proof of Proposition~\ref{prop:S-Closed}]
Let $(W/\rho^{(n)})_{n\geq 1}$ be a sequence of quotients in $\iv\SS_q(W)$
(or $\iv\SS_\ba(W)$). By Lemma~\ref{lem:FP-limit} we can restrict to a
subsequence so that $\rho^{(n)}$ converges to some $\rho \in \FP_q$ over
rational intervals. By Lemma~\ref{lem:FP-quotients-converge}, we have
$d_1(W/\rho^{(n)},W/\rho)\to 0$, thereby proving that the space of quotients
is closed and hence compact.
\end{proof}

The claim \eqref{eq:E-W-eps->0} has a similar proof: we have $\EE_{\rho}(W,J)
= - \langle \beta(W/\rho), J \rangle$, so that $d_1(W/\rho^{(n)},W/\rho)\to
0$ implies $\EE_{\rho^{(n)}}(W,J) \to \EE_{\rho}(W,J)$.

\subsubsection{Entropy and lower semicontinuity}

Now we prove \eqref{eq:F-W-eps->0} and \eqref{eq:I-W-eps->0}, and furthermore
the claim that in the definitions \eqref{FaW-def} and \eqref{IW-def} for
$\FF_\ba(W,J)$ and $I_q((\alpha,\beta),W)$ the infimum is attained by some
fractional partition. In fact, they are all immediate consequences of
Lemma~\ref{lem:FP-limit} along with the following lemma.

\begin{lemma} \label{lem:Ent-convexity}
If $\rho^{(n)} \in \FP_q$ converges to $\rho \in \FP_q$ over rational
intervals, then
  \[
  \limsup_{n \to \infty} \Ent(\rho^{(n)})
  \leq \Ent(\rho).
  \]
\end{lemma}

\begin{proof}
For any positive integer $k$, let $\rho_{[k]} \in \FP_q$ (and similarly
$\rho^{(n)}_{[k]}$) denote the fractional partition obtained from
$\rho_{[k]}$ by averaging over the interval $[(j-1)/k, j/k)$ for each integer
$j \in [k]$. Specifically, we set the value of $\rho_{[k],i}$ on $[(j-1)/k,
j/k)$ to be $k \int_{(j-1)/k}^{j/k} \rho_{i}(x) \,dx$.

Since $-x \log x$ is concave, $\Ent(\rho^{(n)}) \leq \Ent(\rho^{(n)}_{[k]})$
by Jensen's inequality. For a fixed $k$ we have
\[
  \limsup_{n \to \infty} \Ent(\rho^{(n)})
  \leq  \limsup_{n \to \infty} \Ent(\rho^{(n)}_{[k]})
  = \Ent(\rho_{[k]})
\]
where the last equality follows from $\rho^{(n)}$ converging to $\rho$ on
rational intervals. Finally, we have $\rho_{[k],i} \to \rho_i$ almost
everywhere  as $k \to \infty$ by the Lebesgue differentiation theorem, and thus
$\Ent(\rho_{[k]}) \to \Ent(\rho)$ as $k \to \infty$ by the bounded
convergence theorem. This proves the lemma.
\end{proof}

\begin{prop} \label{prop:I_q-lower-semicontinuous}
Let $q \in \N$. The function $I_q((\alpha, \beta), W)$ is lower
semicontinuous (with the metric $d_1$ on the first argument and
$\delta_\square$ on the second).
\end{prop}

\begin{proof}
We need to show that if $(\alpha^{(n)},\beta^{(n)}) \to (\alpha,\beta)$ in
$d_1$ and $W_n \to W$ in $\delta_\square$ as $n \to \infty$, then
\[
  \liminf_{n \to \infty} I_q((\alpha^{(n)}, \beta^{(n)}),W_n) \geq I_q((\alpha,\beta),W).
\]
We may restrict to a subsequence so that
$I_q((\alpha^{(n)},\beta^{(n)}),W_n)$ converges to the original $\liminf$.
Since $I_q$ is invariant under measure preserving bijections for the graphon,
we may assume that $\norm{W_n - W}_\square \to 0$. The result is automatic if
the limit is infinity, so we might as well assume that
$I_q((\alpha^{(n)},\beta^{(n)}),W_n) < \infty$ (and hence at most $\log q$)
for all $n$, so that there is some $\rho^{(n)} \in \FP_q$ with $W/\rho^{(n)}
= (\alpha^{(n)},\beta^{(n)})$ and $I_q((\alpha^{(n)}, \beta^{(n)}),W_n) =
\log q - \Ent(\rho^{(n)})$. By Lemma~\ref{lem:FP-limit} we can further
restrict to a subsequence so that $\rho^{(n)}$ converges to some $\rho \in
\FP_q$ over rational intervals. By Lemma~\ref{lem:FP-quotients-converge} we
have $W/\rho = \lim_{n \to \infty} W_n/\rho^{(n)} = \lim_{n \to \infty}
(\alpha^{(n)},\beta^{(n)}) = (\alpha,\beta)$, so $I_q((\alpha,\beta),W) \leq
\log q - \Ent(\rho)$. By Lemma~\ref{lem:Ent-convexity},
\[
\begin{aligned}
  \liminf_{n \to \infty} I_q((\alpha^{(n)}, \beta^{(n)}),W_n)
  &= \liminf_{n\to\infty} (\log q - \Ent(\rho_n))
  \\
  &\geq \log q - \Ent(\rho)
  \geq I_q((\alpha,\beta),W),
  \end{aligned}
\]
as desired.
\end{proof}

\subsection{Proof of (i)$\Rightarrow$(ii)$\Rightarrow$(iii) in Theorem~\ref{thm:W-conv-equiv}}
\label{sec:Wpf-of-i>ii>iii}

The claim that (i) implies (ii) follows from Lemma~\ref{lem:U-W-Haus} below. The
claim that (ii) implies (iii)---with the limit expressed as described in
Remark~\ref{rem:W-equiv-conv}(ii)---is essentially
identical to the proof of Theorem~\ref{thm:quotients-to-E} from
Section~\ref{sec:proof-quotients-to-E} and is left to the reader.

\begin{lemma}\label{lem:U-W-Haus}
Let $q\in\N$, $U$ and $W$ be graphons, and $\rho \in \FP_q$. Then
$d_1(U/\rho,W/\rho)\leq q^2\|U-W\|_\square$ and hence
\begin{equation}
  \label{U2S3}
 d^\Haus_1(\iv\SS_{{q}}(U),\iv\SS_{{q}}(W))
  \le q^2\delta_\square(U,W).
\end{equation}
\end{lemma}

\begin{proof}
We have $\alpha(U/\rho) = \alpha(\rho) = \alpha(W/\rho)$. Also for $ i,j \in
[q]$ we have
\begin{equation} \label{eq:beta-diff-cut}
\begin{aligned}
  \abs{\beta_{ij}(U/\rho) - \beta_{ij}(W/\rho)}
  &= \abs{\int_{[0,1]^2} (U - W)(x,y) \rho_i(x) \rho_j(y) \,dx\,dy}\\
  &\leq \norm{U - W}_\square.
\end{aligned}
\end{equation}
Summing over all $i,j \in [q]$ gives $d_1(U/\rho,W/\rho) \leq q^2 \norm{U -
W}_\square$. The claim \eqref{U2S3} follows immediately.
\end{proof}

We next observe that microcanonical ground state energies are continuous in
the cut metric.  To state this result, we define the norm $\|J\|_1$ of a
matrix $J \in \R^{q \times q}$ to be $\sum_{i,j\in [q]}|J_{ij}|$.

\begin{prop}
\label{prop:U2E&Ea} Let $q\in\N$, $\ba\in\PD_q$, and $h\in \R^q$, and let $J
\in \R^{q \times q}$ be a symmetric matrix. If $U$ and $W$ are arbitrary
graphons, then
\begin{equation}
\label{U2Ea}
|\EE_\ba(U,J)-\EE_\ba(W,J)|\le\|J\|_1
\delta_\square(W,U).
\end{equation}
\end{prop}

\begin{proof}
Since the left side of this bound does not change if we replace $U$ by
$U^\phi$ for some measure preserving bijection $\phi\colon [0,1]\to[0,1]$, it
is enough to prove it in terms of $\|W-U\|_\square$ instead of
$\delta_\square(W,U)$. Let $\rho \in \FP_q$. Using \eqref{eq:beta-diff-cut}
we obtain
\begin{equation}
\label{Erho-Erho}
\begin{aligned}
\bigl|\EE_\rho(U,J,h)-\EE_\rho(W,J,h)\bigr| &= \Bigl|\sum_{i,j\in[q]}
(\beta_{ij}(U/\rho) - \beta_{ij}(W/\rho)) J_{ij}\Bigr|\\
& \leq \|J\|_1 \|U-W\|_\square,
\end{aligned}
\end{equation}
as desired.
\end{proof}

\subsection{Proof of (iii)$\Rightarrow$(i) in Theorem~\ref{thm:W-conv-equiv}}
\label{sec:Proof-E2Dist}

By the bound \eqref{U2Ea}, convergence in metric implies convergence of the
microcanonical ground state energies. To prove the converse, we will
establish the following proposition, one of the main results of this section.

\begin{theorem}
\label{thm:E-->delta} Let $W$ and $U$ be two  graphons. If
\begin{equation}
\label{Ea-Assumption}
\EE_\ba(U,J)=\EE_\ba(W,J)
\end{equation}
for all $q\in\N$, every symmetric matrix $J \in \R^{q \times q}$, and all
$\ba$ of the form
\begin{equation}
\label{aq}
\ba_q=(1/q,\dots,1/q),
\end{equation}
then $\delta_\square(W,U)=0$.
\end{theorem}

This theorem proves the implication (iii)$\Rightarrow$(i) in
Theorem~\ref{thm:W-conv-equiv}, as well as the fact that convergence of
$\EE_\ba(W_n,J)$ is sufficient for convergence in metric (see
Remark~\ref{rem:W-equiv-conv}(iii)).  Indeed, for the second of these
assertions, assume first that the ground state energies $\EE_\ba(W_n,J)$
converge for all $q\in\N$, all $\ba$ of the form \eqref{aq}, and all $J$,
while $W_n$ does not converge in the cut metric. Since $(W_n)_{n\geq 0}$ is
assumed to be uniformly upper regular, we may use
Theorem~\ref{thm:UI-to-subsequence} to find two subsequences $W'_n$ and
$W''_n$ of $W_n$ that converge to two graphons $W$ and $U$ in the cut
distance $\delta_\square$, while $\delta_\square(W,U)>0$. But convergence in
the cut distance implies convergence of the ground state energies by
\eqref{U2Ea}, which means that $U$ and $W$ have identical ground state
energies, a contradiction. The proof of (iii)$\Rightarrow$(i) in
Theorem~\ref{thm:W-conv-equiv} is similar, since convergence of $W_n'$ to $W$
in metric implies that $\iv\SS_q(W_n')\to\iv\SS_q(W)$ in the Hausdorff
distance, which in turn can easily be seen to give convergence of the
quantities in (iii) to $\EE_\ba(W,J)$ (the proof is the same as that of
Theorem~\ref{thm:quotients-to-E}), and similarly for the convergence along
the subsequence $W''_n$ to $\EE_\ba(U,J)$.

To prove Theorem~\ref{thm:E-->delta}, we will work with the quasi-inner
product
\[
\CC(W,Y)=\sup_{\phi} \E[ WY^\phi] ,
\]
where the supremum is taken over all measure preserving bijections
$\phi\colon [0,1]\to[0,1]$ and the expectation $\E[\cdot]$ is with respect to
the Lebesgue measure on $[0,1]^2$, i.e.,
\[
 \E[ WY^\phi] =
\int\limits_{[0,1]^2} W(x,y)Y^\phi(x,y)\,dx\,dy
=\int\limits_{[0,1]^2} W(x,y)Y(\phi(x),\phi(y))\,dx\,dy.
\]
This quantity was defined in \cite{dense2}, where it was assumed that both
$W$ and $Y$ are in $L^\infty$. But the definition makes sense in our more
general context, where we will assume that $W$ is an arbitrary graphon and
$Y$ is bounded.

\begin{lemma}
\label{lem:E-to-C} Let $W$ and $U$ be two  graphons such that
\eqref{Ea-Assumption} holds for all $q\in\N$ and all $\ba$ of the form
\eqref{aq}. Then
\begin{equation}
\label{C-equal}
\CC(W,Y)=\CC(U,Y)
\end{equation}
for all bounded graphons $Y$.
\end{lemma}
\begin{proof}
If $Y=W^H$ for a weighted graph $H$ on $q$ nodes, where $H$ has edge weights
$J\in\R^{q \times q}$ and vertex weights $\ba$ of the form \eqref{aq}, then
\[
-\EE_\mathbf{a}(W,J)=\CC(W,W^H),
\]
and the claim follows directly from the assumption \eqref{Ea-Assumption}. For
general $Y$, we use Lemma~\ref{lem:as-approx} to approximate $Y$ by step
functions. More explicitly, let $P_n$ be as in Lemma~\ref{lem:as-approx}, and
let $Y_n=Y_{P_n}$. Then $Y_n\to Y$ in $L^1$, and $\|Y_n\|_\infty\leq
\|Y\|_\infty$. Hence
\[
\begin{aligned}
\Bigl|\E[ W\, Y_n^\phi]- \E[ W \,Y^\phi]\Bigr|&\leq
2\|Y\|_\infty|\|W\one_{|W|\geq K}\|_1 +
\|W \one_{|W|\leq K}\|_\infty \|Y^\phi - Y_n^\phi\|_1
\\
&\leq
2\|Y\|_\infty|\|W\one_{|W|\geq K}\|_1 +
K \|Y - Y_n\|_1.
\end{aligned}
\]
The right side can be made as small as desired by first choosing $K$ large
enough and then $n$ large enough. Observing that the resulting convergence is
uniform in $\phi$, this implies that $\CC(W,Y_n)\to\CC(W,Y)$ and similarly
for $\CC(U,Y_n)$. From this, the claim follows.
\end{proof}

\begin{lemma}
\label{lem:C-to-distribution} Let $W$ and $U$ be two graphons such that
\eqref{C-equal} holds for all bounded graphons $Y$. Consider the real valued
random variables $\iv W=W(x,y)$ and $\iv U=U(x,y)$ where $x,y$ are chosen
independently uniformly at random from $[0,1]$. Then $\iv W$ and $\iv U$ have
the same distribution.
\end{lemma}

To prove this lemma, we will use some notions and results from the theory of
monotone rearrangement.

\subsubsection{Monotone rearrangements and proof of Lemma~\ref{lem:C-to-distribution}}

Throughout this section, we identify graphons $W$ with the real-valued random
variables $W(x,y)$ obtained by choosing $x,y$ independently uniformly at
random from $[0,1]$; if $W$ is such a random variable, we use $E[W]$ to
denote its expectation. For $s\in \R$ we use $\{W>s\}$ to denote the event
that $W>s$, namely $\{W>s\}=\{(x,y) : W(x,y)>s\}$, and $\Pr[W>s]$ to denote
the probability of this event.

For a graphon $W$, we define the monotone rearrangement as the function
\[
W^*(x_1,x_2)=\sup\{t\in \R : \Pr[W>t] > \|x\|_\infty^2\},
\]
where $\|x\|_\infty=\max\{x_1,x_2\}$.  Then $W^*(x_1,x_2)$ is a
weakly decreasing function of $\|x\|_\infty$, and it has the same distribution as $W$;
i.e.,
\[\Pr[W>t]=\Pr[W^*>t].\]
To see why, note that
\[
\begin{aligned}
\Pr[W^* > t_0] &= \Pr[\sup\{t\in \R : \Pr[W>t]> \|x\|_\infty^2\} > t_0]\\
&= \Pr[\text{there exists $t>t_0$ such that $\Pr[W>t] > \|x\|_\infty^2$}]\\
&= \Pr[\Pr[W>t_0] > \|x\|_\infty^2]\\
&= \Pr\Big[\text{$x_1 < \sqrt{\Pr[W>t_0]}$ and $x_2 < \sqrt{\Pr[W>t_0]}$}\Big]\\
&= \Pr[W>t_0],
\end{aligned}
\]
where the third line follows from the fact that $t\mapsto Pr[W>t]$ is right-continuous.

Define two graphons $W$ and $U$ to be \emph{aligned} if their level sets are
nested, in the sense that for all $s,t\in\R$, either
$\{W>s\}\subseteq\{U>t\}$ or $\{U>s\}\subseteq \{W>s\}$.  It is easy to see
that for any two graphons $U$, $W$, the monotone rearrangements $U^*$ and
$W^*$ are aligned.

Let $W$ be an $L^1$ graphon, and let $Y$ be a bounded graphon. Then we have
the \emph{rearrangement inequality}
\[
\E[W\,Y]\leq \E[W^*\,Y^*].
\]
See Appendix~\ref{app:rearr} for a proof. The proof also tells us that
\[
\E[W\,Y]= \E[W^*\,Y^*]
\]
if $Y$ and $W$ are aligned. Before delving into the proof of
Lemma~\ref{lem:C-to-distribution} we observe that
\begin{equation}
\label{C-Algin}
\E[ WY]\leq \CC(W,Y)\leq \E[W^*Y^*].
\end{equation}
The first inequality follows immediately from the definition of $\CC(W,Y)$,
and the second follows from the definition and the fact that $Y$ and $Y^\phi$
have the same distribution, which in turn implies that $Y^*=(Y^\phi)^*$.

\begin{proof}[Proof of Lemma~\ref{lem:C-to-distribution}]

Define $\text{top}_\lambda(W)\subseteq [0,1]^2$ in such a way that the
Lebesgue measure of $\text{top}_\lambda(W)$ is $\lambda$, and $W(u,v)\leq
\inf_{(x,y)\in {\text{top}}_\lambda(W)}W(x,y)$ whenever $(u,v)\notin
\text{top}_{\lambda}(W)$. Explicitly, let $M=\sup\{s\in \R : \Pr[W>s]\geq
\lambda\}$. If $\Pr\{W=M\}=0$, then we have that $\Pr\{W>M\}=\Pr\{W\geq
M\}=\lambda$, and we define $\text{top}_{\lambda}(W)=\{W>M\}$. Otherwise,
$\Pr\{W>M\}\leq \lambda\leq \Pr\{W\geq M\}$, in which case we chose
$\text{top}_{\lambda}(W)$ in such a way that
$\{W>M\}\subseteq\text{top}_{\lambda}(W)\subseteq\{W\geq M\}$ and
$\mu[\text{top}_\lambda(W)]=\lambda$, where $\mu$ denotes the Lebesgue
measure on $[0,1]^2$. In either case, we have $M= \inf_{(x,y)\in
\text{top}_{\lambda}(W)}W(x,y)$ and $W(u,v)\leq M$ whenever $(u,v)\notin
\text{top}_\lambda(W)$.

It is easy to see that $W$ and the indicator function
$\one_{\text{top}_\lambda (W)}$ are aligned, implying that
$\E\Bigl[W\one_{\text{top}_\lambda (W)}\bigr]=\E\Bigl[W^*\one_{
\text{top}_\lambda}(W)^*\Bigr]$. Consider now the $L^1$ graphon $Y=\one_{
\text{top}_\lambda (W)}$. With the help of \eqref{C-Algin} and the fact that
$\E[WY]=E[W^*Y^*]$ we have
\[
\E[WY]=\CC(W,Y)=\CC(U,Y)\leq \E[U^* Y^*].
\]
Let $\tilde Y=\one_{\text{top}_\lambda(U)}$. Then $Y$ and $\tilde Y$ have the
same distribution, implying that ${\tilde Y}^*=Y^*$. On the other hand,
$\tilde Y$ and $U$ are aligned, implying that $\E[U\tilde Y]=\E[U^*\tilde
Y^*]$. Putting everything together, we conclude that
\[
\E[WY]\leq \E[U\tilde Y].
\]
In a similar way, we show that $\E[U\tilde Y]\leq\E[WY] $. We thus have shown
that for all $\lambda\in [0,1]$,
\[
\E[W\one_{\text{top}_\lambda(W)}]=\E[U\one_{\text{top}_\lambda(U)}].
\]
This in turn implies that $W$ and $U$ have the same distribution.
\end{proof}

\subsubsection{Proof of Theorem~\ref{thm:E-->delta}}
Theorem~\ref{thm:E-->delta} follows immediately from Lemma~\ref{lem:E-to-C}
and the following proposition. We remark that when $W$ is bounded, or even in
$L^2$, the proof of the proposition is much easier as one can consider
$\cC(W, W) = \E[W^2]$. This does not work when $W$ is only assumed to be in
$L^1$. The proof begins by transforming $W$ into a bounded graphon.  In what
follows, we use the metric
\[
\delta_p(U,W) = \inf_\phi \|U-W^\phi\|_p,
\]
where the infimum is over all measure-preserving bijections $\phi \colon
[0,1] \to [0,1]$.  It clearly satisfies $\delta_\square(U,W) \le
\delta_1(U,W)$.

\begin{prop}
If $U$ and $W$ are graphons such that $\cC(U, Y) = \cC(W, Y)$ for all
$L^\infty$ graphons $Y$, then $\delta_1(U, W) = 0$.
\end{prop}

\begin{proof}
We know that $U$ and $W$ have the same distribution. Let $\wt W = \arctan W$ and
$\wt U = \arctan U$ (note that $\wt W$ and $\wt U$ are both bounded). Since both
$\arctan x$ and $x - \arctan x$ are increasing in $x$, for every measure preserving
bijection $\sigma \colon [0,1] \to [0,1]$ the rearrangement inequality implies that
\[
\E [(W - \wt W) \wt W] \geq \E [(U -\wt U)^\sigma \wt W]
\]
and
\[
\E [\wt W^2] \geq \E [\wt U^\sigma \wt W].
\]
Thus
\begin{equation} \label{eq:sigma-sandwich}
\E[W \wt W] - \E[U^\sigma \wt W] \geq \E[ \wt W^2] - \E[\wt
U^\sigma \wt W] \geq 0,
\end{equation}
By assumption we have $\cC(U,\wt W)=\cC(W, \wt W) $, which must equal $\E[W
\wt W]$ by the rearrangement inequality. Taking the infimum over all $\sigma$
in \eqref{eq:sigma-sandwich} and using the facts from the previous sentence
yields $\cC(\wt W, \wt W) = \cC(\wt U, \wt W)$. A similar argument shows that
$\cC(\wt U, \wt U) = \cC(\wt U, \wt W)$. Therefore
\begin{align*}
\delta_2(\wt U, \wt W)^2 = \inf_\sigma \E[(\wt U^\sigma - \wt W)^2]
&=
\E[\wt U^2] + \E[\wt W^2] - 2 \sup_\sigma \E[ \wt U^\sigma \wt W]
\\&= \cC(\wt U, \wt U) + \cC(\wt W, \wt W) - 2\cC(\wt U, \wt W) = 0.
\end{align*}
Since $\delta_1(\wt U, \wt W) \leq \delta_2(\wt U, \wt W)$, we must have
$\delta_1(\wt U, \wt W) = 0$ as well.

Finally we need to deduce that $\delta_1(U, W) = 0$. Let $K > 0$. From the
mean value theorem, we know that
\[
\abs{x - y} \leq (1 + K^2)\abs{\arctan x - \arctan y}
\]
whenever $x,y \in [-K,K]$. It follows that for every $\sigma$,
\[
\snorm{(U^\sigma - W)\one_{\abs{U^\sigma} \leq K,\abs{W} \leq K}}_1 \leq (1 + K^2)
\snorm{\wt U^\sigma - \wt W}_1.
\]
The left side differs from $\norm{U^\sigma - W}_1$ by at most $4 \norm{U
\one_{\abs{U} > K}}_1$ (here we use the triangle inequality, and the fact
that $U$ and $W$ have the same distribution, so that $\norm{W
\one_{\abs{U^\sigma}
> K}}_1 \leq \norm{U \one_{\abs{U} > K}}_1$ and $\norm{U^\sigma \one_{\abs{W}
> K}}_1 \leq \norm{U \one_{\abs{U} > K}}_1$ by the rearrangement inequality).
Thus
\[
\snorm{U^\sigma - W}_1 \leq (1 + K^2)
\snorm{\wt U^\sigma - \wt W}_1 + 4 \norm{U\one_{\abs{U} > K}}_1.
\]
Taking the infimum over $\sigma$ and using $\delta_1(\wt U, \wt W) = 0$, we
find that $\delta_1(U, W) \leq 4 \norm{U \one_{\abs{U} > K}}$. Since $K$ can
be made arbitrarily large, $\delta_1(U, W) = 0$.
\end{proof}

\subsection{Proof of (i)$\Rightarrow$(iv)$\Rightarrow$(iii) in Theorem~\ref{thm:W-conv-equiv}}
\label{sec:WProof-DtoF}

The following result gives the implication (i)$\Rightarrow$(iv) in
Theorem~\ref{thm:W-conv-equiv} as well as the statement that
$\FF_{\ba}(W_n,J)$ converges whenever $W_n$ converges in metric.

\begin{prop}
\label{prop:U2F} Let $q\in\N$, $\ba\in \PD_q$, $\eps>0$, and $h\in\R^q$, and
let $J \in \R^{q \times q}$ be a symmetric matrix. For any two graphons $U$
and $W$,
\[
\Bigl|\FF_\ba(U,J)-\FF_\ba(W,J)\Bigr|\leq \|J\|_1\delta_\square(U,W)
\]
and
\[
\Bigl|\FF_{\ba,\eps}(U,J)-\FF_{\ba,\eps}(W,J)\Bigr|\leq \|J\|_1\delta_\square(U,W).
\]
\end{prop}

\begin{proof}
Since the left sides of the above bounds do not change if we replace $U$ by
$U^\phi$ for a measure preserving bijection $\phi\colon [0,1]\to[0,1]$, it is
enough to prove the lemma with a bound in terms of $\|U-W\|_\square$ instead
of $\delta_\square(U,W)$. The result then follows immediately from
\eqref{Erho-Erho} and the definitions \eqref{FaW-def} and \eqref{FaepsW-def}.
\end{proof}

Next we show that convergence of $\FF_\ba(W_n,J)$ for all $J$ implies
convergence of $\EE_\ba(W_n,J)$.  Together with our results from the last
section, this shows that convergence of $\FF_\ba(W_n,J)$ for all $q\in\N$,
$\ba\in\PD_q$, and symmetric $J\in\R^{q \times q}$ is sufficient for metric
convergence, which concludes the proof of Remark~\ref{rem:W-equiv-conv}(ii).

\begin{lemma}
\label{FF2EE} Let $q\in\N$, $\ba\in\PD_q$, and $c>0$, let $J \in \R^{q \times
q}$ be symmetric, and let $U$ and $W$ be two graphons. Then
\[
\bigl|\EE_\ba(W,J)-\EE_\ba(U,J)\bigr|
\le
\frac{1}{c}\bigl|\FF_\ba(W,cJ)-\FF_\ba(U,cJ)\bigr|+ \frac{2 \log q}{c}.
\]
\end{lemma}

\begin{proof}
Using the fact that $\Ent(\rho)\leq \log q$, we get by \eqref{EaW-def} and
\eqref{FaW-def}
\[
\bigl|\EE_\ba(W,J)-\FF_\ba(W,J)\bigr|\le \log q,
\]
for all $J$, and similarly for $U$. Hence
\begin{align*}
\bigl|\EE_\ba(W,J)-\EE_\ba(U,J)\bigr|
&=\frac{1}{c}\bigl|\EE_\ba(W,cJ)-\EE_\ba(U,cJ)\bigr|\\
& \le
\frac{1}{c}\Bigl(\bigl|\FF_\ba(W,cJ)-\FF_\ba(U,cJ)\bigr| +2\log
q\Bigr),
\end{align*}
which proves the claim.
\end{proof}

\begin{proof}[Proof of (iv)$\Rightarrow$(iii) in Theorem~\ref{thm:W-conv-equiv}]
As in the proof above, one sees that
\[\frac{1}{c}\FF_{\ba,\eps}(W_n,cJ)-\frac{\log q}c
\leq
\EE_{\ba,\eps}(W_n,J)
\leq\frac{1}{c}\FF_{\ba,\eps}(W_n,cJ)+\frac{\log q}c.
\]
But this clearly shows that (iv)$\Rightarrow$(iii) in
Theorem~\ref{thm:W-conv-equiv}.
\end{proof}

\subsection{Quantitative bounds on distance between fractional quotients}
\label{sec:Graphon-Quotients}

In this section, we prove a quantitative bound on the distance between two
different quotients of the same graphon, which will be used in the next
section.

We define the $L^1$ distance in $\FP_q$ to be
\begin{equation}
\label{d-rho-rho}
d_1(\rho,\rho')=
\sum_{i\in[q]}\int_0^1|\rho_i(x)-\rho_i'(x)| \,dx;
\end{equation}
note that $d_1(\rho,\rho')\leq 2$. We also need the definition of $K$-bounded
tails from \eqref{K-bd-tails}.

\begin{lemma}\label{lem:drho->dW/phi}
Let $q\in\N$, and let $W$ be a graphon with $K$-bounded tails for some
function $K \colon (0,\infty) \to (0,\infty)$. Then there exists a weakly
increasing function $\eps_K:[0,2]\to [0,\infty)$ such that $\eps_K(x)\to 0$
as $x\to 0$ and
\[
d_1(W/\rho,W/\rho')
\leq
\eps_K(d_1(\rho,\rho'))
\]
for all $\rho,\rho' \in \FP_q$. For $\ba,\bb \in \PD_q$,
\begin{equation}
\label{Haus-Sa-Sb}
d_1^{\Haus}\Bigl(\iv\SS_\ba(W),\iv\SS_\bb(W)\Bigl)\leq \eps_K(\|\ba-\bb\|_1),
\end{equation}
where $\|\ba-\bb\|_1=\sum_i|a_i - b_i|$.
\end{lemma}

\begin{proof}
Fix $\eps>0$.  Clearly
\[
\sum_i|\alpha_i(\rho)-\alpha_i(\rho')|
\leq \sum_i\int_0^1|\rho_i(x)-\rho'_i(x)|\,dx=d_1(\rho,\rho').
\]
On the other hand, using the fact that $\sum_i\rho_i(x)=\sum_j\rho'_i(x)=1$
for all $x\in[0,1]$, we have
\begin{align*}
\sum_{i,j}\Bigl|\beta_{ij}(W/\rho)
-\beta_{ij}(W/\rho')\Bigr|
&=\sum_{i,j}
\Bigl|
\int _{[0,1]^2} W(x,y)\Bigl(\rho_i(x)\rho_j(y)- \rho'_i(x)\rho'_j(y)\Bigr)\,dx\,dy
\Bigr|
\\
&\leq
\int _{[0,1]^2} |W(x,y)|\sum_{i,j}\Bigl|\rho_i(x)\rho_j(y)-
\rho_i(x)\rho'_j(y)\Bigr|\,dx\,dy
\\
&\quad \phantom{}+
\int _{[0,1]^2} |W(x,y)|\sum_{i,j}\Bigl|\rho_i(x)\rho'_j(y)-
\rho'_i(x)\rho'_j(y)\Bigr|\,dx\,dy
\\
&=
 2 \int _{[0,1]^2} |W(x,y)|
\sum_{i}\Bigl|\rho_i(x)-\rho'_i(x)\Bigr|\,dx\,dy
\\
&\leq
2 \int _{[0,1]^2} K(\eps/8)
\sum_{i}\Bigl|\rho_i(x)-\rho'_i(x)\Bigr|\,dx\,dy
\\
&\quad \phantom{}+
4 \int _{[0,1]^2} |W(x,y)|\one_{|W(x,y)|\geq K(\eps/8)}
\,dx\,dy
\\
&\leq 2K(\eps/8)d_1(\rho,\rho') + \eps/2,
\end{align*}
showing that
\[
d_1(W/\rho,W/\rho')
\leq (1+ 2K(\eps/8))d_1(\rho,\rho') + \eps/2,
\]
which is at most $\eps$ provided that $d_1(\rho,\rho')\leq
\eps/(2+4K(\eps/8))$. Since $\eps>0$ was arbitrary, this immediately implies
the existence of the desired function $\eps_K$.

The claim \eqref{Haus-Sa-Sb} follows from noting that for any $\ba, \bb \in
\PD_q$ and $\rho \in \FP_q$ with $\alpha(\rho) = \ba$, we can find a $\rho'
\in \FP_q$ with $\alpha(\rho') = \bb$ such that $d_1(\rho,\rho') =
\|\ba-\bb\|_1$.  (In fact, we can choose $\rho'$ so that $\rho_i \le \rho'_i$
if and only if $a_i \le b_i$.)
\end{proof}

The above lemma can be used to show that $\EE_\ba(W,J)$ is continuous in
$\ba$. The continuity of $\FF_\ba(W,J)$ then follows from noting that
$\Ent(\rho)$ is uniformly continuous in $\rho \in \FP_q$ with respect to
$d_1$ (as $-x \log x$ is continuous on $[0,1]$). More explicitly, we have the
following lemma.

\begin{lemma}\label{lem:S-cont}
Let $q\in\N$. Then the function $\Ent\colon \FP_q\to [0,\log q]$ with
$\rho\mapsto \Ent(\rho)$ is uniformly continuous in the metric $d_1$ defined
in \eqref{d-rho-rho}. Explicitly,
\[
|\Ent(\rho)-\Ent(\rho')|
\leq q \tilde f\Bigl(\frac 1qd_1(\rho,\rho')\Bigr),
\]
where $\tilde f(x)=x(1-\log x)$.
\end{lemma}

\begin{proof}
Because the function $f(x) = -x\log x$ is concave, for each $t>0$ the
function $x \mapsto f(x) - f(x+t)$ is weakly increasing.  It follows from
this and $f(0)=f(1)=0$ that for any $x,y \in [0,1]$,
\[
\begin{aligned}
|f(x)-f(y)| &\le \max\{|f(0) - f(|x-y|)|,|f(1-|x-y|)-f(1)|\}\\
& = \max\{f(|x-y|),f(1-|x-y|)\}\\
& \le f(|x-y|)+f(1-|x-y|)\\
& \le f(|x-y|) + |x-y|\\
& = \tilde f(|x-y|),
\end{aligned}
\]
where the last inequality holds because $f(1-t) \le t$ for all $t \in [0,1]$.
Since $d_1(\rho,\rho') = \sum_i \int_{[0,1]} |\rho_i(x) - \rho'_i(x)| \,dx$
and $\tilde f$ is concave, we have
\[
\begin{aligned}
\frac 1q\Bigl| \Ent(\rho)-\Ent(\rho')\Bigr|
&=
\frac 1q
\Bigl|
\sum_{i=1}^q
\int_0^1\Bigl( f(\rho_i(x))-f(\rho_i'(x)) \Bigr)\, dx\Bigr|\\
&\leq
\frac 1q\sum_{i=1}^q\int_0^1\tilde f(|\rho_i(x)-\rho'_i(x)|) \, dx\\
&\leq
\tilde f\Bigl(\frac 1qd_1(\rho,\rho')\Bigr).
\end{aligned}
\]
Because $\tilde f$ is continuous at $0$, this completes the proof.
\end{proof}

\section{Convergent sequences of uniformly upper regular graphs}
\label{sec:graphs->graphons}

In this section we prove Theorem~\ref{thm:G-conv-equiv}.
Theorem~\ref{thm:limit-expressions} will follow from this theorem and
Theorem~\ref{thm:inverse-limit}, which we prove in
Section~\ref{sec:inferring-upper-regularity}.

\subsection{Preliminaries} \label{subsec:graphgraphonprelim}
We start by stating some of the results from \cite{part1}, which will allow
us to replace uniformly upper regular sequences of weighted graphs by
sequences of weighted graphs with $K$-bounded tails. To state them, we will
use the cut distance between two weighted graphs $G$, $G'$ with identical
node sets $V(G)=V(G')=V$ and identical node weights
$\alpha_x=\alpha_x(G)=\alpha_x(G')$, defined as
\begin{equation}
\label{d-cut-def}
d_\square(G,G')=\max_{S,T\subseteq V}
\Bigl|
\sum_{(x,y)\in S\times T}\frac{\alpha_x\alpha_y}{\alpha_G^2} \Bigl(\beta_{xy}(G)-\beta_{xy}(G')\Bigr)
\Bigr|.
\end{equation}
Note that this distance is equal to $\|W^G-W^{G'}\|_\square$, where $W^G$ and
$W^{G'}$ are the step functions defined in \eqref{W-G-def} and
$\|\cdot\|_\square$ is the cut norm defined in \eqref{cutnorm-def}. Indeed,
for $W=W^G$, the supremum in \eqref{cutnorm-def} can easily be shown to be a
maximum that is attained for sets $S$ and $T$ which are both unions of the
intervals $I_i$.

We will also use the notion of an equipartition of the vertex set $V(G)$ of a weighted
graph $G$, defined by requiring that the weights of the parts of the partition differ from
an equal distribution by at most $\alpha_{\max}(G)$. Explicitly, a partition
$\PP=(V_1,\dots,V_k)$ of $V(G)$ is called an \emph{equipartition} if $|\alpha_{V_i}-\frac
1k\alpha_G|\leq\alpha_{\max}(G)$ for all $i\in [k]$. The following version of the weak
regularity lemma was proved in \cite{part1}.

\begin{theorem}[Theorem~C.12 in \cite{part1}] \label{thm:weak-regular-for-graphs}
Let $K\colon (0,\infty)\to (0,\infty)$ and $0<\eps<1$. Then there exist
constants $N=N(K,\eps)$ and $\eta_0=\eta_0(K,\eps)$ such that the following
holds for all $\eta \le \eta_0$: for every $(K,\eta)$-upper regular graph $G$
and each natural number $k \ge N$, there exists a equipartition
$\PP=(V_1,\dots,V_k)$ of $V(G)$ into $k$ parts, such that
\[
\dcut(G,G_\PP)\leq \eps \|G\|_1.
\]
\end{theorem}

As a consequence, given a $K$-upper regular sequence of weighted graphs $G_n$
with no dominant nodes, we can find a sequence of equipartitions $\PP_n$ of
$V(G_n)$ into $k_n$ classes such that the sequence of weighted graphs%
\footnote{If $\|G_n\|_1=0$, we choose $\iv G_n$ to have edge weights $0$.}
$\iv G_n=\frac{1}{\|G_n\|_1}(G_n)_{\PP_n}$ satisfies
\begin{equation}
\label{reg-Gn1}
\dcut\Bigl(\frac{1}{\|G_n\|_1}G_n,\iv G_n\Bigr)\to 0,
\qquad
k_n\frac{\alpha_{\max}(G_n)}{\alpha_{G_n}}\to 0,
\qquad \|\iv G_n\|_1\leq 1
\end{equation}
and
\begin{equation}
\label{reg-Gn2}
 W^{\iv G_n}\text{ has $K$-bounded tails.}
\end{equation}
We call a sequence $(\iv G_n)_{n \ge 0}$ with these properties a \emph{regularized
version} of $(G_n)_{n \ge 0}$, and $\PP_n$ a \emph{regularizing partition} for $G_n$.

To see that all these conditions can be simultaneously achieved, let $\eta_n\to 0$ be
such that $G_n$ is $(K,\eta_n)$-upper regular.  Assume that $\eps_n$ goes to zero
slowly enough that $\eta_n\leq\eta_0(K,\eps_n)$ and $\eta_n N(K,\eps_n)\to
0$ in Theorem~\ref{thm:weak-regular-for-graphs}.  Choosing $k_n=N(K,\eps_n)$, the theorem then gives a sequence
of equipartitions $\PP_n$ of $V(G_n)$ into $k_n$ classes such that \eqref{reg-Gn1}
holds.  The weight of each class of $\PP_n$ is bounded from below by
$\alpha_{G_n}/k_n -\alpha_{\max}(G_n)$, which is asymptotically greater than
$\eta_n \alpha_{G_n}$ because $\eta_n k_n \to 0$ and $k_n\alpha_{\max}(G_n)/\alpha_{G_n} \to 0$.  Thus, the bound
\eqref{K-eta-upreg} holds for $\PP=\PP_n$, establishing that
\[
\sum_{x,y\in V(G_n)}\frac{\alpha_x(G_n)\alpha_y(G_n)}{\alpha_{G_n}^2}
{|\beta_{xy}(\hat G_n)|}\one_{|\beta_{xy}(\hat G_n)|\geq K(\eps)}
\le \varepsilon
\]
for all $\eps>0$.  In other words, $W^{\iv G_n}$ has $K$-bounded tails.

\subsection{Comparing sequence of graphs to sequences of graphons}

In this section we prove three lemmas, which are the main technical lemmas
used to reduce many statements in Section~\ref{sec:results} to those in
Section \ref{sec:W-sequences}. We define
\[
\iv\SS_{\ba,\eps}(W)=\{(\alpha,\beta) \in\iv\SS_q(W) : \|\alpha-\ba\|_\infty\leq\eps\}.
\]

\begin{lemma}
\label{lem:drhophi} Let $G$ be a weighted graph, let $\PP$ be an
equipartition of $V(G)$ into $k$ classes such that $G=G_\PP$, and let
$q\in\N$. Then there exist two maps $\phi\mapsto \rho_\phi$ and
$\rho\to\bar\rho$ from the set of configurations $\phi\colon V(G)\to [q]$
into the set of fractional partitions $\FP_q$ and from $\FP_q$ to $\FP_q$,
respectively, such that the following hold:
\begin{enumerate}
\item[(i)] $W^G/\rho=W^G/\bar\rho$ for all $\rho\in\FP_q$.

\item[(ii)] $\|G\|_1( G/\phi) = W^G/\rho_{\phi}$ for all $\phi\colon V(G)\to [q]$.

\item[(iii)] For each $\rho\in\FP_q $ there exists a $\phi\colon
    V(G)\to\FP_q$ such that
\begin{equation}
\label{aprox-ai}
d_1(\rho_\phi,\bar\rho)\leq qk\frac{\alpha_{\max(G)}}{\alpha_G}.
\end{equation}
\end{enumerate}
\end{lemma}

\begin{proof}
Let $\PP=(V_1,\dots,V_k)$, and assume that the vertices in $G$ are ordered in
such a way that $V_1=\{1,2,\dots,|V_1|\}$, $V_2=\{|V_1|+1,\dots,
|V_1|+|V_2|\}$, etc. Let $x_\mu=\alpha_{V_\mu}(G)/\alpha_G$ for $\mu \in
[k]$, and let $I_1,\dots,I_k\subseteq[0,1]$ be consecutive intervals of
length $x_1$, $\dots$, $x_k$.

For $\rho\in\FP_q$ define $\bar\rho$ by averaging $\rho$ over the intervals
$I_\mu$, i.e., $\bar\rho_i(x)=\frac 1{x_\mu}\int_{I_\mu}\rho_i(y)\,dy $ if $
x\in I_\mu$. Since $W^G$ is constant on sets of the form $I_\mu\times I_\nu$
for $\mu,\nu\in [k]$, we clearly have $W^G/\rho=W^G/\bar\rho$, proving (i).

Next, given $\phi\colon V(G)\to[q]$, let $\rho_\phi$ be the fractional
$q$-partition defined by
\begin{equation}
\label{rhophi}
(\rho_\phi)_i(x)=\frac1{\alpha_{V_\mu(G)}}\sum_{u\in V_\mu(G)}\alpha_u(G)\one_{\phi(u)=i}
\qquad\text{when } x\in I_\mu.
\end{equation}
Using the fact that $\beta_{uv}(G)$ is constant on sets of the form
$V_\mu\times V_\nu$, it is then easy to check that $ \|G\|_1G/\phi = W^G /
\rho_\phi$, proving (ii).

To prove (iii), consider $\rho\in\FP_q$, and let
$\alpha_{\mu,i}=\int_{I_\mu}\rho_i(x)\,dx$. We then decompose $V_\mu$ into
$q$ sets $V_{\mu,i}$ such that
\[
|\alpha_{\mu,i}-\alpha_{V_{\mu,i}}(G)/\alpha_G|\leq \frac{\alpha_{\max}(G)}{\alpha_G}
\]
for all $i$ and $\mu$. Setting $\phi=i$ on $\bigcup_\mu V_{\mu,i}$, we then
get a map $\phi\colon V(G)\to[q]$ such that \eqref{aprox-ai} holds.
\end{proof}

\begin{lemma}
\label{lem:G-W-quotients} Let $q\in\N$ and let $(G_n)_{n \ge 0}$ be a
uniformly upper regular sequence of weighted graphs. If $\iv G_n$ is a
regularized version of $G_n$, then
\begin{equation} \label{G-W-quotients}
d_1^\Haus\Bigl(\SS_q(G_n),\iv\SS_q(W^{\iv G_n})\Bigr) \to 0.
\end{equation}
\end{lemma}

\begin{proof}[Proof of Lemma~\ref{lem:G-W-quotients}]
We start by showing that
\begin{equation}
\label{d1Gphi-ivGphi}
 d_1(G_n/\phi,\|\iv G_n\|_1 (\iv G_n/\phi))
 \leq q^2\dcut\left(\frac1{\|G_n\|_1}G_n,\iv G_n\right)
\end{equation}
whenever $\iv G_n$ is a regularized version of $G_n$. Indeed, from the
definition of the cut distance \eqref{d-cut-def} it is easy to see that
$d_1(G/\phi, G'/\phi)\leq q^2\dcut(\frac 1{\|G\|_1}G,\frac 1{\|G'\|_1}G')$
whenever $G$ and $G'$ have identical node sets and node weights and for any
$\phi\colon V(G)\to [q]$. More generally, for any $\lambda\geq 0$, we have
$d_1(G/\phi,\lambda ( G'/\phi))\leq q^2\dcut(\frac 1{\|G\|_1}G,\frac
{\lambda}{\|G'\|_1}G')$. Setting $\lambda_n=\|\iv G_n\|_1$ proves
\eqref{d1Gphi-ivGphi}.

Next we observe that by Lemma~\ref{lem:drhophi}(ii),
\begin{equation}
\label{ivS-subset S}
{\|\iv G_n\|_1}\SS_q(\iv G_n)\subseteq \iv\SS_q(W^{\iv G_n}).
\end{equation}
On the other hand, given $W^{\iv G_n}/\rho\in \iv\SS_q(W^{\iv G_n})$, we may
use Lemma~\ref{lem:drhophi} and Lemma~\ref{lem:drho->dW/phi} to find a
quotient $\iv G_n /\phi\in\SS_q(\iv G_n)$ such that
\begin{equation} \label{G-W-quotients-2}
\begin{aligned}
d_1\Bigl(W^{\iv G_n}/\rho,
\|\iv G_n\|_1\iv G_n /\phi\Bigr) &= d_1(W^{\iv G_n}/\bar\rho,W^{\iv G_n}/\rho_\phi)\\
&\leq
\eps_K\left(qk_n\frac{\alpha_{\max(G_n)}}{\alpha_{G_n}}\right),
\end{aligned}
\end{equation}
where $k_n$ is the number of classes in the regularizing partition
corresponding to $\iv G_n$ and $K$ is the function from \eqref{reg-Gn2}.

Taking into account the second bound in \eqref{reg-Gn1}, the bound
\eqref{G-W-quotients-2} together with \eqref{ivS-subset S} implies that
\[
d_1^\Haus\Bigl(\|\iv G_n\|_1\SS_q(\iv G_n),\iv\SS_q\Bigl(W^{\iv G_n}\Bigr)\Bigr)\to 0,
\]
while \eqref{d1Gphi-ivGphi} and the fact that $\dcut(\frac 1{\|G_n\|_1}G_n,
\iv G_n)\to 0$ show that
\[
d_1^\Haus(\SS_q( G_n),\|\iv G_n\|_1\SS_q(\iv G_n))\to 0.
\]
Together, these two bounds imply \eqref{G-W-quotients}.
\end{proof}

\subsection{Proof of (i)$\Rightarrow$(ii) and (ii)$\Rightarrow$(iii) in Theorem~\ref{thm:limit-expressions}.}

Recall that by Theorem~\ref{thm:metric>upper-regular}, convergence in metric implies
uniform upper regularity.  As a consequence, (i)$\Rightarrow$(ii) follows immediately
from Lemma~\ref{lem:G-W-quotients}. Indeed,
$\delta_\square\Bigl(\frac 1{\|G_n\|_1}W^{G_n},W\Bigr)\to 0$ by \eqref{reg-Gn1}, which implies that
$\delta_\square(W^{\iv G_n},W)\to 0$. We then use the bound \eqref{G-W-quotients}
from Lemma~\ref{lem:G-W-quotients} in conjunction with \eqref{U2S3} from
Lemma~\ref{lem:U-W-Haus} to conclude that
\[
d_1^\Haus\Bigl(\SS_q(G_n),\iv\SS_q(W)\Bigr)\to 0,
\]
completing the proof of (i)$\Rightarrow$(ii) in
Theorem~\ref{thm:limit-expressions}.

The implication (ii)$\Rightarrow$(iii) follows with the help of
Theorem~\ref{thm:quotients-to-E}, the closedness of $\iv\SS_q(W)$, and the
representation \eqref{Ea-frac-def}.

\subsection{Proof of the equivalence (i)$\Leftrightarrow$(ii)$\Leftrightarrow$(iii)
in Theorem~\ref{thm:G-conv-equiv}}

Recall that by Theorem~\ref{thm:compact}, for a uniformly upper regular sequence
convergence in metric implies convergence in metric to some graphon $W$, so the
above proof of the implication (i)$\Rightarrow$(ii)  for
Theorem~\ref{thm:limit-expressions} also proves it for
Theorem~\ref{thm:G-conv-equiv}.  The implication
(ii)$\Rightarrow$(iii) follows from Theorem~\ref{thm:quotients-to-E}.

It remains to show (iii)$\Rightarrow$(i); i.e., under the assumption of
uniform upper regularity, convergence of the microcanonical ground state
energies implies convergence in metric.  Assume for the sake of contradiction
that $G_n$ does not converge in metric.  By Theorem~\ref{thm:compact} this
implies that there are two subsequences $G_n'$ and $G_n''$ and two graphons
$U$ and $W$ such that $G_n'\to U$ and $G_n''\to W$ in metric while
$\delta_\square(U,W)>0$.  By the already proved (i)$\Rightarrow$(iii) in
Theorem~\ref{thm:limit-expressions}, the microcanonical ground state energies
of $G_n'$ and $G_n''$ converge to those of $U$ and $W$, and by our assumption
that $G_n$ has convergent ground state energies, this implies that $U$ and
$W$ have identical ground state energies, contradicting
Theorem~\ref{thm:E-->delta}.

\subsection{Convergence in metric implies LD convergence}

Our main result in this section is the following theorem, which by \eqref{eq:I-W-eps->0}
proves the implication (i)$\Rightarrow$(iv) in Theorem~\ref{thm:limit-expressions} and
hence also Theorem~\ref{thm:G-conv-equiv}.

\begin{theorem}
\label{thm:metric-->Iq} Let $q\in\N$  and let $(G_n)_{n \ge 0}$ be a
uniformly upper regular sequence of weighted graphs
such that $G_n$ converges to a graphon $W$ in metric and $G_n$ has vertex weights one. Then the limit \eqref{LD-conv-alt} exists
with
\[
I_q((\alpha,\beta))=
\log q-\lim_{\eps\to 0}\sup_{\substack{\rho\in\FP_q \\ d_1(W/\rho,(\alpha,\beta))\leq \eps}}\Ent(\rho).
\]
\end{theorem}

\begin{proof}
To prove the theorem, we will need to calculate probabilities of the form
$\PP_{q,G}\bigl[d_1((\alpha,\beta),G/\phi)\leq \eps\bigr]$ for graphs $G$ that are near to
$W$ in the normalized cut distance. Assume for the moment that $G$ is well
approximated by a weighted graph whose edge weights are constant over large
blocks. More precisely, assume that there exists an equipartition
$\PP=(V_1,\dots,V_k)$ such that $G$ and $G_\PP$ are close in the cut norm. Under
such a condition, the quotients of $G$ are close to those of $G_\PP$, provided they are
suitably normalized.  More precisely, if we define $\iv G$ to be the normalized
weighted graph $\iv G=\frac{1}{\|G\|_1}(G)_{\PP}$, the quotients of $G$ are close to
those of $ \iv G$ multiplied by $\|G\|_1$ (see \eqref{d1Gphi-ivGphi}). Consider thus the
probabilities $\PP_{q,G}\bigl[d_1((\alpha,\beta),\|G\|_1(\iv G/\phi))\leq \eps\bigr]$.

Since the edge weights of $\iv G$ are constant on sets of the form
$V_\mu\times V_\nu$, we have that $\iv G/\phi=\iv G/\phi'$ if for all $i\in
[q]$ and all $\mu\in[k]$, the number of vertices in $V_\mu$ that are mapped
to $i\in [q]$ is the same in $\phi$ and $\phi'$. Denote this number by
$k_{i,\mu}=k_{i,\mu}(\phi)$. The number of configurations $\phi\colon V(G)\to
[q]$ with given $k_{i,\mu}$ is
\begin{equation}
\label{Nkimu}
N(\{k_{i,\mu}\})=\prod_{\mu=1}^k\frac{n_\mu!}{k_{1,\mu}! \dots k_{q,\mu}!},
\end{equation}
where $n_\mu=|V_\mu|$. Approximating $k!$ as $(k/e)^k$ (we analyze the error
term below), we have
\[
 N(\{k_{i,\mu}\})\approx
 \exp\Bigl( -\sum_{\mu=1}^k \sum_{i=1}^q
-{k_{i,\mu}} \log \Bigl(\frac{k_{i,\mu}}{n_\mu}\Bigr)\Bigr)
=e^{|V(G)|\Ent(\rho_\phi)},
 \]
where $\rho_\phi$ is the fractional partition defined in \eqref{rhophi}. Observing that
the number of choices for $\{k_{\mu,i}\}$ is polynomial in $|V(G)|$, and hence will not
contribute to $I_q$ (again we bound the error later), and noting further that $\|\iv G\|_1 (
\iv G/\phi) = W^{\iv G}/\rho_{\phi}$ by Lemma~\ref{lem:drhophi}, we then approximate
$\PP_{q,G}\bigl[d_1((\alpha,\beta),\|G\|_1(\iv G/\phi))\leq \eps\bigr]$ by
\[
q^{-|V(G)|}\max_{\substack{\phi\colon V(G)\to [q] \\ d_1((\alpha,\beta), W^{\iv G}/\rho_\phi)\leq\eps}}
e^{|V(G)|\Ent(\rho_\phi)}.
\]
Taking into account that (again by Lemma~\ref{lem:drhophi}) any fractional
quotient $W^{\iv G}/\rho$ can be well approximated by a quotient of the form
$W^{\iv G}/\rho_\phi$, we obtain the theorem.

The formal proof proceeds as follows.  Let $K:(0,\infty)\to(0,\infty)$ and
$(\eta_n)_{n \ge 0}$ be such that $\lim_{n\to\infty}\eta_n=0$ and $G_n$ is
$(K,\eta_n)$-upper regular. Fix $\eps>0$. By
Theorem~\ref{thm:weak-regular-for-graphs} and
Definition~\ref{def:unif-up-reg}, there are constants $k\in\N$ and
$n_0<\infty$ such that for each $n\geq n_0$ there exists an equipartition
$\PP=\PP_n=(V_1,\dots,V_k)$ of $V(G_n)$ into $k$ parts such that
\[
q^2\dcut\Bigl(G_n,(G_n)_\PP\Bigr)\leq\frac \eps{2} \|G_n\|_1
\]
and
$W^{\iv G_n}$ with $\iv G_n=\frac 1{\|G_n\|_1}(G_n)_\PP$ has $K$-bounded tails.

Let $G=G_n$ and $\iv G=\iv G_n$.
By the bound
\eqref{d1Gphi-ivGphi},
\[
d_1(G/\phi,\|\iv G\|_1 (\iv G/\phi))
 \leq q^2\dcut\left(\frac1{\|G\|_1}G,\iv G\right)\leq \frac{\eps}2,
\]
implying that
\begin{align*}
\PP_{q,G}\bigl[d_1((\alpha,\beta),\|\iv G\|_1(\iv G/\phi))\leq \eps/2\bigr]
&\leq \PP_{q,G}\bigl[d_1((\alpha,\beta),G/\phi)\leq \eps\bigr]\\
&\leq \PP_{q,G}\bigl[d_1((\alpha,\beta),\|\iv G\|_1(\iv G/\phi))\leq 3\eps/2\bigr].
\end{align*}
Given a configuration $\phi\colon V(G)\to[q]$, let $k_{i,\mu}(\phi)$ be the number of
vertices $v\in V_\mu$ such that $\phi(v)=i$, and let $N(\{k_{i,\mu}\})$ be the number of
maps $\phi$ leading to the same $k_{i,\mu}$; see \eqref{Nkimu} above. Bounding the
number of choices for $\{k_{\mu,i}\}$ by $|V(G)|^{qk}$ and observing that $\|\iv G\|_1(
\iv G/\phi) = W^{\iv G}/\rho_{\phi}$, we then bound
\[
\begin{aligned}
\max_{\phi:d_1((\alpha,\beta),W^{\iv G}/\rho_\phi)\leq \eps/2}N(\{k_{i,\mu}(\phi)\})
&\leq
q^{|V(G)|}
\PP_{q,G}\bigl[d_1((\alpha,\beta),G/\phi)\leq \eps\bigr]
\\
&\leq |V(G)|^{qk}\max_{\phi:d_1((\alpha,\beta),W^{\iv G}/\rho_\phi)\leq 3\eps/2}N(\{k_{i,\mu}(\phi)\}).
\end{aligned}
\]
Since $(k/e)^k\leq k!\leq ek (k/e)^k$,
\[
\Bigl(\frac 1{e|V(G)|}\Bigr)^{kq}
e^{\Ent(\rho_\phi)|V(G)|}
\leq
N(\{k_{i,\mu}\})\leq (e|V(G)|)^k e^{\Ent(\rho_\phi)|V(G)|},
\]
implying
\[
\begin{aligned} 
q^{|V(G)|}\PP_{q,G}\bigl[&d_1((\alpha,\beta),G/\phi)\leq \eps\bigr]\\
&\leq
(e|V(G)|)^{k(q+1)}\max_{\substack{\phi\colon V(G)\to [q] \\ d_1((\alpha,\beta),W^{\iv G}/\rho_\phi)\leq 3\eps/2}}e^{|V(G)|\Ent(\rho_\phi)}
\\
&\leq
(e|V(G)|)^{k(q+1)}\sup_{\substack{\rho\in\FP_q \\ d_1((\alpha,\beta),W^{\iv G}/\rho)\leq 3\eps/2}}e^{|V(G)|\Ent(\rho)}
\end{aligned}
\]
and
\[
\begin{aligned} 
q^{|V(G)|}\PP_{q,G}\bigl[&d_1((\alpha,\beta),G/\phi)\leq \eps\bigr]\\
 &\geq (e|V(G)|)^{-kq}\max_{\substack{\phi\colon V(G)\to
[q] \\ d_1((\alpha,\beta),W^{\iv G}/\rho_\phi)\leq \eps/2}}e^{|V(G)|\Ent(\rho_\phi)}.
\end{aligned}
\]
Next we use Lemma~\ref{lem:drhophi} to approximate an arbitrary fractional
partition $\rho\in\FP_q$ by a fractional partition of the form $\rho_\phi$,
with a error of $qk/|V(G)|$ in the $d_1$ distance. With the help of
Lemmas~\ref{lem:drho->dW/phi} and \ref{lem:S-cont}, we can ensure that for
$n$ (and hence $|V(G)|=|V(G_n)|$) large enough, the resulting errors in
$d_1((\alpha,\beta),W^{\iv G}/\rho)$ and $\Ent(\rho)$ are bounded by $\eps/4$
and $\eps$, respectively, leading to the lower bound
\[
\begin{aligned}
q^{|V(G)|}\PP_{q,G}&\bigl[d_1((\alpha,\beta),G/\phi)\leq \eps\bigr]
\\&\geq
(e|V(G)|)^{-kq}\sup_{\substack{\rho\in\FP_q \\ d_1((\alpha,\beta),W^{\iv G}/\rho)\leq \eps/4}}e^{|V(G)|(\Ent(\rho)-\eps)}.
\end{aligned}
\]
To conclude the proof, we note that if $G_n\to W$ in metric, then
$\delta_\square(W^{\iv G_n},W)\to 0$. Taking into account
Lemma~\ref{lem:drho->dW/phi} and the fact that the entropy $\Ent(\rho)$ is
invariant under measure preserving transformations, we get that for $n$
sufficiently large
\[
\begin{aligned}
(e|V(G_n)|)^{-kq}\sup_{\substack{\rho\in\FP_q \\ d_1((\alpha,\beta),W/\rho)\leq
\eps/8}}&e^{|V(G_n)|(\Ent(\rho)-\eps)}\\
&\leq
q^{|V(G_n)|}\PP_{q,G_n}\bigl[d_1((\alpha,\beta),G_n/\phi)\leq \eps\bigr]
\\
&\leq
(e|V(G)|)^{k(q+1)}\sup_{\substack{\rho\in\FP_q \\ d_1((\alpha,\beta),W/\rho)\leq 2\eps}}e^{|V(G)|\Ent(\rho)}.
\end{aligned}
\]
As a consequence
\[
\begin{aligned}
-\log q -\eps+\sup_{\substack{\rho\in\FP_q \\ d_1((\alpha,\beta),W/\rho)\leq
\eps/8}} \Ent(\rho)
&\leq
\liminf_{n\to\infty}\frac{\log\PP_{q,G_n}\bigl[d_1((\alpha,\beta),G_n/\phi)\leq \eps\bigr]}{|V(G_n)|}
\\
&\leq
\limsup_{n\to\infty}\frac{\log\PP_{q,G_n}\bigl[d_1((\alpha,\beta),G_n/\phi)\leq \eps\bigr]}{|V(G_n)|}
\\
&\leq
-\log q+\sup_{\substack{\rho\in\FP_q \\ d_1((\alpha,\beta),W/\rho)\leq 2\eps}}\Ent(\rho).
\end{aligned}
\]
Sending $\eps\to 0$ completes the proof.
\end{proof}

\subsection{Completion of the proofs of Theorems~\ref{thm:G-conv-equiv} and \ref{thm:limit-expressions}}
\label{sec:completion}

To complete the proof of Theorem~\ref{thm:G-conv-equiv}, we still need to
show for graphs with node weights one, statements (iv) and (v) are equivalent
to the other statements of the theorem. We will also have to establish the
limit expressions given in Theorem~\ref{thm:limit-expressions}.

By Theorem~\ref{thm:metric-->Iq}, we know the implication
(i)$\Rightarrow$(iv) in Theorem~\ref{thm:G-conv-equiv}, and we also know that
the rate function is given by \eqref{IW-def}, as claimed in
Theorem~\ref{thm:limit-expressions}(iv). Finally, by
Theorem~\ref{thm:LD->F-and-Sq}(ii), in Theorem~\ref{thm:G-conv-equiv}
statement (iv) implies statement (v), and by Lemma~\ref{lem:FatoFetc}(iii),
this in turn implies statement (iii) of Theorem~\ref{thm:G-conv-equiv},
completing the proof of Theorem~\ref{thm:G-conv-equiv}.

Theorem~\ref{thm:inverse-limit}, which we prove in the next section, is
nearly enough to deduce Theorem~\ref{thm:limit-expressions} from
Theorem~\ref{thm:G-conv-equiv}.  The only missing piece is the explicit limit
expressions stated in Theorem~\ref{thm:limit-expressions} for how limiting
quotients, ground state energies, free energies, and large deviations rate
function depend on $W$.  So far, we have dealt with all of them except for
the microcanonical free energies. Since we already have shown that
convergence in metric for graphs with node weights one implies LD convergence
with rate function given by \eqref{IW-def}, this follows from the following
lemma.

\begin{lemma} \label{lem:W-LD=>E}
Let $W$ be a graphon, and let $(G_n)_{n \ge 0}$ be a sequence of weighted
graphs. If $(G_n)_{n \ge 0}$ is LD convergent with rate function
$I_q=I_q(\cdot,W)$ as defined in \eqref{IW-def}, then the microcanonical free
energies of $(G_n)_{n \ge 0}$ converge to those of $W$, as defined in
\eqref{FaW-def}.
\end{lemma}

\begin{proof}
By Theorem~\ref{thm:LD->F-and-Sq}(ii), the assumption implies convergence of
the microcanonical free energies, with the limiting free energies given by
\begin{align*}
F_\ba(J)
&= \inf_{(\alpha,\beta)\in\iv\SS_{\ba}} \bigl(-\langle \beta,J\rangle+ I_q((\alpha,\beta),W)\bigr)-\log q.
\\
&= \inf_{(\alpha,\beta)\in\iv\SS_{\ba}} \inf_{\substack{\rho \in \FP_q
    \\ W/\rho = (\alpha,\beta)}} \bigl(-\langle \beta,J\rangle -
\Ent(\rho)\bigr)
\\
&= \inf_{\substack{\rho \in \FP_q
    \\ \alpha(\rho) = \ba}} \bigl(-\langle \beta(W/\rho),J\rangle -
\Ent(\rho)\bigr)\\
&= \FF_\ba(W,J),
\end{align*}
as desired.
\end{proof}

\section{Inferring uniform upper regularity}
\label{sec:inferring-upper-regularity}

In this section we prove Theorem~\ref{thm:inverse-limit}. We have already
proved a number of implications between the four conditions in the theorem
statement. Specifically, from Lemma~\ref{lem:W-LD=>E} we know that (iv)
implies (iii). From Lemma~\ref{lem:FatoFetc}(iii) (and the analogous
assertion that $\EE_\ba(W,J) = \lim_{\lambda \to \infty} \lambda^{-1}
\FF_\ba(\lambda J)$ with an essentially identical proof) we know that (iii)
implies (i). From Theorem~\ref{thm:quotients-to-E} we deduce that (ii)
implies (i). Thus, it remains to show that (i) implies that $(G_n)_{n \ge 0}$
is uniformly upper regular, which is the statement of the following
proposition.

\begin{prop} \label{prop:E=>upp-reg}
Let $(G_n)_{n \ge 0}$ be a sequence of weighted graphs with no dominant nodes
and $W$ a graphon. If the microcanonical ground state energies of $G_n$
converge to those of $W$, then $(G_n)_{n \ge 0}$ is uniformly upper regular.
\end{prop}

In order to prove this proposition, we introduce a notion of
\emph{equipartition upper regularity}, where instead of considering all
partitions of the vertex set with no part having weight smaller than $\eta
\alpha_G$, we consider all equipartitions into $q$ parts. Following the
definition, we prove a lemma which says that the two notions of upper
regularity are qualitatively equivalent.

\begin{definition}
Let $K\colon (0,\infty)\to(0,\infty)$ be any function and let $q \in \N$. A
weighted graph $G$ is \emph{$(K,q)$-equipartition upper regular} if
$\alpha_{\max}(G) \leq \alpha_G/(2q)$ and for every $\eps > 0$ and
equipartition $\cP$ of $V(G)$ into $q$ parts,
\[
  \sum_{i,j \in [q]} |\beta_{ij}(G/\cP)|\one_{|\beta_{ij}(G/\cP)|
  \geq K(\eps) \alpha_i(G/\cP)\alpha_j(G/\cP)}
  \leq \eps.
\]
\end{definition}

Equipartitions are defined in Section~\ref{subsec:graphgraphonprelim}.  We
use $\beta_{ij}(G/\cP)$ and $\beta_{ij}(G/\phi)$ as synonyms (see
\eqref{G-phi-def} for the definition), where the function $\phi \colon V(G)
\to [q]$ defines the partition $\cP$ as the preimages of the points in $[q]$,
but recall from \eqref{G_P-def} that $\beta_{ij}(G_\cP)$ is normalized
differently from $\beta_{ij}(G/\cP)$.

Using \eqref{G-phi-alpha-def}, \eqref{G-phi-def}, \eqref{G_P-def}, and
\eqref{K-eta-upreg}, it is clear that if $G$ is $(K, \eta)$-upper regular,
then it is $(K, q)$-equipartition upper regular for $q \leq 1/(2\eta)$, since
in every equipartition of $V(G)$ into $q$ parts, the weight of each part is
at least
\[
\alpha_G/q - \alpha_{\max}(G) \ge \alpha_G/q-\eta\alpha_G \geq \eta\alpha_G.
\]
Conversely, we also have the following.

\begin{lemma}
\label{lem:equi-upper-reg} Let $K' \colon (0,\infty) \to (0,\infty)$ be any
function, and let
\[
K(\eps) = \max\{4\eps^{-1}K'(\eps/4), 16\eps^{-2}\}.
\]
Let $\eta > 0$ and $q_0 = \eta^{-2}$. If a weighted graph $G$ is $(K',
q')$-equipartition upper regular for some $q' \geq q_0$, then $G$ is $(K,
\eta)$-upper regular.
\end{lemma}

\begin{proof}
By scaling the vertex weights, we may assume without loss of generality that
$\alpha_G = 1$. Let $\cP=(V_1,\dots,V_q)$ be a partition of $V(G)$ into $q$
classes, where $\alpha_{V_i} \geq \eta$ for each $i \in [q]$, and let $\eps >
0$. Define
\[
S := \{(i,j)\in [q]\times [q] : |\beta_{ij}(G/\cP)| \geq K(\eps)
\alpha_i(G/\cP)\alpha_j(G/\cP)\}.
\]
We need to prove that
\begin{equation} \label{eq:K-tail-e}
\sum_{(i,j)\in S}|\beta_{ij}(G/\cP)| \leq \eps.
\end{equation}
Since $\alpha_i(G/\cP) \geq \eta$ for all $i \in [q]$ and $\sum_{i,j\in[q]}
|\beta_{ij}(G/\cP)| \leq 1$, we have
\begin{equation} \label{eq:S-K-bound}
\eta^2 |S| \leq \sum_{(i,j)\in S} \alpha_i(G/\cP)\alpha_j(G/\cP)
\leq \frac{1}{K(\eps)} \sum_{(i,j)\in S} |\beta_{ij}(G/\cP)|
\leq \frac{1}{K(\eps)}.
\end{equation}
Thus if $\eta^2 K(\eps) > 1$, then $S = \emptyset$ and \eqref{eq:K-tail-e}
trivially holds. So assume from now on that $\eta^2 K(\eps) \leq 1$. Since
$K(\eps) \geq 16\eps^{-2}$ by assumption, we have
\begin{equation}
  \label{eq:e-eta}
  \eps \geq 4\eta.
\end{equation}

For each $x \in V(G)$, define the (weighted) degree of $x$ to
be
\[
 \deg_x(G) = \sum_{y \in V(G)} \alpha_y(G) \abs{\beta_{xy}(G)}.
\]
We construct an equipartition $\cP'$ of $V(G)$ as follows. For each $i \in
[q]$, we partition $V_i$ into subsets $V_{i,0},V_{i,1}, \dots, V_{i,k_i}$
such that
\[
\alpha_{V_{i,k}}(G) \in
\Bigl[\frac 1{q'}- \alpha_{\max}(G), \frac 1{q'} + \alpha_{\max}(G)\Bigr]
\qquad \text{for $1 \leq k
\leq k_i$},
\]
$\alpha_{V_{i,0}}(G) < 1/q'$, and the vertices in $V_{i,0}$ all have the
lowest degree present in $V_i$. We will do this in such a way that for all
$i=1\dots,q$ and all $j=1,\dots,k_{i}$,
\[
\frac 1{q'} \left(\sum_{i'=1}^{i-1} k_{i'} + j\right)\leq
\sum_{i'=1}^{i-1}\sum_{j'=1}^{k_{i'}}\alpha_{V_{i',j'}(G)}
+\sum_{j'=1}^j\alpha_{V_{i,j'}(G)}
\leq \frac 1{q'} \left(\sum_{i'=1}^{i-1} k_{i'} + j\right) +\alpha_{\max}(G).
\]
For example, we can do this greedily by sorting all the vertices in $V_i$
according to their degrees and then placing them into $V_{i,j}$ for
$j=1,2,\dots$ in decreasing order by degree until the lower bound in the
above inequality is satisfied.  Since the last vertex added contributed at
most $\alpha_{\max}(G)$ to the sum, we are guaranteed to have the upper bound
as well. When the total weight left in $V_i$ is too small to fill another
$V_{i,j}$, we are necessarily left with a remainder $V_{i,0}$ that has weight
less than $1/q'$ and contains only vertices with the lowest degree from
$V_i$.

Next, consider the remainder sets $V_{1,0}, \dots, V_{q,0}$, whose union we
denote by $V_0$. By construction, either $V_0$ is empty, in which case we do
nothing, or $\alpha_{V_0}$ lies between $k_0/q'-\alpha_{\max}(G)$ and
$k_0/q'$, where $k_0=q'-\sum_{i\geq 1} k_i$.  Proceeding again greedily (this
time ignoring the degrees), we decompose $V_0$ into $k_0$ sets
$V_{0,1},\dots,V_{0,k_0}$ with weights between $1/q' - \alpha_{\max}(G)$ and
$1/q' + \alpha_{\max}(G)$.  The sets $V_{i,j}$ with $0 \le i \le q$ and $j
\ge 1$ then form an equipartition $\cP'$ of $V(G)$ into $q'$ sets.

Define $S'$ to be the set of pairs $(u,v)\in [q']\times [q']$ for which we
can find an $(i,j)\in S$ such that $u$ is the label of a subclass of $V_i$
and $v$ is the label of a subclass of $V_j$.  In other words, $S'$ refines
the set $S$ from $[q]\times [q]$ to $[q'] \times [q']$, except that it does
not necessarily contain pairs $(u,v)$ for which $u$ or $v$ is in $V_0$ (since
the remainder sets used to form $V_0$ do not necessarily come from a single
part of $\cP$). Thus, we have
\begin{align}
\sum_{(i,j)\in S}|\beta_{ij}(G/\cP)|
\notag
&\leq
\sum_{(u,v)\in S'}|\beta_{uv}(G/\cP')|+
\frac{1}{\norm{G}_1} \sum_{\substack{x,y \in V(G) \\ x \in V_0
\text{ or } y \in V_0}} \alpha_x(G)\alpha_y(G)| \beta_{xy}(G)|
\notag
\\
&\leq
\sum_{(u,v)\in S'}|\beta_{uv}(G/\cP')|+
\frac{2}{\norm{G}_1}\sum_{x \in V_0} \alpha_x(G)\deg_x(G).
\label{tail-bd1}
\end{align}
It remains to prove that \eqref{tail-bd1} is at most $\eps$.

We begin with the second term. For each $i$, since the vertices in $V_{i,0}$
are among the lowest degree vertices of $V_i$, $\alpha_{V_{i,0}}(G) < 1/q'$,
and $\alpha_{V_i}(G) \geq \eta$, we have
\[
\sum_{x \in V_{i,0}} \alpha_x(G) \deg_x(G) \leq
\frac{\alpha_{V_{i,0}}(G)}{\alpha_{V_i}(G)} \sum_{x \in V_i} \alpha_x(G)
\deg_x(G)
\leq \frac{1}{q'\eta} \sum_{x \in V_i} \alpha_x(G) \deg_x(G).
\]
Summing over $i \in [q]$ and using $\sum_{x \in V(G)} \alpha_x(G) \deg_x(G) =
\norm{G}_1$, we see that the second term in \eqref{tail-bd1} is at most
\[
\frac{2}{q'\eta} \leq \frac{2}{q_0 \eta}= 2 \eta \leq \frac{\eps}{2}
\]
by \eqref{eq:e-eta}.

To bound the first term in \eqref{tail-bd1}, we decompose the sum
into a sum of those terms for which $|\beta_{u,v}(G/\PP')|$ is larger than
$K'(\eps/4) \alpha_u(G/\cP')\alpha_v(G/\cP')$ and a sum of those for which it is at most this large.  Since $G$ is
$(K',q')$-equipartition upper regular, we can bound the first sum by
$\eps/4$, while the second is clearly bounded
by the sum of $K'(\eps/4) \alpha_u(G/\cP')\alpha_v(G/\cP')$ over $(u,v) \in S'$.
Taking into account that $K'(\eps/4)\leq \frac{\eps}4 K(\eps)$, this proves that
\[\sum_{(u,v)\in S'}|\beta_{uv}(G/\cP')|
\leq \frac {\eps}4 +\frac{\eps K(\eps)}4
\sum_{(u,v) \in S'} \alpha_u(G/\cP')\alpha_v(G/\cP')
.
\]
Since
\[
\sum_{(u,v) \in S'} \alpha_u(G/\cP')\alpha_v(G/\cP')
\leq
\sum_{(i,j) \in S} \alpha_i(G/\cP)\alpha_j(G/\cP)
 \leq \frac 1{K(\eps)}
\]
by
\eqref{eq:S-K-bound}, we have shown that the first term in \eqref{tail-bd1} is bounded by
$\eps/2$, which completes our proof.
\end{proof}

\begin{proof}[Proof of Proposition~\ref{prop:E=>upp-reg}]
To show that $(G_n)_{n \ge 0}$ is uniformly upper regular by
Lemma~\ref{lem:equi-upper-reg}, it suffices to show that there is some $K
\colon (0,\infty) \to (0,\infty)$ and some sequence of integers $q_n \to
\infty$ such that $G_n$ is $(K, q_n)$-equipartition upper regular.

Equivalently, this amounts to showing that we can find some $q_n \to \infty$
so that $\alpha_{\max}(G_n)/\alpha_{G_n} \leq 1/(2q_n)$ and for every $\eps >
0$, there is some real $K >0$ so that for sufficiently large\footnote{This is
equivalent to the same claim for all $n$ since we can increase $K$ to account
for the first finitely many values of $n$.} $n$, we have
\begin{equation} \label{eq:equipartition-tail}
\max_{\substack{\phi \colon V(G_n) \to [q_n] \\ \text{equipartition}}}
\sum_{i,j \in [q_n]} |\beta_{ij}(G_n/\phi)|\one_{|\beta_{ij}(G_n/\phi)|
\geq K \alpha_i(G_n/\phi)\alpha_j(G_n/\phi)}
\leq \eps.
\end{equation}

For any weighted graph $G$ with $\alpha_{\max}(G)/\alpha_G \leq 1/(2q)$ and
equipartition $\phi \colon V(G) \to [q]$, we have
\[
\sum_{i,j \in [q]} |\beta_{ij}(G/\phi)|\one_{|\beta_{ij}(G/\phi)|
\geq K \alpha_i(G/\phi)\alpha_j(G/\phi)}
= - E_\phi(G,J)
\]
by the definition \eqref{Ephi(G,J)-def} of $E_\phi(G,J)$, where $J \in
\{-1,0,1\}^{q \times q}$ is given by
\[
  J_{ij} = \sign(\beta_{ij}(G/\phi))\one_{|\beta_{ij}(G/\phi)|
    \geq K \alpha_i(G/\phi)\alpha_j(G/\phi)}.
\]
Using $\alpha_i(G/\phi) \geq 1/q - \alpha_{\max}(G)/\alpha_G \geq 1/(2q)$, we
obtain
\begin{align*}
\sum_{i,j\in[q]} |J_{ij}|
&=\sum_{i,j \in [q]} \one_{|\beta_{ij}(G/\phi)| \geq K \alpha_i(G/\phi)\alpha_j(G/\phi)}
\\
&\leq \sum_{i,j\in[q]} \frac{|\beta_{ij}(G/\phi)|}{K \alpha_i(G/\phi)\alpha_j(G/\phi)}
\\
&\leq \frac{4q^2}{K} \sum_{i,j\in[q]} |\beta_{ij}(G/\phi)| \leq \frac{4q^2}{K}.
\end{align*}
It follows from the definition \eqref{Ea(G)-def} of $E_{\ba,\eps}(G,J)$ that
\begin{align}
\text{left side of }\eqref{eq:equipartition-tail}
&\leq \max_{\substack{\phi \colon V(G_n) \to [q_n] \\ \text{equipartition}}}
\max_{\substack{J \in \{-1,0,1\}^{q_n \times q_n} \\ \text{symmetric} \\
\sum_{i,j} |J_{ij}| \leq 4q_n^2/K}} ( - E_\phi(G_n,J)) \notag
\\
& =  \max_{\substack{J \in \{-1,0,1\}^{q_n \times q_n} \\ \text{symmetric} \\
\sum_{i,j} |J_{ij}| \leq 4q_n^2/K}} ( - E_{\bq_n,\alpha_{\max}(G_n)/\alpha_{G_n}}(G_n,J)) \label{eq:equi-tail2}
\end{align}
(here $\bq_n = (1/q_n, \dots, 1/q_n) \in \PD^{q_n}$, and we also write $\bq = (1/q, \dots, 1/q) \in \PD^q$ below).

Since the microcanonical ground state energies of $G_n$ converge to those of $W$,
we know that for every $q \in \N$, we can find $\eps_0(q) > 0$ and $n_0(q)$ so that
\[
  - E_{\bq,\eps}(G_n,J) \leq - \EE_\bq(W,J) + 1/q
\]
for all $0 < \eps < \eps_0(q)$, all $n > n_0(q)$, and every symmetric matrix
$J \in \{-1, 0, 1\}^{q \times q}$ (as there are only finitely many such $J$
for each $q$). Set $\eps_n = \alpha_{\max}(G_n) / \alpha_{G_n}$, so that
$\eps_n \to 0$ because there are no dominant nodes. It follows that we can
find a slowly growing sequence $q_n \to \infty$ so that (for sufficiently
large $n$) we have $\eps_n < \eps_0(q_n)$, $n > n_0(q_n)$, and $q_n\eps_n
\leq 1/2$, from which it follows that
\[
  - E_{\bq_n,\eps_n}(G_n, J) \leq - \EE_{\bq_n}(W, J) + 1/{q_n}
\]
for all symmetric matrices $J \in \{-1,0,1\}^{q_n \times q_n}$.  Hence for
sufficiently large $n$
\begin{align*}
  \eqref{eq:equi-tail2}
  &\leq
    \max_{\substack{J \in \{-1,0,1\}^{q_n \times q_n} \\ \text{symmetric} \\
        \sum_{i,j} |J_{ij}| \leq 4q_n^2/K}} \big(-\EE_{\bq_n}(W,J)
    + 1/q_n)
  \\
  & \leq \sup_{\substack{S \subseteq [0,1]^2 \\ \lambda(S) \leq 4/K}}
  \int_S \abs{W(x,y)} \, dx\,dy + 1/q_n.
\end{align*}
We can choose $K$ large enough that the first term in the final bound above
is at most $\eps/2$. Since $q_n \to \infty$, the second term is also at most
$\eps/2$ for sufficiently large $n$. This proves
\eqref{eq:equipartition-tail}, showing that $(G_n)_{n \ge 0}$ is uniformly
upper regular.
\end{proof}

\appendix

\section{Proof of the rearrangement inequality}
\label{app:rearr}

In this appendix, we prove that
\[
\E[W\,U]\leq \E[W^*\,U^*].
\]
when $W$ is an $L^p$ graphon and $U$ is an $L^{p'}$ graphon with $\frac
1p+\frac1{p'}=1$, with equality holding whenever $U$ and $W$ are aligned.

If $W,U\geq 0$, the proof of the rearrangement inequality is standard, and
can, e.g., be deduced from the following level-set representations
\[
U(x,y)=\int_0^\infty dt \,\one[U(x,y)>t]
\quad\text{and}\quad
W(x,y)=\int_0^\infty ds \,\one[U(x,y)>s].
\]
Indeed, with the help of this representation, we get
\[\begin{aligned}
\E[W\,U]
&=\E\Bigl[\int_0^\infty ds \,\one[W>s]\int_0^\infty dt\,\one[U>t]\Bigr]
\\
&=\int_0^\infty ds \,\int_0^\infty dt \,\Pr\Bigl[W>s\text{ and }U>t]\Bigr]
\\
&\leq\int_0^\infty ds \,\int_0^\infty dt \,\min\Bigl\{\Pr[W>s],\Pr[U>t]\Bigr\}
\\
&=\int_0^\infty ds \,\int_0^\infty dt \,\min\Bigl\{\Pr[W^*>s],\Pr[U^*>t]\Bigr\},
\end{aligned}
\]
where in the last step we used that $U$ and $U^*$ as well as $W$ and $W^*$
have the same distribution. Since the $U^*$ and $W^*$ have nested level sets,
the expression in the last line is equal to
\[
\begin{aligned}
\int_0^\infty &ds \,\int_0^\infty dt \,\Pr\Bigl[W^*>s\text{ and }U^*>t\Bigr]
\\
&=\E\Bigl[\int_0^\infty ds\,\one[W^*>s] \int_0^\infty dt\,\one[U^*>t]\Bigr]
\\
&= \E[W^*\,U^*].
\end{aligned}
\]
If $W$ and $U$ are aligned themselves, the only inequality in the above proof
becomes an equality, showing that $ \E[W\,U]= \E[W^*\,U^*]$ if $W$ and $U$
are aligned.

If $W$ and $U$ are bounded below, say by $W\geq -M$ and $U\geq -M$ for some
$M<\infty$, then we just use that $ \E[W^*\,U^*] - \E[W\,U] =
\E[(W+M)^*\,(U+M)^*] - \E[(W+M)\,(U+M)] $, which follows from linearity of
expectations and the fact that $\E[W] = \E[W^*] $ and $ \E[U]= \E[U^*] $.
Finally, to control the tails as $M\to\infty$, we bound
\[
\begin{aligned}
&\left|
\E[WU]-\E\Bigl[W\one[W\geq-M]\,U\one[U\geq -M]\Bigr]\right|
\\
&\qquad\leq
\E\Bigl[|WU|\one[|U|\geq M]\Bigr]
+
\E\Bigl[|WU|\one[|W|\geq M]\Bigr]
\\
&\qquad\leq
\|W\|_p\Bigl\|U\one[|U|\geq M]\Bigr\|_{p'}
+
\|U|\|_{p'}\Bigl\|W\one[|W|\geq M]\Bigr\|_p.
\end{aligned}
\]
Now the right side goes to zero as $M\to\infty$ by our assumption that $W\in
L^p$ and $U\in L^{p'}$.


\begin{thebibliography}{99}

\bibitem{BS}
  I.~Benjamini and O.~Schramm,
  \emph{Recurrence of distributional limits of finite planar graphs},
  Electron.\ J.\ Probab.\ \textbf{6} (2001), no.\ 23, 13 pp.
  \arXiv{math/0011019}  \doi{10.1214/EJP.v6-96} \MR{1873300}

\bibitem{BR}
 B.\ Bollob\'as and O.\ Riordan,
 \emph{Metrics for sparse graphs}, in
 S.\ Huczynska, J.\ D.\ Mitchell, and C.\ M.\ Roney-Dougal, eds.,
 \emph{Surveys in combinatorics 2009}, pages 211--287,
 London Math.\ Soc.\ Lecture Note Ser.\ \textbf{365}, Cambridge University Press,
 Cambridge, 2009. \arXiv{0708.1919} \MR{2588543}

\bibitem{part1}
 C.~Borgs, J.~T.~Chayes, H.~Cohn, and Y.~Zhao,
 \emph{An $L^p$ theory of sparse graph convergence I: limits,
 sparse random graph models, and power law distributions},
 preprint, 2014. \arXiv{1401.2906}

\bibitem{BCG}
 C.~Borgs, J.~Chayes, and D.~Gamarnik,
 \emph{Convergent sequences of sparse graphs: a large deviations approach},
 preprint, 2013. \arXiv{1302.4615}

\bibitem{bounded} C.~Borgs, J.~Chayes, J.~Kahn, and L.~Lovasz,
 \emph{Left and right convergence of graphs with bounded degree},
  Random Struct.\ Alg.\ \textbf{42} (2013), 1--28. \arXiv{1002.0115}
  \doi{10.1002/rsa.20414} \MR{2999210}

\bibitem{dense0}
 C.~Borgs, J.~Chayes, L.~Lov\'asz, V.~T.~S\'os, and K.~Vesztergombi,
 \emph{Counting graph homomorphisms}, in M.~Klazar, J.~Kratochv\'il, M.~Loebl, J.~Matou\v sek, R.~Thomas, and
 P.~Valtr, eds., \emph{Topics in discrete mathematics}, pages 315--371,
 Algorithms Combin.\ \textbf{26}, Springer, Berlin, 2006.
 \doi{10.1007/3-540-33700-8_18} \MR{2249277}

\bibitem{dense1}
 C.\ Borgs, J.\ T.\ Chayes, L.\ Lov\'asz, V.\ T.\ S\'os, and
 K.\ Vesztergombi,
 \emph{Convergent sequences of dense
 graphs I: Subgraph frequencies, metric properties and testing},
 Adv.\ Math.\ \textbf{219} (2008), 1801--1851. \arXiv{math/0702004} \doi{10.1016/j.aim.2008.07.008} \MR{2455626}

\bibitem{dense2}
 C.\ Borgs, J.\ T.\ Chayes, L.\ Lov\'asz, V.\ T.\ S\'os, and
 K.\ Vesztergombi.
 \emph{Convergent sequences of dense graphs II: Multiway cuts and statistical
 physics},
 Ann.\ of Math.\ (2) \textbf{176} (2012), 151--219. \doi{10.4007/annals.2012.176.1.2 } \MR{2925382}

\bibitem{DZ}
 A.~Dembo and O.~Zeitouni,
 \emph{Large deviations techniques and applications},
 corrected reprint of the second (1998) edition, Stochastic
 Modelling and Applied Probability \textbf{38}, Springer-Verlag, Berlin,
 2010.
 \doi{10.1007/978-3-642-03311-7} \MR{2571413}

\bibitem{FK}
 A.\ Frieze and R.\ Kannan,
 \emph{Quick approximation to matrices and applications}, Combinatorica
 \textbf{19} (1999), 175--220. \doi{10.1007/s004930050052} \MR{1723039}

\bibitem{Knapp}
  A.~W.~Knapp,
  \emph{Basic real analysis}, Cornerstones, Birkh\"auser Boston, Inc., Boston, MA, 2005.
  \doi{10.1007/0-8176-4441-5} \MR{2155259}

\bibitem{LSz-Anal}
 L.\ Lov\'asz and B.\ Szegedy,
 \emph{Szemer\'edi's lemma for the analyst},
 Geom.\ Funct.\ Anal.\ \textbf{17} (2007), 252--270.
 \doi{10.1007/s00039-007-0599-6} \MR{2306658}

\bibitem{GBP}
 G.~B.~Price,
 \emph{On the completeness of a certain metric space with an
 application to Blaschke's selection theorem},
 Bull.\ Amer.\ Math.\ Soc.\ \textbf{46} (1940), 278--280.
 \doi{10.1090/S0002-9904-1940-07195-2} \MR{0002010}
\end{thebibliography}
\end{document}